\documentclass[10pt]{article}

\usepackage{amsmath}
\usepackage{amsthm}
\usepackage{amssymb}
\usepackage{amsfonts}
\usepackage{graphicx}
\usepackage{verbatim}
\usepackage{mathrsfs}
\usepackage{subfig}
\usepackage{latexsym}
\usepackage{color}
\usepackage[hidelinks]{hyperref}
\usepackage{paralist}
\usepackage[]{algorithm2e}
\usepackage{bm}
\usepackage{overpic} 

\newcommand{\rmd}{\mathrm{d}}
\newcommand{\rmD}{\mathrm{D}}
\newcommand{\fa}{\quad \text{for all } \,}

\newcommand{\N}{\mathbb{N}}
\newcommand{\ZZ}{\mathbb{Z}}
\newcommand{\R}{\mathbb{R}}
\newcommand{\C}{\mathbb{C}}
\renewcommand{\d}{\text{d}}

\newcommand{\B}{\mathcal{B}}

\newcommand{\F}{\mathcal{F}}

\renewcommand{\L}{\mathcal{L}}

\renewcommand{\S}{\mathcal{S}}
\newcommand{\X}{\mathcal{X}}

\newcommand{\Y}{\mathcal{Y}}
\newcommand{\Z}{\mathcal{Z}}
\newcommand{\tD}{\tilde\Delta}
\newcommand{\bD}{\bar\Delta}
\newcommand{\eD}{\epsilon_\Delta}
\newcommand{\tN}{\tilde\nabla}
\newcommand{\rmin}{r_{\text{min}}}
\newcommand{\rmax}{r_{\text{max}}}
\newcommand{\rmean}{r_{\text{mean}}}
\newcommand{\radial}{\text{radial}}

\newcommand{\inv}{\text{inv}}

\newcommand{\XX}{\mathfrak{X}}
\renewcommand{\SS}{\mathfrak{S}}
\newcommand{\M}{\mathfrak{M}}
\newcommand{\s}{\mathfrak{s}}
\newcommand{\fF}{\mathfrak{F}}

\renewcommand{\l}{\lambda}
\newcommand{\ol}{\overline{\l}}
\newcommand{\ul}{\underline{\l}}

\newcommand{\tl}{\tilde{\lambda}}
\newcommand{\bu}{\bar{u}}
\newcommand{\bl}{\bar{\lambda}}
\newcommand{\bphi}{\bar\varphi}
\newcommand{\Bphi}{\bm\varphi}
\newcommand{\bBphi}{\bar{\Bphi}}
\newcommand{\Bpsi}{\bm\psi}
\newcommand{\BR}{\bm{R}}
\newcommand{\Bone}{\bm{1}}

\makeatletter
\renewcommand*\env@matrix[1][\arraystretch]{%
  \edef\arraystretch{#1}%
  \hskip -\arraycolsep
  \let\@ifnextchar\new@ifnextchar
  \array{*\c@MaxMatrixCols c}}
\makeatother

\DeclareMathOperator{\erfc}{erfc}

\DeclareMathOperator{\id}{id}
\DeclareMathOperator{\Id}{Id}

\theoremstyle{plain}
\newtheorem{theorem}{Theorem}[section]
\newtheorem{proposition}[theorem]{Proposition}
\newtheorem{lemma}[theorem]{Lemma}
\newtheorem{corollary}[theorem]{Corollary}
\newtheorem{conjecture}[theorem]{Conjecture}
\newtheorem{remark}[theorem]{Remark}
\newtheorem{definition}[theorem]{Definition}

\numberwithin{equation}{section}

\begin{document}

\author{ Maxime Breden\thanks{CMAP, \'Ecole Polytechnique, route de Saclay, 91120 Palaiseau, France.}~~and Maximilian Engel\thanks{
Department of Mathematics, Freie Universit{\"a}t Berlin, Arnimallee 6, 14195 Berlin, Germany.}}

\title{Computer-assisted proof of shear-induced chaos in stochastically perturbed Hopf systems}

\maketitle

\begin{abstract}
We confirm a long-standing conjecture concerning shear-induced chaos in stochastically perturbed systems exhibiting a Hopf bifurcation. The method of showing the main chaotic property, a positive 
\emph{Lyapunov exponent}, is a \emph{computer-assisted proof}. Using the recently developed theory of conditioned Lyapunov exponents on bounded domains and the modified Furstenberg-Khasminskii formula, the problem boils down to the rigorous computation of eigenfunctions of the Kolmogorov operators describing distributions of the underlying stochastic process.
\end{abstract}



\maketitle


\section{Introduction} 
The impact of stochastic noise on the behaviour of a dynamical system is an intensely studied topic of mathematical and physical research.
Very often the mathematical analysis focusses on the statistics of trajectories for different noise realisations and does not consider the dynamical aspects of the system. In contrast, the theory of random dynamical systems as coined by the works of Ludwig Arnold and his co-workers in the 1980s and 1990s, and manifested in Arnold's book \textit{Random Dynamical Systems} \cite{a98}, compares trajectories with different initial conditions but driven by the same noise. A random dynamical system in this sense consists of a model of the time-dependent noise seen as a dynamical system $\theta$ on the probability space, and a model of the dynamics on the state space formalized as a \textit{cocycle} $\varphi$ over $\theta$.

In this framework, we can study the asymptotic behaviour of typical trajectories.  
In many situations there is a spectrum of exponential asymptotic growth rates, the \textit{Lyapunov exponents}. 
The sign of the largest (or top) Lyapunov exponent, $\Lambda_1$, determines if two nearby trajectories converge or separate from each other. If $\Lambda_1$ is negative, we typically observe the convergence of trajectories, a phenomenon labelled \textit{synchronization}.
Positivity of $\Lambda_1$ implies sensitivity of initial conditions and is thereby associated with \emph{chaotic} behaviour. The sensitivity of initial conditions means that any two trajectories starting arbitrarily close to each other will separate at a certain point of time. In other words, even the smallest error in the initial conditions leads to a considerably large error in the future. 

Transitions between synchronization and chaos have become an essential part of bifurcation theory for random dynamical systems. From an applied point of view, specific laser dynamics constitute an important example for such bifurcations. Wieczorek \cite{wieczorek2009stochastic} has conducted numerical bifurcation studies for stochastically forced laser models exploring transitions from synchronisation to chaos. He has also shown similar phenomena for coupled lasers with his co-workers in \cite{bdew14, bew11}. In the context of oceanography, stochastic Hopf bifurcation has been discussed for example in \cite{dfh08}.
Beyond such applications, there is a genuine mathematical interest and motivation for studying the stability of random systems. In dissipative as well as conservative systems proving chaotic behaviour has turned out to be a very challenging and rarely resolved problem. The classical example for this problem is the \textit{standard map}, an area-preserving mapping of the two-torus which is characterised by expansion on large regions of the state space and small islands of contraction. Positivity of the first Lyapunov exponent has not been shown analytically, even if the volumes of these critical regions of contraction tend to zero, and the difficulties have been quantified in \cite{d94}. Blumenthal, Xue and Young  have shown recently in \cite{bxy17} that adding a tiny bit of noise to the system allows for averaging arguments over a stationary measure of the induced Markov chain. As long as the stationary measure allocates just a small amount of mass to the regions of contraction, the existence of a positive Lyapunov exponent can be shown. This turns out to hold true for a large class of maps, also in dissipative systems.

We will prove a similar result for a random dynamical system induced by a stochastic differential equation in this work. The result lines up with the research program suggested in \cite{y08}, and exemplified also by other recent works~\cite{BedrossBlumPanshon2020, engellambrasmussen19_1}: Young expressed the hope that if the geometry of a random map or stochastic flow suggests a positive Lyapunov exponent, then this is actually the case. The underlying philosophy is that a system with expansions on a large enough portion of its phase space can overcome tendencies to form sinks as long as the randomness is strong enough.
This work adds evidence for this conjecture.

\subsection{The random dynamical system}
Our main example is the two-dimensional stochastic differential equation 
\begin{equation}\label{MainEq_killed}
dZ_t= f(Z_t)\rmd t
+
\sigma \, \rmd W_t \,, \quad Z_0 \in E \subseteq \R^2\,,
\end{equation}
where $Z_t=(x_t,y_t)^{T} \in \bar E \subseteq \R^2$, $W_t=\left(
W^1_t, 
W^2_t
\right)^{T}
$ is a two-dimensional standard Brownian motion and the function $f:\mathbb R^2\rightarrow \mathbb R^2$ is defined by
\[
f(Z):=
\left(
\begin{array}{ll}
\alpha & -\beta\\
\beta & \alpha
\end{array}
\right) Z- \|Z\|^2 \left(
\begin{array}{ll}
a & b\\
-b & a
\end{array}
\right)Z.
\]
The constant $\sigma \geq 0$ represents the strength of the purely additive noise and $\alpha\in\mathbb{R}$ is a parameter equal to the real part of eigenvalues of the linearisation of the vector field at $(0,0)$. The parameter $b\in\mathbb{R}$ determines \emph{shear strength} by amplitude-phase coupling, as can be seen when writing the deterministic part of~\eqref{MainEq_killed} in polar coordinates, and $a > 0$, $\beta\in \mathbb{R}$ are additional parameters.

In the absence of noise ($\sigma=0$), the differential equation \eqref{MainEq_killed} is a normal form for the supercritical Hopf bifurcation:
when $\alpha \leq 0$ the system has a globally attracting equilibrium at $(x,y) = (0,0)$ which is exponentially stable until $\alpha=0$ and, when $\alpha >0$, the system has a limit cycle at $\left\{(x,y) \in \mathbb{R}^2 \,:\, x^2 + y^2 = \alpha/a\right\}$ which is globally attracting on $\mathbb{R}^{2}\setminus{\{0\}}$.

In the case with noise ($\sigma >0$), it has been shown in \cite{DoanEngelLambRasmussen} that the solutions of \eqref{MainEq_killed} generate a random dynamical system $(\theta, \varphi)$ (see Appendix~\ref{appx-rds}), where $\theta$ is the shift over \emph{Wiener space} (see Appendix~\ref{app:RDSbySDE}) and  $\varphi$ is a cocycle over $\theta$, i.e.
\[
  \varphi(0, \omega, \cdot) \equiv \Id \quad \text{and} \quad \varphi(t+s, \omega,Z) = \varphi(t, \theta_s \omega,\varphi(s, \omega,Z)) \fa \omega \in \Omega, Z \in \R^2 \text{ and } t, s \ge 0\,.
\]
For $E =\mathbb R^2$, the stochastic system has a unique stationary density  
\begin{equation*}
p(x,y)= K_{a,\alpha,\sigma}\exp\left(\frac{2\alpha(x^2+y^2)-a(x^2+y^2)^2}{2\sigma^2}\right)\,,
\end{equation*}
where $K_{a,\alpha,\sigma} >0$ is the normalisation constant and is given by
\begin{equation*}
K_{a,\alpha,\sigma} = \frac{ 2 \sqrt{2 a} }{\sqrt{\pi}\sigma \erfc \left(- \frac{\alpha}{\sqrt{2 a \sigma^2}}   \right)}\,.
\end{equation*}

This density allows to understand the dynamics of the system by means of ergodic theory: the unique stationary measure $\rmd \rho = p(x,y) \rmd(x,y)$ gives rise to an ergodic invariant measure $\mu$ (see  Appendix~\ref{app:invmeas}) for the skew product flow $(\Theta_t)_{t \in\R^+_0}$ on $\Omega\times \mathbb R^2$, defined by
\[
\Theta_t(\omega,Z):= (\theta_t\omega, \varphi(t,\omega,Z))\,.
\] 
To analyse asymptotic stability, we study the linearisation $\Phi(t,\omega,Z):=\rmD_Z\varphi(t,\omega,Z)$. A direct computation yields that $\Phi(0,\omega,Z)=\Id$ and
\begin{equation}\label{Linearization}
\dot\Phi(t,\omega,Z)=
\rmD f(\varphi(t,\omega,Z))\Phi(t,\omega,Z)\,,
\end{equation}
where
\begin{equation*}
\rmD f(x,y) = \begin{pmatrix}
  \alpha - a y^2 - 3a x^2 -2byx & -\beta -2axy -bx^2-3by^2\\
  \beta- 2axy + by^2 + 3bx^2&  \alpha - ax^2 - 3ay^2 + 2byx
\end{pmatrix}\,.
\end{equation*}
The key observation is that $(\Theta, \Phi)$ is a linear random dynamical system, where the ergodic dynamical system $(\theta_t)_{t\in\R}$ is replaced by $(\Theta_t)_{t \in\R^+_0}$.
We know from \cite{DoanEngelLambRasmussen} that the linear system $\Phi$ defined in \eqref{Linearization} satisfies the integrability condition
\[
\sup_{0 \leq t \leq 1} \ln^+ \| \Phi(t, \omega,Z) \| \in L^1(\mu)\,.
\]
Therefore, we can apply Oseledets' Multiplicative Ergodic Theorem to obtain the Lyapunov spectrum of the linear random dynamical system $(\Theta, \Phi)$ (see Appendix~\ref{appx-LyapSpec}). In particular, the largest Lyapunov exponent $\Lambda_1$ is given by
\begin{equation*}
\Lambda_1
=
\lim_{t \to \infty}
\frac{1}{t} \ln\|\Phi(t,\omega,Z)\|\quad \hbox{for } \mu\hbox{-almost all } (\omega,Z)\in\Omega\times \mathbb R^2\,.
\end{equation*}
This exponent is the crucial measure of stability. In case $\Lambda_1$ is negative, synchronisation of trajectories can be proven, as in \cite{DoanEngelLambRasmussen}, where it has also been shown that $\Lambda_1$ is negative if $|b|\leq \kappa$, where
  \[
  \kappa:=
  a\sqrt{\frac{\pi K_{a,\alpha,\sigma} \sigma^2}{\alpha+\pi K_{a,\alpha,\sigma} \sigma^2}\left(\frac{\pi K_{a,\alpha,\sigma} \sigma^2}{\alpha+\pi K_{a,\alpha,\sigma} \sigma^2}+2\right)}.
  \]
Furthermore, DeVille et al.~\cite{dsr11} have demonstrated that $\Lambda_1 < 0$ for $\sigma \frac{a}{\alpha} \to 0$, i.e.~for sufficiently small noise.
Numerical evidence from \cite{dsr11, DoanEngelLambRasmussen, wieczorek2009stochastic} has suggested that large shear $\left|b \right|$ leads  to a positive largest Lyapunov exponent, indicating chaotic behavior. Except for the strongly simplified model in \cite{engellambrasmussen19_1} inspired by numerical experiments in \cite{ly08}, an analytical proof has so far appeared out of reach, and the following conjecture has been formulated:
\begin{conjecture}[\cite{DoanEngelLambRasmussen}] \label{conjecture}
  Consider the random dynamical system induced by the stochastic differential equation~\eqref{MainEq_killed}, and fix $a > 0$ and $\beta \in \mathbb{R}$. Then there exists a function $C: \R\times \R^+ \to \mathbb{R}^+$ such that if
  \[
  	b\geq C(\alpha, \sigma)\,,
  	\]
  then the largest Lyapunov exponent $\Lambda_1$ is positive.
\end{conjecture}\label{conj:R2}
\subsection{The main result} \label{sec:mainresult}
In this paper, we restrict to a bounded domain $E \subset \R^2$, where we consider the \emph{conditioned Lyapunov exponent} $\Lambda_c$, giving an approximation of $\Lambda_1$ but also having crucial dynamical significance in its own right. For this situation, we can prove a weaker version of Conjecture~\ref{conjecture}, by finding values of $b$, given the other parameters, such that the Lyapunov exponent $\Lambda_c$ is positive.

The main modification is to consider trajectories of the SDE starting inside the interior of $E \subset \mathbb{R}^d$ conditioned on the fact that they do not reach the boundary $\partial E$. In other words, the boundary $\partial E$ constitutes a trap, reached at the \emph{hitting} or \emph{absorption} time
$$ T = \inf \{t \geq0 \,:\, X_t \in \partial E \}.$$
In many situations of stochastic bifurcation theory, this setting may even be more appropriate and insightful since the deterministic counterparts are local bifurcations in contrast to the globally spreading noise. Considering the problem on a bounded domain helps to control the geometric forces within a dynamically relevant neighbourhood and makes the local effects of noise detectable. 

The theory of conditioned processes goes back to the pioneering work of Yaglom in 1947 \cite{yag47}, but in recent years, new ideas have been developed (see \cite{cmm13, mv12} for recent surveys). 
Due to the loss of mass by absorption at the boundary, the existence of a stationary distribution is impossible and, therefore, stationarity is replaced by quasi-stationarity. A \textit{quasi-stationary distribution} preserves mass along the process conditioned on survival. Given a unique quasi-stationary distribution for a Markov process $(Z_t)_{t \geq 0}$ on a state space $E$, one can derive the existence of a \textit{quasi-ergodic distribution} $m$ \cite{bro99}.
If the unit tangent bundle process $(Z_t, s_t)_{t \geq 0}$ possesses a joint quasi-ergodic distribution $m$ on $E\times \mathbb S^{d-1}$, 
Engel et al.~\cite{engellambrasmussen19_2} obtain the existence of a conditioned Lyapunov exponent, independently from $Z_0 \in E$ and $s_0 \in \mathbb S^{d-1}$,
\begin{equation} \label{Conditioned Lyapunov exponent}
\Lambda_c = \lim_{t \to \infty} \frac{1}{t} \mathbb{E}_{Z_0} \left[ \ln \| \rmD_Z \varphi(t,\cdot,Z_0) s_0 \| \bigg|  T > t \right] = \int_{\mathbb S^{d-1} \times E}  \langle s, \rmD_Z f(Z) s \rangle \ m( \rmd Z, \rmd s)\,,
\end{equation} 
where $\rmD_Z \varphi(t,\omega,Z_0)$ solves equation~\eqref{Linearization} for $t < T(\omega, Z_0)$ (see Appendix~\ref{appx-cle}).

\medskip
In this paper, we prove (with computer assistance) that for some given domain $E$ and parameter values $\alpha$, $\beta$, $a$, $b$ and $\sigma$, this conditioned Lyapunov exponent is positive. Here is a typical result that we can obtain. More examples with different parameter values are given in Section~\ref{sec:results}.
\begin{theorem} \label{thm:main}
Consider the random dynamical system induced by the SDE~\eqref{MainEq_killed} on an annulus $B_{\rmax}(0) \setminus B_{\rmin}(0) \subset \R^2$ with absorption at the boundary. For $\rmin=0.5$, $\rmax=1.5$, $a =\beta = \alpha = 1$, $ b =3.6$ and $\sigma = 1.3$,
the conditioned Lyapunov exponent $\Lambda_c$ is positive. 
\end{theorem}
The highly challenging ingredient for computing or estimating values of $\Lambda_c$ is to find the quasi-ergodic distribution $m$ in formula~\eqref{Conditioned Lyapunov exponent}. We will see that the statistics of the unit tangent bundle process $X_t = (Z_t, s_t)_{t \geq 0}$ can be obtained from an SDE on some bounded domain $\tilde E$
\begin{equation}\label{Gen_killed}
\rmd \tilde X_t= \tilde f (\tilde X_t)\rmd t
+
\tilde \sigma(\tilde X_t) \, \rmd W_t \,, \quad \tilde X_0 \in \tilde E \subseteq \R^d\,.
\end{equation}
Let us assume that the associated generator $\mathcal L$, given by
$$ \mathcal L  = \tilde f \cdot \nabla  + \frac{1}{2} \tilde \sigma \tilde \sigma^* : \nabla^2 \,,$$
and its formal $L^2$-adjoint $\mathcal L^*$ are uniformly elliptic. Then one observes (see e.g.~\cite{engellambrasmussen19_2} or for more general background~\cite[Chapter 6]{Schuss}) that the quasi-stationary distribution for the process solving~\eqref{Gen_killed} has the density $\phi$, vanishing at the boundary and satisfying for the exponential escape rate $\lambda_0 < 0$
\begin{equation*} 
  \mathcal{L}^* \phi = \lambda_0 \phi\,,  
\end{equation*}
where $\lambda_0$ is the eigenvalue with largest non-zero real part.
Furthermore we know (see e.g.~\cite{bro99,engellambrasmussen19_2}) that, given the eigenfunction $\eta$ with
\begin{equation*} 
\mathcal{L} \eta = \lambda_0 \eta\,, \ \eta = 0 \text{ on } \partial \tilde E\,,
\end{equation*}
the quasi-ergodic distribution $m$ satisfies
\begin{equation*} 
m(\rmd x) = \eta(x) \phi(x) \rmd x\,.
\end{equation*}
For our situation we will see that the calculations of $\eta$ and $\phi$, and by that $m$, can be done numerically; to make the calculation rigorous, we need guaranteed error bounds on the computed objects. This is a highly non-trivial task, which the major part of the paper is dedicated to, and which is computer-assisted.

\medskip

Since the proof of the universality of the Feigenbaum constant~\cite{Lan82}, and later on the proof of chaos~\cite{GalZgl98, MisMro95} and of the existence of a strange attractor~\cite{Tuc02} in the Lorenz system, computer-assisted proofs have become more and more frequent in dynamical systems. The techniques that we use in this paper fall into the category of \emph{a posteriori validation methods}, meaning that we first compute a numerical approximation of the solution of interest --- in our case, an approximate eigenpair of $\L$ or $\L^*$ --- and then use a fixed point argument to simultaneously prove the existence of an exact solution nearby and get explicit error bounds. We describe such techniques in more details in Section~\ref{sec:validation}, and refer the interested reader to the survey papers~\cite{Gom19,KapMroWilZgl20,KocSchWit96,Rum10,BerLes15} and books~\cite{NakPluWat19,Tuc11} for a broader overview on rigorous numerics and computed-assisted proofs for non-linear equations. 

Until recently, most of these computer-assisted proofs for dynamical system where focused on \emph{deterministic} dynamical systems. However, several questions about stochastic dynamical systems can be reduced to questions about deterministic objects, for instance using large deviation theory in the small-noise case, or more generally via the transfer operator (resp. forward Kolmogorov/Fokker-Planck operator) for stochastic maps (resp. stochastic differential equations). These deterministic objects can be studied very precisely using computer-assisted techniques, and the results can then be transferred back to give new rigorous insight on the initial stochastic system, see e.g.~\cite{BreKue19, GalMonNis20}. This is the general strategy that we pursue in this work.
While we believe that the main interest of this work lies in the result itself, namely the proof of shear induced chaos, we also believe that some of the techniques that we introduce might be of interest. In particular, even if the fundamental ideas behind the computer-assisted part of this work are by now standard in some communities, it is, up to our knowledge, the first time that these ideas could be adapted and brought to fruition in order to directly handle an elliptic operator with a leading differential operator (in our case a Laplacian) having non-constant coefficients; this could open the door for many interesting further problems, in particular in the context of Fokker-Planck equations associated to SDEs.

The remainder of this paper is structured as follows. Section~\ref{sec:FKformula} describes the derivation of the formula for the first and conditioned Lyapunov exponent and introduces the corresponding PDE problem. In Section~\ref{sec:validation} we give the abstract framework for the computer-assisted proof and show that the problem at hand fits into this setting with suitable a-priori bounds. Finally, in Section~\ref{sec:numerics}, we implement the proof method and conduct the rigorous numerics to get tight enclosures of the conditioned Lyapunov exponent for different parameter values, including those of Theorem~\ref{thm:main}. Appendices \ref{appx-rds}, \ref{appx-LyapSpec} and \ref{appx-cle} provide background information on random dynamical systems and Lyapunov exponents while appendices \ref{appx-elem_est}, \ref{appx-embedding} and \ref{appx-validation_inv_r} include basic estimates, embedding constants for the functional-analytic framework and error bounds for the rigorous numerics.


\section{Furstenberg-Khasminskii formula and PDE formulation} \label{sec:FKformula}
Based on an approach by DeVille et al.~in \cite{dsr11}, we consider the two-dimensional problem~\eqref{MainEq_killed} in polar coordinates
\begin{equation*}
r = \sqrt{x^2 + y^2}, \ \phi = \arctan(\frac{y}{x})\,.
\end{equation*}
Applying It\^{o}'s rule to the stochastic differential equation~\eqref{MainEq_killed} we obtain
\begin{equation} \label{Polar_original}
\left\{\begin{aligned} 
\rmd r &= \left( \alpha r  - a r^3  +  \frac{\sigma^2}{2r} \right) \rmd t  + \sigma (\cos \phi \, \rmd W_t^1 +  \sin \phi \, \rmd W_t^2), \nonumber\\
\rmd \phi &= (\beta + br^2) \, \rmd t + \frac{\sigma}{r}( -  \sin \phi \, \rmd  W_t^1 +  \cos \phi \, \rmd W_t^2).
\end{aligned}\right.
\end{equation}
This form illustrates the role of the parameter $b$ inducing a shear force: if $b > 0$, the phase velocity $\frac{\rmd \phi}{\rmd t}$ depends on the amplitude $r$. 
Since Gaussian random vectors are invariant under orthogonal transformations, we can define the independent Wiener processes
\begin{align*}
\rmd W_r &= \cos \phi  \,\rmd W_t^1 + \sin \phi \, \rmd W_t^2,\\ 
\rmd W_\phi  &= - \sin \phi \, \rmd  W_t^1 + \cos \phi  \, \rmd W_t^2.
\end{align*}
Hence, the Markov process solving 
\begin{equation*}
\left\{\begin{aligned}
&\d r = \left(\alpha r - ar^3 + \frac{\sigma^2}{2r}\right)\d t + \sigma \d W_r \\
&\d \phi = \left(\beta + br^2\right) \d t + \frac{\sigma}{r} \d W_\phi.
\end{aligned}\right.
\end{equation*}
corresponds with \eqref{MainEq_killed} in terms of the It\^{o} integral. 

The associated variational equation~\eqref{Linearization}, also when taking into account killing at the boundary, reads in polar coordinates
\begin{equation*}
\left\{\begin{aligned}
&\d\rho =\rho \left(\alpha -2ar^2 +r^2\sqrt{a^2+b^2}\sin(2\theta-\chi_0-2\phi)\right) \d t \\
&\d \theta = \left(\beta + 2br^2 +r^2\sqrt{a^2+b^2}\cos(2\theta-\chi_0-2\phi)\right) \d t,
\end{aligned}\right.
\end{equation*}
where $\chi_0=\arccos\left(\frac{b}{\sqrt{a^2+b^2}}\right)$. Introducing $\psi=2\theta-\chi_0-2\phi$ we see that, the linear expansion rate
\begin{equation*}
e(r,\psi)=\alpha -2ar^2 +r^2\sqrt{a^2+b^2}\sin\psi
\end{equation*}
is determined only by $r$ and $\psi$, which satisfy the following system
\begin{equation} \label{eq:SDE_FKformula}
\left\{\begin{aligned}
&\d r = \left(\alpha r - ar^3 + \frac{\sigma^2}{2r}\right)\d t + \sigma \d W_r, \\
&\d \psi = 2r^2\left(b + \sqrt{a^2+b^2}\cos\psi\right)\d t - \frac{2\sigma}{r} \d W_\varphi.
\end{aligned}\right.
\end{equation}
The associated \emph{backward} and \emph{forward Kolmogorov operators} are then given by
\begin{align} \label{eq:backwardKO}
Lu = \frac{\sigma^2}{2}\left(\frac{\partial^2 u}{\partial r^2} + \frac{4}{r^2}\frac{\partial^2 u}{\partial \psi^2}\right) + \left(\alpha r - ar^3 + \frac{\sigma^2}{2r}\right)\frac{\partial u}{\partial r} + 2r^2\left(b+\sqrt{a^2+b^2}\cos\psi\right)\frac{\partial u}{\partial \psi},
\end{align}
and
\begin{align} \label{eq:forwardKO}
L^*u = \frac{\sigma^2}{2}\left(\frac{\partial^2 u}{\partial r^2} + \frac{4}{r^2}\frac{\partial^2 u}{\partial \psi^2}\right) - \frac{\partial}{\partial r}\left[\left(\alpha r - ar^3 + \frac{\sigma^2}{2r}\right)u\right] -\frac{\partial}{\partial \psi}\left[2r^2\left(b+\sqrt{a^2+b^2}\cos\psi\right) u\right].
\end{align}
In the case without killing, one can derive the \emph{Furstenberg-Khasminskii formula} \cite{a98} for the largest Lyapunov exponent $\Lambda_1$ on $E=\R^2$
\begin{equation*}
\Lambda_1 = \iint e(r,\psi) p(r, \psi) \d r \d \psi,
\end{equation*}
where $p$ is the stationary density for system~\eqref{eq:SDE_FKformula}, solving
the stationary forward Kolmogorov equation
$$ L^* p =0.$$
Finding or just making useful estimates for this density has proven to be extremely difficult (cf.~\cite{dsr11, DoanEngelLambRasmussen}). A rigorous computation of $p$ on the whole space $\R^2$ also seems out of reach for the moment.

Hence, we make use of the theory of conditioned Lyapunov exponents as an approximation of $\Lambda_1$ in a well-defined analytical framework, and consider the SDE~\eqref{MainEq_killed} on the bounded domain $E:=\tilde \Omega=B_{\rmax}(0)\setminus B_{\rmin}(0)$, with $0<\rmin<\rmax<\infty$.
Denoting by $\eta$ and $\phi$ the (normalized) eigenfunctions associated to the eigenvalue $\lambda_0$ with largest real part of $L$ and $L^*$ respectively, the conditioned Lyapunov exponent~\eqref{Conditioned Lyapunov exponent} can be expressed by the modified Furstenberg-Khasminskii formula \cite{engellambrasmussen19_2}
\begin{equation} \label{eq:FKformula_cond}
 \Lambda_c = \iint e(r,\psi) \eta(r,\psi) \phi(r,\psi) \d r \d \psi.
\end{equation}
\begin{figure}[h!]
\begin{center}
\subfloat[\label{fig:eta_largedomain}$\eta(r, \psi)$]{
\begin{overpic}[width=0.45\linewidth]{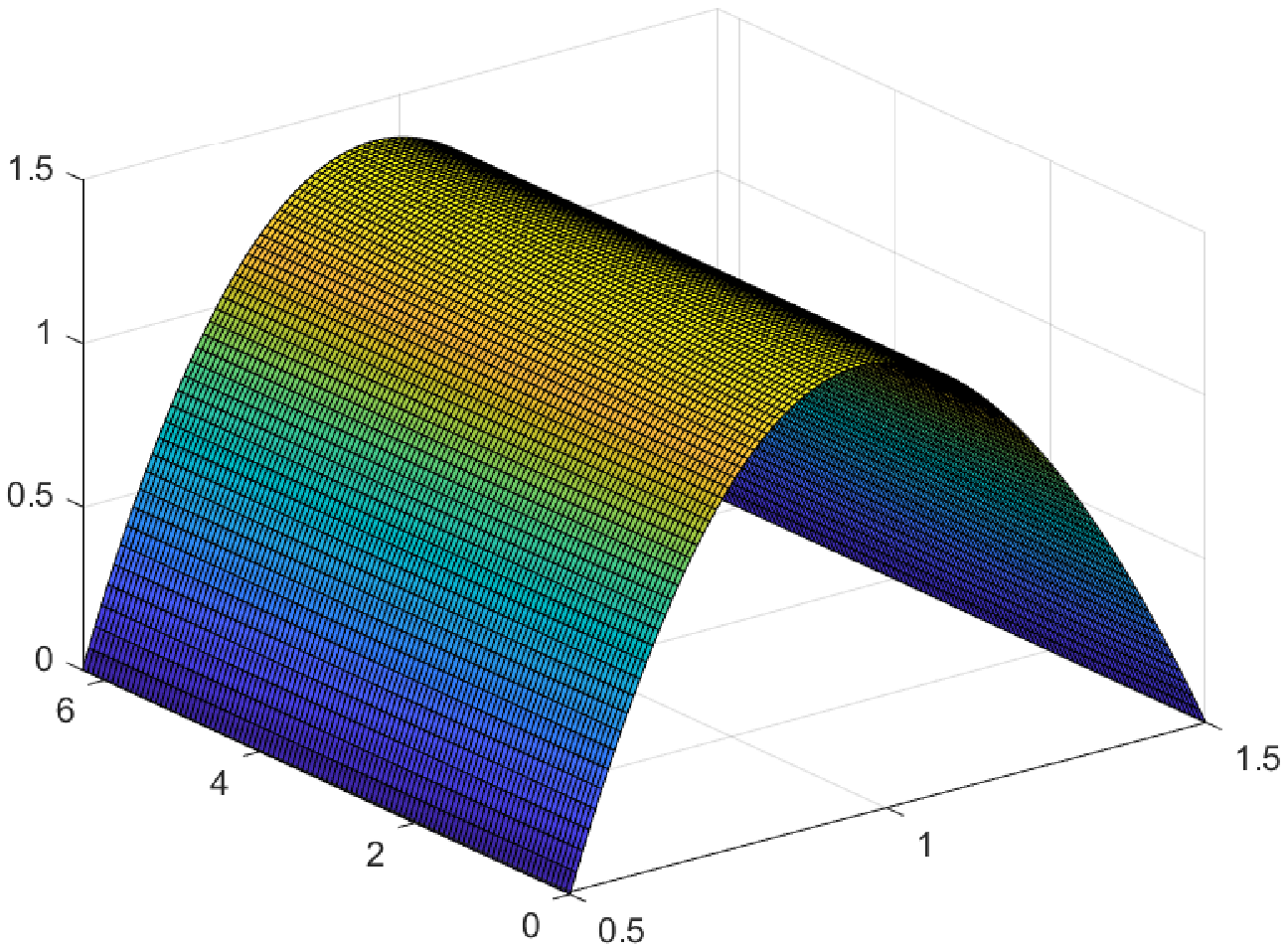}
\put(75,8){\scriptsize $r$}
\put(20,10){\scriptsize $\psi$}
\end{overpic} 
}
\subfloat[\label{fig:phi_largedomain}$\phi(r, \psi)$]{
\begin{overpic}[width=0.45\linewidth]{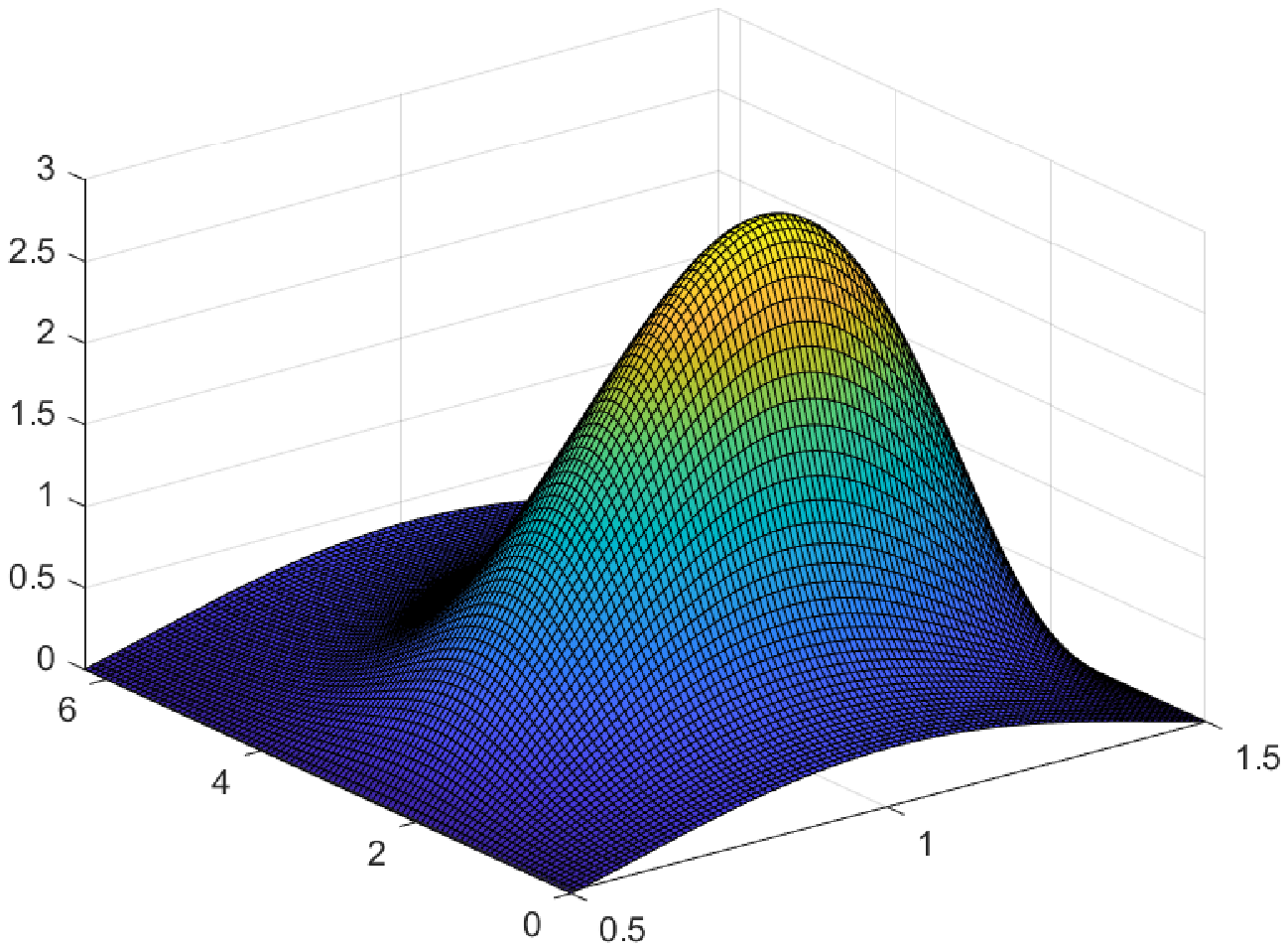}
\put(75,8){\scriptsize $r$}
\put(20,10){\scriptsize $\psi$}
\end{overpic} 
}

\subfloat[\label{fig:e_largedomain}$e(r, \psi)$]{
\begin{overpic}[width=0.45\linewidth]{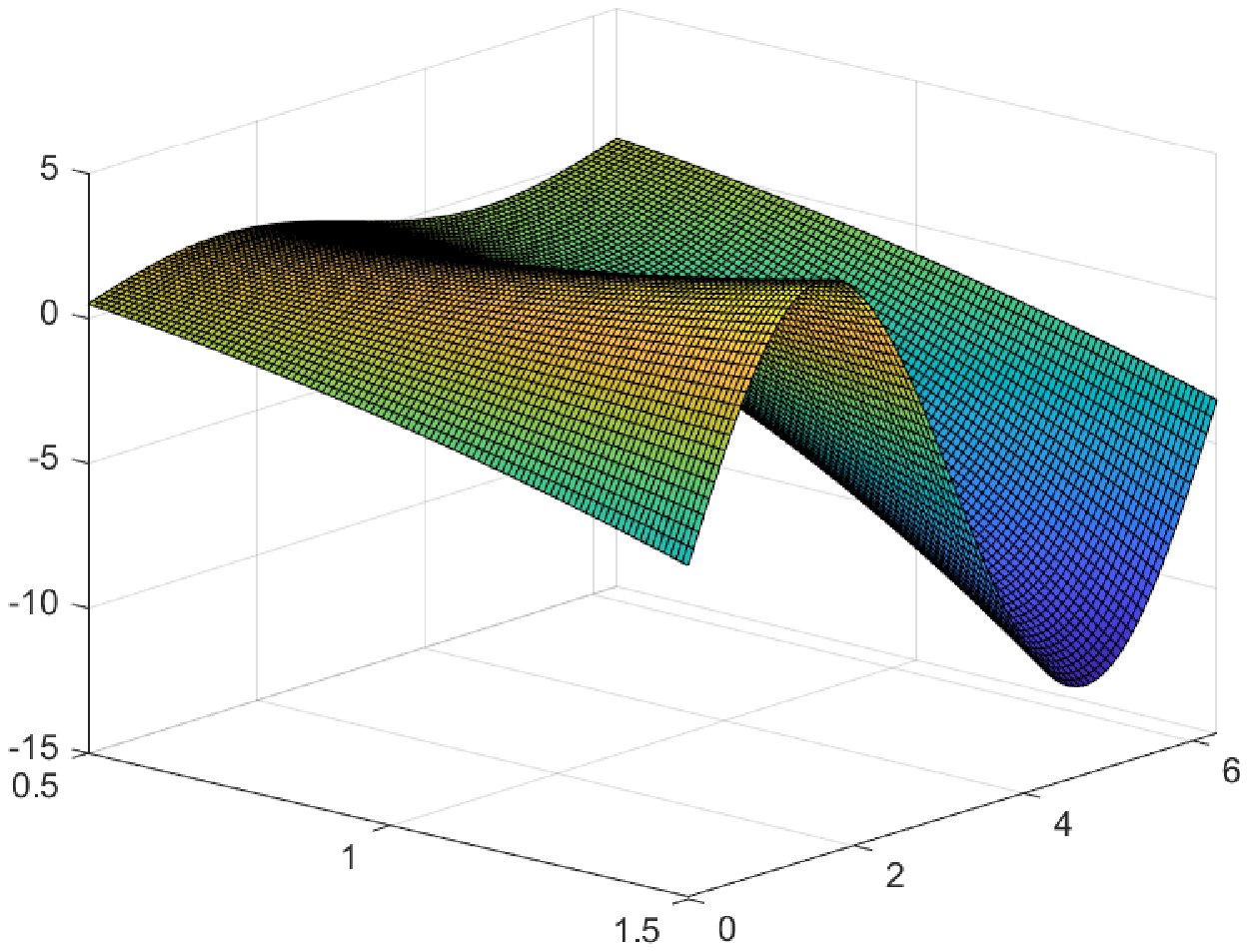}
\put(75,8){\scriptsize $\psi$}
\put(20,10){\scriptsize $r$}
\end{overpic} 
}
\end{center}
\caption{The eigenfunctions $\eta$ of $L$ (a) and $\phi$ of $L^*$ (b), and the modified Furstenberg-Khasminskii functional $e$ (c), for the domain and the parameter values as used in Theorem~\ref{thm:main}.}
\label{fig:eigenfunctions_rlarge}
\end{figure}
In Figure~\ref{fig:eigenfunctions_rlarge}, we illustrate numerical approximations of the functions $e, \eta, \phi$ involved in formula~\eqref{eq:FKformula_cond}, with the domain and the parameter values chosen as in Theorem~\ref{thm:main}. Observe that the function $e$ has regions of positive and negative output, expressing areas of expansion and contraction as it is typical for non-uniformly hyperbolic dynamics. Hence, to be able to say something about the sign of the integral~\eqref{eq:FKformula_cond}, one must be able to describe quite precisely the eigenfunctions $\eta$ and $\phi$. This is why a computer-assisted argument seems particularly relevant (and almost mandatory) here, as it will allow us to get a precise numerical approximation of $\eta$ and $\phi$ together with guaranteed and reasonably tight error bounds. In the remainder of the paper, we show how to obtain such rigorous computation of the eigenfunctions, and use them to study the sign of $\Lambda_c$.
\begin{remark}
In general, when the bounded domain $E \to \mathbb R^2$, we clearly have that the escape rate and eigenvalue $\lambda_0$ converges to $0$ such that $\eta \to 1$ pointwise and $\phi \to p$ as tempered distributions, and, hence, $\Lambda_c \to \Lambda_1$ (see \cite{engellambrasmussen19_2} for more details). In our case, we can take $\rmax \to \infty$ accordingly but for $\rmin \to 0$ we arrive at absorption at the origin which is not congruent with the full $\mathbb{R}^2$-model. However, for $\alpha >0$, small neighbourhoods of the origin contribute to a positive rather than a negative Lyapunov exponent such that our results can still be seen also as strong evidence for Conjecture~\ref{conj:R2}; see Section~\ref{sec:results} for a more detailed discussion of computational findings and numerical indications with respect to this issue.
\end{remark}


\section{A posteriori validation of an eigenpair}
\label{sec:validation}

In this Section, we describe the computer-assisted techniques allowing us to get our hands on the eigenfunctions $\eta$ and $\phi$ required for the computation of the Furstenberg-Khasminskii formula~\eqref{eq:FKformula_cond}. In Section~\ref{sec:notations}, we start by introducing some notations and the functions spaces we are going to work in. We then reformulate the problem of looking for eigenfunctions and eigenvalues of the backward Kolmogorov operator $L$ into a fixed-point problem in Section~\ref{sec:fixed-point}, and give sufficient conditions for a fixed-point Theorem to be used. These sufficient conditions are to be checked, rigorously, with the help of a computer, and we present the required theory in Section~\ref{sec:norm_inverse}. Eigenfunctions and eigenvalues of the forward Kolmogorov operator $L^*$ can be handled in a very similar fashion, and we describe the small adjustement that have to be made in Section~\ref{sec:Lstar}. Finally, in~\ref{sec:smallesteigenvalue}, we show how to make sure that we did not obtain any eigenpairs of $L$ and $L^*$, but that we indeed have the ones needed for the Furstenberg-Khasminskii formula~\eqref{eq:FKformula_cond}.

Although some of the arguments presented in this section rely on the computer in a crucial way, we mainly focus on the theoretical aspects here, and postpone the discussions about practical implementation to Section~\ref{sec:numerics}.

\subsection{Setting and notations}
\label{sec:notations}

Introducing the notations
\begin{equation} \label{nablatilde}
\tN = \begin{pmatrix}
\frac{\partial}{\partial_r} \\
\frac{2}{r}\frac{\partial}{\partial_\psi}
\end{pmatrix},
\end{equation}
\begin{equation} \label{deltatilde}
\tD u := \tN \cdot \tN u = \frac{\partial^2 u}{\partial r^2} + \frac{4}{r^2}\frac{\partial^2 u}{\partial \psi^2},
\end{equation}
\begin{equation} \label{fandg}
f(r) = \alpha r-ar^3 +\frac{\sigma^2}{2r}, \quad g(r,\psi) = 2r^2\left(b+\sqrt{a^2+b^2}\cos\psi\right),
\end{equation}
\begin{equation} \label{h}
h(r,\psi) = \frac{\partial f}{\partial r}(r) + \frac{\partial g}{\partial \psi}(r,\psi) = \alpha -3ar^2-\frac{\sigma^2}{2r^2} - 2r^2\sqrt{a^2+b^2}\sin\psi,
\end{equation}
and 
\begin{equation} \label{Vformula}
V = f(r)\frac{\partial}{\partial r} + g(r,\psi)\frac{\partial}{\partial \psi},
\end{equation}
we can write the backward and forward Kolmogorov operators~\eqref{eq:backwardKO} and~\eqref{eq:forwardKO} in a more condensed form:
\begin{equation*}
Lu = \frac{\sigma^2}{2}\tD u + V u \quad \text{and}\quad L^*u = \frac{\sigma^2}{2}\tD u + V^* u = \frac{\sigma^2}{2}\tD u - Vu -hu.
\end{equation*}
We consider the domain
\begin{equation} \label{eq:domain}
\Omega=(\rmin,\rmax)\times(0,2\pi).
\end{equation}
The main functions spaces that we are going to use are  $L^2$ and the Sobolev space $H^2$, with appropriate boundary conditions:
\begin{equation*}
X:= \left\{ u\in H^2(\Omega) \ |\ u(\rmin,\psi)=u(\rmax,\psi)=0\ \forall~\psi\in(0,2\pi),\ u(r,0)=u(r,2\pi)\ \forall~r\in (\rmin,\rmax) \right\},
\end{equation*}
\begin{equation*}
Y:= L^2(\Omega).
\end{equation*}
These spaces will be equipped with weighted inner products:
\begin{equation*}
\langle u_1,u_2\rangle_X :=  \langle u_1,u_2\rangle_{L^2} +  \langle \tN u_1,\tN u_2\rangle_{L^2} + \xi_2 \langle \tD u_1, \tD u_2\rangle_{L^2},\qquad \langle u_1,u_2\rangle_Y :=  \langle u_1,u_2\rangle_{L^2},
\end{equation*}
where the scalar product on $L^2(\Omega)$ is defined by
\begin{equation}
\label{eq:scal_L2}
\langle u_1,u_2\rangle_{L^2} := \frac{1}{2\pi(\rmax-\rmin)}\int_\Omega u_1(r,\psi) \left(u_2(r,\psi)\right)^*\, dr\, d\psi,
\end{equation}
$u_2^*$ denotes the complex conjugate and $\xi_2$ is some positive weight (whose role will become apparent in Section~\ref{sec:kappa0tokappa}). We denote the associated norms by
\begin{equation*}
\left\Vert u \right\Vert_X := \sqrt{\langle u,u\rangle_X} \qquad\text{and}\qquad 
\left\Vert u \right\Vert_Y := \sqrt{\langle u,u\rangle_Y}.
\end{equation*}
Notice that $\left\Vert \cdot \right\Vert_X$ is equivalent to the canonical norm on $H^2$ (see Appendix~\ref{appx-elem_est}), therefore $\left(X,\left\Vert\cdot  \right\Vert_X\right)$ is still a Hilbert space.

\begin{remark}
We could have added wore weights in the norm on $X$, by considering
\begin{equation*}
\langle u_1,u_2\rangle_X :=  \xi_0 \langle u_1,u_2\rangle_{L^2} +  \xi_1 \langle \tN u_1,\tN u_2\rangle_{L^2} + \xi_2 \langle \tD u_1, \tD u_2\rangle_{L^2}.
\end{equation*}
Adding such weights is often a good idea for the quantitative type of arguments we are going to use in this work, because their value can be chosen a posteriori in order to make the validation easier. However, in this case it turned out that this extra flexibility did not bring significant improvements, and we therefore chose to remove those weights in order to simplify a bit the presentation.
\end{remark}

Since we are going to solve for both an eigenfunction and an eigenvalue at the same time, we must append an extra scalar field to our spaces, and thus define
\begin{equation*}
\X := X\times\C \qquad\text{and}\qquad \Y := Y\times\C.
\end{equation*}
A generic element in $\X$ (resp. $\Y$) will typically be denoted by $(u,\lambda)$, where $u\in X$ (resp. $u\in Y$) and $\lambda\in \C$. The inner products (and the associated norms) we consider on $\X$ and $\Y$ are directly inherited from those defined on $X$ and $Y$:
\begin{equation*}
\langle (u_1,\l_1) , (u_2,\l_2) \rangle_\X := \langle u_1 , u_2 \rangle_X + \l_1\l_2^* \qquad\text{and}\qquad \langle (u_1,\l_1) , (u_2,\l_2) \rangle_\Y := \langle u_1 , u_2 \rangle_Y + \l_1\l_2^*.
\end{equation*}

\subsection{The fixed point problem}
\label{sec:fixed-point}

We now give the analytical framework within which we are going to conduct the computer-assisted proof.

Let $(\bu,\bl)$ be a numerically computed approximate eigenpair of $L$, and the map $\mathcal F: \X \to \Y$ defined by
\begin{equation} \label{OperatorcalF}
\mathcal{F} [(u, \lambda)] := 
\begin{pmatrix}[1.5]
L u - \lambda u \\ 
\langle u,\bu\rangle_{L^2} -1
\end{pmatrix}.
\end{equation}
Our first goal is to prove that there is a genuine zero of $\F$ (i.e.~a genuine eigenpair of $L$) in a small and explicit neighborhood of $(\bu, \bl)$. 
\begin{remark}
\label{rem:normalization}
Since our proof will be based on a contraction argument, the map $\F$ that we consider must have a locally unique zero in a neighborhood of $(\bu,\bl)$, which is why we must append a \emph{normalization condition} to the equation $Lu-\l u =0$. Many such conditions could be chosen, but taking $\langle u,\bu\rangle_{L^2} -1=0$ will prove convenient later on, in particular for the estimates of Lemma~\ref{lem:kappa}.
\end{remark}
\begin{remark}
For the moment, we focus exclusively on rigorously validating an eigenpair for $L$, but the same method can and will be used to validate an eigenpair of $L^*$. In Section~\ref{sec:Lstar}, we highlight the small changes that have to be made when we consider $L^*$ instead of $L$.
\end{remark}
The corner stone of our compter-assisted approach is the following Newton-Kantorovich-like theorem. Very similar statements have already appeared in many instances, especially in the computer-assisted-proof literature (and at least as early as 1965~\cite{Ura65}). The version below is strongly inspired from \cite[Theorem 6.2]{NakPluWat19}.
\begin{theorem} \label{thm:zero_of_F}
Let $(\bu,\bl)\in \X$ and $\mathcal F$ be as in~\eqref{OperatorcalF}. Denoting by $\F'$ the Frechet derivative of $\F$, we assume there exist nonnegative constants $\delta$, $\kappa$ and $\gamma$ satisfying
\begin{align}
\label{eq:delta}
\Vert \F[(\bu,\bl)] \Vert_\Y &\leq \delta, \\
\label{eq:kappa}
\Vert (u,\l) \Vert_\X &\leq \kappa \Vert \F'(\bu,\bl)[(u,\l)] \Vert_\Y \qquad \forall~(u,\l)\in \X,\\
\label{eq:gamma}
\Vert \F''(\bu,\bl)[(u_1,\l_1),(u_2,\l_2)] \Vert_\Y &\leq \gamma \Vert (u_1,\l_1) \Vert_\X \Vert (u_2,\l_2) \Vert_\X \qquad \forall~(u_1,\l_1),(u_2,\l_2)\in \X,
\end{align}
such that
\begin{equation}
\label{eq:cond_fixed-point}
2\kappa^2\gamma\delta <1.
\end{equation}
Then, for any $\rho$ satisfying
\begin{equation}
\label{eq:interval_rho}
\frac{1-\sqrt{1-2\kappa^2\gamma\delta}}{\kappa\gamma} \leq \rho < \frac{1}{\kappa\gamma},
\end{equation}
 $\F$ has a unique zero $(u, \lambda) \in \X$, such that
\begin{equation} \label{eq:rho_close}
\|(u,\lambda) - (\bu, \bl) \|_\X \leq \rho.
\end{equation}
\end{theorem}
\begin{proof}
The main idea is to apply Banach's fixed point theorem to the Newton-like operator
\begin{equation}
\label{eq:T}
T:(u,\l) \mapsto (u,\l) - \left(\F'(\bu,\bl)\right)^{-1} \F(u,\l).
\end{equation}
First, we show that $\F'(\bu,\bl)$ is indeed an isomorphism between $\X$ and $\Y$. For any $(u_1,\l_1)\in\X$ and $(u_2,\l_2)\in\Y$, the equation
\begin{equation}
\label{eq:Fprime}
\F'(\bu,\bl)(u_1,\l_1) = (u_2,\l_2),
\end{equation}
is equivalent to
\begin{equation*}
\left\{
\begin{aligned}
&u_1 + \frac{2}{\sigma^2}\tD^{-1}\left(Vu_1-\bl u_1-\lambda_1\bu\right) = \frac{2}{\sigma^2}\tD^{-1}u_2, \\
&\lambda_1 + \left(\langle u_1,\bu\rangle -1 -\lambda_1\right) = \lambda_2.
\end{aligned}
\right.
\end{equation*}
Since $\Omega$ is bounded, the operator on the left hand side is a compact perturbation of the identity, and the Fredholm alternative holds. Besides, we know by~\eqref{eq:kappa} that, when $(u_2,\l_2) = (0,0)$, the only solution of~\eqref{eq:Fprime} is $(u_1,\l_1) =(0,0)$. The Fredholm alternative then yields that~\eqref{eq:Fprime} has a unique solution for any $(u_2,\l_2)\in\Y$.   

The operator $T$ introduced in~\eqref{eq:T} is thus well defined, and maps $\X$ into itself. Besides, $T'(\bu,\bl)=0$, and therefore $T$ will be contracting near $(\bu,\bl)$. Provided $(\bu,\bl)$ gets not mapped too far away from itself by $T$, we should have a stable neighborhood of $(\bu,\bl)$ on which we can apply the contraction mapping theorem. The hypothesis~\eqref{eq:delta}-\eqref{eq:cond_fixed-point} allow us to make this statement quantitative. For any $\rho\geq0$, we denote by $B_{\X,\rho}(\bu,\bl)$ the closed ball with center $(\bu,\bl)$ and radius $\rho$ in $\X$, and we now show that, for any $\rho$ satisfying~\eqref{eq:interval_rho}, $T$ maps $B_{\X,\rho}(\bu,\bl)$ into itself, and that it is contracting on this ball. 

For any $(u,\l)\in\X$, we estimate, using~\eqref{eq:delta}-\eqref{eq:gamma} and the fact that $T'(\bu,\bl)=0$,
\begin{align*}
\Vert T(u,\l) - (\bu,\bl)\Vert_\X &\leq \Vert T(u,\l) - T(\bu,\bl)\Vert_\X + \Vert T(\bu,\bl) - (\bu,\bl)\Vert_\X \\
 &= \frac{1}{2}\Vert T''(\bu,\bl)[(u-\bu,\l-\bl),(u-\bu,\l-\bl)]\Vert_\X + \Vert \left(\F'(\bu,\bl)\right)^{-1} \F(\bu,\bl)\Vert_\X \\
 &\leq \frac{1}{2}\kappa\gamma\Vert (u-\bu,\l-\bl) \Vert_\X^2  + \kappa\delta.
\end{align*}
Assuming $(u,\l)\in B_{\X,\rho}(\bu,\bl)$, where $\rho$ satisfies~\eqref{eq:interval_rho}, we have
\begin{align*}
\Vert T(u,\l) - (\bu,\bl)\Vert_\X &\leq \frac{1}{2}\kappa\gamma\rho^2  + \kappa\delta \\
&\leq \rho,
\end{align*}
by~\eqref{eq:cond_fixed-point}, and therefore $T(B_{\X,\rho}(\bu,\bl)) \subset B_{\X,\rho}(\bu,\bl)$. Besides, for any $(u,\l)$ in $B_{\X,\rho}(\bu,\bl)$, we have
\begin{align*}
\Vert T'(u,\l) \Vert_\X &= \Vert T''(\bu,\bl)(u-\bu,\l-\bl) \Vert_\X \\
&\leq \kappa\gamma\rho \\
&<1
\end{align*}
by~\eqref{eq:cond_fixed-point}. The contraction mapping thus yields the existence of a unique fixed point of $T$ in $B_{\X,\rho}(\bu,\bl)$, which corresponds to a unique zero of $\F$ in that ball.
\end{proof}

In order to apply this theorem, we need computable estimates to find $\delta$, $\kappa$ and $\gamma$ satisfying~\eqref{eq:delta}-\eqref{eq:gamma}. Notice that one can think of the quantities $\kappa$ and $\gamma$ as being intrinsic to the problem, as they only depend on the first two derivatives of $\F$ at the zero we are interested in. Therefore, one can think of them as being given and fixed (even though in practice we will have to work quite hard to get them explicitly), and~\eqref{eq:cond_fixed-point} is going to be satisfied as soon as we have a \emph{sufficiently accurate} numerical approximation $(\bu,\bl)$, so that the residual error $\delta$ is less than $\frac{1}{2\kappa^2\gamma}$.

Using interval arithmetic, it is rather straightforward to obtain a computable and reasonably sharp estimate for~\eqref{eq:delta}. This is detailed in Section~\ref{sec:numerics}. Besides, noticing that the only nonlinear term in $\F$ is the product $\lambda u$, it is easy to check that~\eqref{eq:gamma} holds with $\gamma=1$.  Our main task is therefore to get a computable estimate for~\eqref{eq:kappa}. We will do so in Section~\ref{sec:norm_inverse}.

Before turning to inequality~\eqref{eq:kappa}, we give a corollary of Theorem~\ref{thm:zero_of_F} which will come in handy later on by allowing us to prove that the eigenfunction of $L$ is in fact one-dimensional, i.e. it does not depend on the angular component $\psi$.
\begin{corollary}
\label{cor:1D}
Make the same assumptions as in Theorem~\ref{thm:zero_of_F}, and suppose further that the function $\bu$ does not depend on the $\psi$-variable. Then the zero $(u,\lambda)$ given by Theorem~\ref{thm:zero_of_F} is such that $u$ also does not depend on the $\psi$-variable.
\end{corollary}
\begin{proof}
Consider the subspace $X_{\radial}$ of $X$ made of functions $u\in X$ which do not depend on the angle variable $\psi$. Similarly, we introduce $\X_{\radial}$ and $\Y_\radial$, and we have that $\F$ maps $\X_\radial$ into $\Y_\radial$. Since $\Y_\radial$ is a closed subset of $\Y$, and since we assumed that $(\bu,\bl)\in\X_\radial$, the derivative $\F'(\bu,\bl)$ also maps $\X_\radial$ into $\Y_\radial$. An argument similar to the one used in the proof of Theorem~\ref{thm:zero_of_F} shows that $\F'(\bu,\bl)$ is in fact an isomorphism between $\X_\radial$ and $\Y_\radial$. Therefore, as soon as $(\bu,\bl)\in\X_\radial$, the fixed point operator $T$ defined in~\eqref{eq:T} maps $\X_\radial$ into itself. In particular, from the proof of Theorem~\ref{thm:zero_of_F} we see that, for any $\rho$ satisfying~\eqref{eq:cond_fixed-point}, $T$ maps $B_{\X,\rho}(\bu,\bl)\cap \X_\radial$ into itself, and that it is contracting on this ball. Since $\X_\radial$ is closed in $\X$, the contraction mapping still applies and yields the existence of a unique zero $(u,\l)$ of $\F$ in $B_{\X,\rho}(\bu,\bl)\cap \X_\radial$, which corresponds to the zero of $\F$ described in Theorem~\ref{thm:zero_of_F} by uniqueness.
\end{proof}
\begin{remark}
Corollary~\ref{cor:1D} is mainly used for the sake of unifying the presentation and the code. Indeed, one could instead introduce one-dimensional versions of the spaces $\X$ and $\Y$, state an analog of Theorem~\ref{thm:zero_of_F} with these spaces, and derive the associated estimates in the following subsections as well. However, since we need the two-dimensional setup for $L^*$ whose  eigenfunction will depend on $\psi$, we choose to use it for $L$ as well, for which the corollary guarantees that the eigenfunction is in fact one-dimensional.
\end{remark}

\subsection{Obtaining a bound for the inverse of the derivative}
\label{sec:norm_inverse}

From now on, let us denote by $S$ the derivative of $\F$ at the approximate solution: 
\begin{equation} \label{OperatorS}
S [( u, \lambda)] := \mathcal{F}'(\bu, \bl)[( u, \lambda)] = \begin{pmatrix}[1.5]
L u - \bl u - \lambda \bu \\  \langle u,\bu\rangle_{L^2}
\end{pmatrix}.
\end{equation}
Our goal is to obtain an explicit constant $\kappa$ such that inequality~\eqref{eq:kappa} holds, which reads
\begin{equation*}
 \left\Vert ( u, \lambda) \right\Vert_\X \leq \kappa \left\Vert S[( u, \lambda)] \right\Vert_\Y \quad \forall~( u, \lambda)\in \X,
\end{equation*} 
or, at least formally for now, $\left\Vert  S^{-1}\right\Vert_{B(\Y,\X)}\leq \kappa$. This is usually the most challenging part of a computer-assisted proof based on an a posteriori fixed point argument, and two main strategies have been developed in order to tackle such problems. 

One of them, see e.g.~\cite{AriKocTer05,DayLesMir07,Ois95,WatKinNak20,Yam98}, consists in combining rigorous computations on a finite dimensional projection with a priori error estimates, which typically amounts to introducing an \emph{approximate inverse} of $S$ whose norm is easier to bound, and then to control the error between this approximate inverse and $S^{-1}$ itself. Because of the fact that the leading differential operator in $L$ and $L^*$, namely $\tD$, has non-constant coefficients, this strategy seems difficult to apply to our current problem. One option that might prove successful with this approach could be to use a discretization based on Zernike polynomials, as was done recently in~\cite{AriKoc19}, but we did not investigate this possibility.

The other approach, see e.g.~\cite{Plu91,TakLiuOis13}, consists in directly estimating the norm of the inverse of $S$ via rigorous eigenvalue bounds. Although this strategy has, to our knowledge, also never been applied before to a problem where the leading order operator --- typically a Laplacian or a bi-Laplacian --- has non-constant coefficients, it seems more amenable to this situation. Therefore we will follow this strategy in the remainder of Section~\ref{sec:validation}.

The starting point is the following. Introducing 
\begin{equation}
\label{eq:def_Z}
Z:= \left\{ u\in H^4(\Omega) \ |\ u\in X,\ L u\in X \right\} \qquad \text{and}\qquad \Z=Z\times\C
\end{equation}
we have, for all $(u,\l)\in\Z$,
\begin{align*}
\left\Vert S(u,\l) \right\Vert_\Y^2 &= \langle S(u,\l) ,S(u,\l) \rangle_\Y \\
&= \langle S^*S(u,\l) ,(u,\l) \rangle_\Y \\
&\geq \l_1(S^*S) \left\Vert (u,\l) \right\Vert_\Y^2,
\end{align*}
where $\l_1(S^*S)$ is the smallest eigenvalue of the self-adjoint operator $S^*S$. Therefore, if we manage to get an explicit lower bound on $\l_1(S^*S)$, we get a constant $\kappa_0$ such that
\begin{equation}
\label{eq:kappa0}
 \left\Vert ( u, \lambda) \right\Vert_\Y \leq \kappa_0 \left\Vert S[( u, \lambda)] \right\Vert_\Y \quad \forall~( u, \lambda)\in X.
\end{equation} 
Combining such a bound with a priori estimates of the form
\begin{equation}
\label{eq:apriori}
\Vert \tN u \Vert_{L^2} \lesssim \left\Vert S[( u, \lambda)] \right\Vert_\Y  \quad \text{and}\quad \Vert \tD u \Vert_{L^2} \lesssim \left\Vert S[( u, \lambda)] \right\Vert_\Y \qquad \forall~( u, \lambda)\in \X,
\end{equation}
we can then obtain~\eqref{eq:kappa} from~\eqref{eq:kappa0}.

Sections~\ref{sec:homotopy} to \ref{sec:eig_tD} are devoted to obtaining an explicit lower bound on $\l_1(S^*S)$, giving $\kappa_0$, and in Section~\ref{sec:kappa0tokappa} we derive the priori estimates~\eqref{eq:apriori} which finally yields $\kappa$.

\subsubsection{Homotopy and eigenvalue bounds}
\label{sec:homotopy}

In this section, we describe the so-called \emph{homotopy method}~\cite{Goe87,Plu90}, which is the key ingredient we use to get an explicit lower bound on $\l_1(S^*S)$. As we will use this argument several times, we do not write it explicitly for the operator $S^*S$ but present it in a slightly more general way. Nonetheless, we tailor the presentation to our specific need, and refer the reader to~\cite[Chapter 10]{NakPluWat19} for a more general description of these techniques, and for the proofs of Propositions~\ref{prop:RR},~\ref{prop:LM} and~\ref{prop:LM_bis}. Before proceeding further, we point out that there are other existing techniques to rigorously compute eigenvalues of self-adjoint operators, at least on bounded domains (see e.g. \cite{CanDusMadStaVoh20, Liu15}).

In this subsection, we consider a densely defined self-adjoint operator $\SS$ on a separable Hilbert space $\XX$,
and assume that its spectrum only consists of eigenvalues, which accumulate only at $+\infty$. This is the operator for which we want rigorous bounds on the eigenvalues. We further assume there exists a family of operators $\SS^{(s)}$, with the same properties, and additionally
\begin{itemize}
\item we know the eigenvalues of $\SS^{(0)}$ exactly (this assumption will be slightly relaxed later on, see Remark~\ref{rem:choice_S0}),
\item the eigenvalues increase with $s$, that is, denoting by 
\begin{equation*}
\lambda^{(s)}_1 \leq \lambda^{(s)}_2 \leq \ldots \leq \lambda^{(s)}_n \leq \ldots
\end{equation*}
the eigenvalues of $\SS^{(s)}$ counted with multiplicity and arranged in ascending order, we assume that
\begin{equation}
\label{eq:monotonicity_homotopy}
\forall~0\leq s' \leq s \leq 1,\ \forall~n\geq 1, \qquad \lambda^{(s')}_n \leq \lambda^{(s)}_n,
\end{equation}
\item $\SS^{(1)}=\SS$.
\end{itemize} 
Under those assumptions, we can use the homotopy method to get rigorous enclosures on the eigenvalues of $\SS$, using the following steps.

The first step is to obtain rigorous upper bounds for the eigenvalues. More precisely, for any given $s\in[0,1]$ and $M\in\N$, we can get rigorous upper bounds for the $M$ smallest eigenvalues of $\SS^{(s)}$ via the well known Rayleigh–Ritz method, see e.g.~\cite[Theorem 10.12]{NakPluWat19}.
\begin{proposition}
\label{prop:RR}
Let $x_1,\ldots,x_M\in D(\SS^{(s)})$ be linearly independent, define the matrices
\begin{equation*}
A_0 = \left(\langle x_i,x_j \rangle\right)_{1\leq i,j\leq M} \qquad\text{and}\qquad A_1 = \left(\langle \SS^{(s)}x_i,x_j \rangle\right)_{1\leq i,j\leq M},
\end{equation*}
and let
\begin{equation*}
\ol_1 \leq \ldots \leq \ol_M,
\end{equation*}
be the eigenvalues of the generalized eigenvalue problem
\begin{equation}
\label{eq:eig_RR}
A_1 v = \ol A_0 v.
\end{equation}
Then, for all $m\leq M$
\begin{equation*}
\l_m^{(s)} \leq \ol_m.
\end{equation*}
\end{proposition}

\begin{remark}
\label{rem:homotopy}
In order to get sharp bounds, the vectors $x_1,\ldots,x_M$ should be chosen to be good numerical approximations of the eigenvectors of $\SS^{(s)}$ associated to the $M$ smallest eigenvalues of $\SS^{(s)}$. For these bounds to be rigorous, one has to
\begin{itemize}
\item make sure that the approximate eigenvectors $x_i$ exactly belong to the domain $D(\SS^{(s)})$ of $\SS^{(s)}$,
\item rigorously compute the entries of $A_0$ and $A_1$,
\item rigorously solve the eigenvalue problem~\eqref{eq:eig_RR}.
\end{itemize}
For the specific problems we are interested in here (see Sections~\ref{sec:base_problem} and~\ref{sec:eig_tD}), this can be easily done with interval arithmetic, at least as long as $M$ is not too large (see Section~\ref{sec:rig_kappa} for more details). Notice that the computation of $x_1,\ldots,x_M$ can and should be done first with usual floating-point arithmetic.
\end{remark}

\begin{remark}
In this work, making sure that the approximate eigenvectors $x_i$ exactly belong to the domain of $\SS^{(s)}$ is an easy task, because of our choice of numerical method (see again Section~\ref{sec:numerics}). However, for situations where this would be a troublesome requirement, we point out the existence of an elegant workaround: the so-called Goerisch extension~\cite{BehGoe94} (see also~\cite[Section 10.2.3]{NakPluWat19}).
\end{remark}

The second step is to obtain rigorous lower bounds for the eigenvalues. The following method from Lehmann and Maehly, see e.g.~\cite[Theorem 10.14]{NakPluWat19}, can be used to do so, and more precisely to obtain lower bounds on the $M$ first eigenvalues of $\SS^{(s)}$, assuming some a priori knowledge on the "next" eigenvalue $\l^{(s)}_{M+1}$.

\begin{proposition}
\label{prop:LM}
Repeat the assumptions of Proposition~\ref{prop:RR}. Assume further that there exists $\nu\in\R$ such that
\begin{equation}
\label{eq:cond_LM}
\ol_M <\nu \leq \l^{(s)}_{M+1},
\end{equation}
define the matrices
\begin{equation*}
A_2 = \left(\langle \SS^{(s)}x_i,\SS^{(s)}x_j \rangle\right)_{1\leq i,j\leq M},\qquad B_1=A_1-\nu A_0 \quad\text{and}\quad B_2 = A_2 -2\nu A_1 + \nu^2 A_0,
\end{equation*}
let
\begin{equation*}
\mu_1 \leq \ldots \leq \mu_M
\end{equation*}
be the eigenvalues of the generalized eigenvalue problem
\begin{equation}
\label{eq:eig_LM}
B_1 v = \mu B_2 v,
\end{equation}
and assume that $\mu_M<0$. Then, for all $m\leq M$
\begin{equation*}
\ul_m \leq \l_m^{(s)},
\end{equation*}
where
\begin{equation*}
\ul_m := \nu + \frac{1}{\mu_{M+1-m}}.
\end{equation*}
\end{proposition}

In practice, one has to get sharp enough upper bounds in Proposition~\ref{prop:RR} first, otherwise $\ol_M>\l_{M+1}^{(s)}$ such that there is no hope of satisfying assumption~\eqref{eq:cond_LM}. Even assuming that $\ol_M<\l_{M+1}^{(s)}$, it may be challenging to find an explicit $\nu$ for which we can ensure $\nu \leq \l_{M+1}^{(s)}$: we do not have any lower bound at this point and that is precisely what we are trying to get with Proposition~\ref{prop:LM}. This is where the monotonicity~\eqref{eq:monotonicity_homotopy} of the homotopy plays a crucial role, because it allows us to get crude lower bounds on the eigenvalues of $\SS^{(s)}$ if we already control the eigenvalues of $\SS^{(s')}$ for some $s'<s$.

A somewhat informal description of how the whole procedure should look like, assuming one wants to get rigorous bounds on the $\M$ smallest eigenvalues of $\SS$, is given in Algorithm~\ref{alg:homotopy}.

\medskip

\begin{algorithm}[hp!]
 \textbf{Initialization:} Numerically compute a rough guess of the largest eigenvalue $\l^{(1)}_\M$ for which we want to get a rigorous lower bound in the end, and fix some $M$ large enough so that $\l^{(0)}_M$ is larger than this guess. Define $s_0=0$ and take $k=0$. \emph{At this stage we have rigorous lower and upper bounds for the $M-k$ first eigenvalues of $\SS^{(s_k)}$} (since we assumed we knew the eigenvalues of $\SS^{(0)}$).\\
 \While{$k<M-1$ and $s_k<1$}{
	\begin{itemize}
	\item Numerically find $\tilde{s}_{k+1}$ such that $\l^{(\tilde{s}_{k+1})}_{M-(k+1)}\approx \l^{(s_k)}_{M-k}$.
	\item Take $s_{k+1}$ slightly smaller than $\tilde{s}_{k+1}$, and compute rigorous upper bounds for the $ M-(k+1)$ first eigenvalues of $\SS^{(s_{k+1})}$ using Proposition~\ref{prop:RR}. Check that $\ol_{M-(k+1)}^{(s_{k+1})}<\l^{(s_k)}_{M-k}$. Otherwise reduce $s_{k+1}$, and repeat this step.
	\item Once we have $\ol_{M-(k+1)}^{(s_{k+1})}<\l^{(s_k)}_{M-k}$, by the monotonicity of the homotopy we have $\ol_{M-(k+1)}^{(s_{k+1})}<\l^{(s_{k+1})}_{M-k}$, and, using Proposition~\ref{prop:LM}, we can compute rigorous lower bounds of $\l^{(s_{k+1})}_1,\ldots,\l^{(s_{k+1})}_{M-(k+1)}$. 
	\item Do $k=k+1$. \emph{At this stage we have rigorous lower and upper bounds for the $M-k$ first eigenvalues of $\SS^{(s_k)}$.} 
	\end{itemize}
 }
 \textbf{Termination:} If we reach $k=M-1$, we have to start the whole procedure again with a larger $M$. Otherwise, once we reach $s_k=1$, we exit with rigorous bounds for the $M-k$ first eigenvalues of $\SS^{1}$. If $M-k$ is smaller than the number $\M$ of eigenvalues that we wanted, we also have to start the whole procedure again with a larger $M$.
 \vspace{0.2cm}
 \caption{The homotopy method.}
 \label{alg:homotopy}
\end{algorithm}

\begin{remark}
Note that, in practice, we do not need to know all the eigenvalues of $\SS^{(0)}$ but only the $M$ smallest, or more precisely rigorous and explicit lower bounds on the $M$ smallest eigenvalues of $\SS^{(0)}$.
\end{remark}

\begin{remark}
During the homotopy, that is for any $s_{k+1}<1$, we do not actually need to obtain rigorous lower bounds for all the eigenvalues $\l^{(s_{k+1})}_1,\ldots,\l^{(s_{k+1})}_{M-(k+1)}$, but only for $\l^{(s_{k+1})}_{M-(k+1)}$, which is the only lower bound that will be required for the next homotopy step. The following modification of Proposition~\ref{prop:LM} can then be useful, allowing us to avoid the a posteriori validation of most eigenvalues of~\eqref{eq:eig_LM}. 
\end{remark}

\begin{proposition}
\label{prop:LM_bis}
Repeat the assumptions of Proposition~\ref{prop:RR}. Assume further that there exists $\nu\in\R$ such that
\begin{equation*}
\ol_M <\nu \leq \l^{(s)}_{M+1},
\end{equation*}
define the matrices
\begin{equation*}
A_2 = \left(\langle \SS^{(s)}x_i,\SS^{(s)}x_j \rangle\right)_{1\leq i,j\leq M},\qquad B_1=A_1-\nu A_0 \quad\text{and}\quad B_2 = A_2 -2\nu A_1 + \nu^2 A_0,
\end{equation*}
and let
\begin{equation*}
\mu<0
\end{equation*}
be an eigenvalue of the generalized eigenvalue problem
\begin{equation*}
B_1 v = \mu B_2 v.
\end{equation*}
Then, 
\begin{equation*}
\ul_M \leq \l_M^{(s)},
\end{equation*}
where
\begin{equation*}
\ul_M := \nu + \frac{1}{\mu}.
\end{equation*}
\end{proposition}
\begin{proof}
The proof of~\cite[Theorem 10.14]{NakPluWat19} can readily be adapted to this case, thanks to~\cite[Theorem 10.10]{NakPluWat19}.
\end{proof}

\subsubsection{Base problem for $S^*S$} 
\label{sec:base_problem}

Let us define $\S:=S^*S:\Z\subset\Y\to\Y$. We want to use the homotopy method described in the previous subsection to compute a rigorous lower bound of $\l_1(\S)$. In order to do so, we first need a so-called \emph{base problem},  i.e. another self-adjoint operator $\S^{(0)}$ which is \emph{simpler} than $\S$ in the sense that we know its spectrum, and for which the eigenvalues are proven to be smaller than the ones of $\S$. This subsection is devoted to the derivation of an appropriate base problem.

It will be convenient to use a block-notation to describe the various operators involved. For instance, introducing $A=L  - \bl \Id$ we write
\begin{equation} \label{OperatorSmatrix}
S  = \begin{pmatrix}[1.5]
A & - \bu \\  \bu^* & 0
\end{pmatrix},
\end{equation}
meaning that, for any $(u,\l)\in\X$
\begin{align*}
S[(u,\l)] &= \begin{pmatrix}[1.5]
A & - \bu \\  \bu^* & 0
\end{pmatrix}
\begin{pmatrix}[1.5]
u \\ \l
\end{pmatrix} \\
&= \begin{pmatrix}[1.5]
Au - \l\bu \\  \langle u,\bu\rangle_{L^2}
\end{pmatrix}.
\end{align*}
Note that we extend the notation ${}^*$ for the adjoint of an operator, and also use it on elements of $X$, $\bu^*$ being the map from $X$ to $\C$ defined by
\begin{equation*}
\bu^* u := \langle u,\bu\rangle_{L^2} \qquad \forall~u\in X.
\end{equation*}
The densely defined self-adjoint operator $\S:\Z\subset \Y\to\Y$ then writes
\begin{equation*}
\S = \begin{pmatrix}
A^* A +  \bu\bu^* & -A^*\bu \\
- (A^*\bu)^* & \bu^*\bu
\end{pmatrix}.
\end{equation*}

Actually, it will not be easy to directly compare $\S$ with the base problem $\S^{(0)}$ that we are going to introduce in the following, because they will not have the same domain. The appropriate view-point here is to compare the associated quadratic forms, which we introduce in the next lemma.

\begin{lemma} \label{lem:BandB0}
Let $\eta_\L \in (0,1)$, $\eta_{\S} \in (0,\left\Vert \bu \right\Vert_{L^2}^2)$, $h_0:=\inf_{\Omega} h$ for $h$ as in~\eqref{h}, $C_V$ the constant introduced in Appendix \ref{appx-elem_est} and
\begin{equation}
\label{eq:s210l}
\s_2:=(1-\eta_{\L})\frac{\sigma^4}{4}, \quad 
\s_1:=\frac{1-\eta_{\L}}{\eta_{\L}}C_V^2-\bl\sigma^2,\quad 
\s_0:=\bl^2+\bl h_0-\frac{1}{\eta_{\S}}\left\Vert A^*\bu \right\Vert_{L^2}^2,\quad 
\s_{\lambda}:=\left(\left\Vert \bu \right\Vert_{L^2}^2 -\eta_{\S}\right).
\end{equation}
Consider, the Hermitian sesquilinear forms $\B:\X\times\X\to\C$ and $\B^{(0)}:\X\times\X\to\C$ defined by
\begin{equation*}
\B\left[(u_1,\l_1),(u_2,\l_2)\right] := \langle S[(u_1,\l_1)],S[(u_2,\l_2)] \rangle_{\Y},
\end{equation*}
and 
\begin{equation*}
\B^{(0)}\left[(u_1,\l_1),(u_2,\l_2)\right] :=  \left(\s_2 \langle \tD u_1, \tD u_2 \rangle_{L^2} - \s_1 \langle \tN u_1, \tN u_2 \rangle_{L^2} + \s_0 \langle u_1,  u_2 \rangle_{L^2} \right) + \s_\l \l_1 \l_2^*.
\end{equation*}
Then, assuming $\bl\in\R$, we have
\begin{equation} \label{eq:S}
\B\left[(u,\l),(u,\l)\right]
\geq  \B^{(0)}\left[(u,\l),(u,\l)\right] \qquad\forall~(u,\l)\in\X.
\end{equation}
\end{lemma}
\begin{proof}
From the definition of $S$, we have, for all $(u,\l)\in\X$,
\begin{align*}
\B[(u,\lambda),(u,\lambda)] =  \left(\langle Au,Au \rangle_{L^2} +  \left\vert\langle u,\bu \rangle_{L^2}\right\vert^2 + \left\Vert \bu \right\Vert_{L^2}^2\vert\lambda\vert^2 -2\Re\left(\langle u, A^*\bu \rangle_{L^2}\lambda^*\right)\right).
\end{align*}
Using that $\left\vert\langle u,\bu \rangle_{L^2}\right\vert^2$ is non negative, and estimating $\Re\left(\langle u, A^*\bu \rangle_{L^2}\lambda^*\right)$ with Young's inequality, we obtain
\begin{align}
\label{eq:B_to_A}
\B[(u,\lambda),(u,\lambda)] \geq  \left(\langle Au,Au \rangle_{L^2} -\frac{1}{\eta_{\S}}\left\Vert A^*\bu \right\Vert_{L^2}^2 \langle u,u \rangle_{L^2} +\left(\left\Vert \bu \right\Vert_{L^2}^2 -\eta_{\S}\right)\vert\lambda\vert^2\right).
\end{align}
Let us now focus on $ \langle Au, Au \rangle_{L^2}$. Recalling that $\bl$ is assumed to be real, we have
\begin{align}
\label{eq:A_to_L}
\langle Au, Au \rangle_{L^2} &= (L  - \bl \Id)u,(L  - \bl \Id)u\rangle_{L^2} \nonumber\\
&= \langle Lu, Lu \rangle_{L^2} +\bl^2 \langle u,u\rangle_{L^2}  -  \bl \Re\left(\langle(L^*  + L )u,u\rangle_{L^2}\right) \nonumber\\
&= \langle Lu, Lu \rangle_{L^2} +\bl^2 \langle u,u\rangle_{L^2}  - \bl \left(\langle \sigma^2\tD u - hu,u\rangle_{L^2}\right) \nonumber\\
&= \langle Lu, Lu \rangle_{L^2} + \bl  \sigma^2 \langle  \tN u,\tN u \rangle_{L^2} +\bl^2 \langle u,u\rangle_{L^2}  + \bl  \left\langle h u, u \right\rangle_{L^2} \nonumber\\
&\geq \langle Lu, Lu \rangle_{L^2} + \bl  \sigma^2 \langle  \tN u,\tN u \rangle_{L^2} +\left(\bl^2 + \bl h_0\right) \langle u,u\rangle_{L^2}.
\end{align}
Finally, recalling that $L=\frac{\sigma^2}{2}\tD + V$ and using Lemma~\ref{lem:basic_est}, we estimate
\begin{align}
\label{eq:L_to_end}
 \langle Lu,Lu \rangle_{L^2} &\geq \frac{\sigma^4}{4} \langle \tD u,\tD u \rangle_{L^2} + \langle Vu,Vu \rangle_{L^2} - \sigma^2\sqrt{\langle \tD u,\tD u \rangle_{L^2}}\sqrt{\langle Vu,Vu \rangle_{L^2}} \nonumber\\
&\geq (1-\eta_{\L})\frac{\sigma^4}{4} \langle \tD u,\tD u \rangle_{L^2} - \left(\frac{1}{\eta_{\L}}-1\right) \langle Vu,Vu \rangle_{L^2} \nonumber\\
&\geq (1-\eta_{\L})\frac{\sigma^4}{4} \langle \tD u,\tD u \rangle_{L^2} - \frac{1-\eta_L}{\eta_{\L}}C_V^2 \langle \tN u,\tN u \rangle_{L^2}.
\end{align}
Combining~\eqref{eq:B_to_A},~\eqref{eq:A_to_L} and~\eqref{eq:L_to_end}, we end up with
\begin{align*}
\B[(u,\lambda),(u,\lambda)] &\geq  \left( (1-\eta_{\L})\frac{\sigma^4}{4} \langle \tD u,\tD u \rangle_{L^2} - \left(\frac{1-\eta_L}{\eta_{\L}}C_V^2 - \bl  \sigma^2 \right) \langle  \tN u,\tN u \rangle_{L^2} \right. \\
 &  \left. \qquad\quad  +\left(\bl^2 + \bl h_0 -\frac{1}{\eta_{\S}}\left\Vert A^*\bu \right\Vert_{L^2}^2\right) \langle u,u \rangle_{L^2} +\left(\left\Vert \bu \right\Vert_{L^2}^2 -\eta_{\S}\right)\vert\lambda\vert^2 \right),
\end{align*}
which finishes the proof.
\end{proof}

We are now almost ready to obtain the base problem $\S^{(0)}$, and the family of operators whose eigenvalues will satisfy~\eqref{eq:monotonicity_homotopy}, but first we need to introduce a few more spaces. For all $s\in[0,1]$, we define
\begin{equation*}
Z^{(s)}:= \left\{ u\in H^4(\Omega) \ \Big\vert\ u\in X,\ \left(sL +(1-s)\frac{2\s_2}{\sigma^2}\tD\right) u\in X \right\} \qquad \text{and}\qquad \Z^{(s)}=Z^{(s)}\times\C.
\end{equation*}
Note that $Z^{(1)}$ is nothing but the space $Z$ introduced in~\eqref{eq:def_Z}. Finally, consider the self-adjoint operator $\S^{(0)}:\Z^{(0)}\subset \Y\to \Y$ defined by
\begin{equation} \label{eq:S0}
\S^{(0)} :=
\begin{pmatrix}
\s_2\tD^2  +\s_1\tD +\s_0 & 0 \\
0 & \s_\lambda
\end{pmatrix},
\end{equation}
with $\s_2$, $\s_1$, $\s_0$ and $\s_\l$ as in~\eqref{eq:s210l} and, for all $s\in[0,1]$, the self-adjoint operator $\S^{(s)}:\Z^{(s)}\subset \Y\to \Y$ defined by
\begin{equation} \label{eq:Ss}
\S^{(s)} := s\S + (1-s)\S^{(0)}.
\end{equation}
For all $s\in[0,1]$, we denote the eigenvalues of $\S^{(s)}$ by
\begin{equation*}
\lambda^{(s)}_1 \leq \lambda^{(s)}_2 \leq \ldots \leq \lambda^{(s)}_n \leq \ldots\,.
\end{equation*}
We can now prove that the family $\S^{(s)}$ gives a suitable homotopy between $\S^{(0)}$ and $\S$.
\begin{proposition}
\label{prop:eig_sesq}
Repeat the assumptions of Lemma~\ref{lem:BandB0}, and assume $\bu\in X$. Then the eigenvalues $\l_n^{(s)}$ increase with $s$, i.e.~\eqref{eq:monotonicity_homotopy} holds.
\end{proposition}
\begin{proof}
For any $s$ in $[0,1]$, we consider the Hermitian sesquilinear forms $\B^{(s)}:\X\times\X\to\C$ defined by
\begin{equation*}
\B^{(s)} = s\B + (1-s)\B^{(0)},
\end{equation*}
and note that $\S^{(s)}$ is the self-adjoint operator associated to $\B^{(s)}$, that is,
\begin{equation}
\label{eq:Bs_to_Ss}
\B^{(s)}\left[(u_1,\l_1),(u_2,\l_2)\right] = \langle \S^{(s)}(u_1,\l_1),(u_2,\l_2) \rangle_\Y \qquad \forall~(u_1,\l_1)\in\Z^{(s)},\ \forall~(u_2,\l_2)\in\X.
\end{equation}
Indeed, in order to obtain~\eqref{eq:Bs_to_Ss} from the definition of $\B^{(s)}$ we only have to integrate by parts, and check that all the boundary terms vanish: most of them do so simply by virtue of $\bu$, $u_1$ and $u_2$ being in $X$, and the remaining ones vanish if and only if the function $
\frac{\sigma^2}{2}sLu_1 +\s_2(1-s)\tD u_1$ also belongs to $X$, hence the definition of $Z^{(s)}$. 

Therefore, the eigenvalues of $\S^{(s)}$ can be expressed in terms of the Rayleigh quotient of $\B^{(s)}$ (see e.g. the proof of~\cite[Theorem 10.33]{NakPluWat19}):
\begin{equation}
\label{eq:Rayleigh_bilin}
\lambda_n^{(s)} = \inf_{\substack{H\subset \X \\ dim H= n}} \max_{(u,\l)\in H\setminus\{0\}} \frac{\B^{(s)}\left[(u,\l),(u,\l)\right]}{\langle (u,\l),(u,\l) \rangle_\Y}.
\end{equation}
However, by Lemma~\ref{lem:BandB0} we have
\begin{equation*}
\B^{(s')}\left[(u,\l),(u,\l)\right]
\leq  \B^{(s)}\left[(u,\l),(u,\l)\right] \qquad\forall~0\leq s'\leq s\leq 1,\ \forall~(u,\l)\in\X,
\end{equation*}
which, combined with~\eqref{eq:Rayleigh_bilin} indeed yields
\begin{equation*}
\l_n^{(s')}
\leq  \l_n^{(s)} \qquad\forall~0\leq s'\leq s\leq 1,\ \forall~n\geq 1. \qedhere
\end{equation*}
\end{proof}

In order to use the homotopy method as explained in Section~\ref{sec:homotopy}, with $\S^{(0)}$ as a base problem for $\S$, we need to know the eigenvalues of $\S^{(0)}$, or at least to get explicit lower bounds for a finite number of eigenvalues of $\S^{(0)}$. We explain how to get such bounds in the next subsection, where we use the homotopy method again, this time to enclose the eigenvalues of $\tD$.

\begin{remark}
\label{rem:choice_S0} 
We could have used a different base problem, say $\hat \S^{(0)}$, written in terms of $\Delta$ rather than $\tD$, for instance by using the estimates of Appendix~\ref{appx-elem_est}. Had we done so,  we would know exactly the eigenvalues of $\hat \S^{(0)}$ --- remember that $\Omega$ is a rectangle, therefore we know explicitly the spectrum of $\Delta$. However, $\hat \S^{(0)}$ would be in some sense further away from $\S$ than the $\S^{(0)}$ we have defined in~\eqref{eq:S0}; this is due to the fact that $\S$ does involve the operator $\tD$, and not $\Delta$. Therefore the homotopy between $\hat \S^{(0)}$ and $\S$ would be longer and more costly in practice. The extra work we have to do in the next subsection to rigorously compute eigenvalues of $\tD$ is therefore more than compensated by the fact that our $\S^{(0)}$ from~\eqref{eq:S0} yields a shorter homotopy.
\end{remark}

\subsubsection{Getting the eigenvalues of $\tD$ and of $\S^{(0)}$}
\label{sec:eig_tD}

As explained in the previous subsection, we need to know the eigenvalues of $\tD$, or at least a finite number of them, in order to know the smallest eigenvalues of $\S^{(0)}$~\eqref{eq:S0} and be able to initialize the homotopy method from $\S^{(0)}$ to $\S$. In order to do so, we can, in fact, also use the homotopy method. 

Indeed, defining $\tD^{(0)}:X\subset Y\to Y$ by
\begin{equation*}
\tD^{(0)} u := \frac{\partial^2 u}{\partial r^2} + \frac{4}{\rmax^2}\frac{\partial^2 u}{\partial \psi^2},
\end{equation*}
we have that
\begin{equation*}
\langle -\tD u,u\rangle_Y \geq \langle -\tD^{(0)} u,u\rangle_Y \qquad \forall~u\in X.
\end{equation*}
Note that $-\tD^{(0)}$ has constant coefficients, therefore its spectrum on the recantagular domain $\Omega$ can be computed by hand: The eigenvalues of $-\tD^{(0)}$ are given by 
\begin{equation*}
\tl^{(0)}_{k,n} = \left(\frac{k\pi}{\rmax-\rmin}\right)^2 + \left(\frac{2n}{\rmax}\right)^2 \qquad \forall~k\in\N_1,\ \forall~n\in\N_0,
\end{equation*}
where $\tl^{(0)}_{k,0}$ is of multiplicity one, and $\tl^{(0)}_{k,n}$, $n\neq 0$, of multiplicity two for all $k$. 
The corresponding eigenvectors are
\begin{equation*}
\tilde u^{(0)}_{k,n} = \sin\left(\frac{rk\pi}{\rmax-\rmin}\right)e^{in\psi} \qquad \forall~k\in\N_1,\ \forall~n\in\ZZ.
\end{equation*}

Therefore, defining $\tD^{(s)} := s\tD + (1-s)\tD^{(0)}$ for $s\in[0,1]$, we have all the ingredients we need to apply the homotopy method between $\tD^{(0)}$ and $\tD$, as described in Section~\ref{sec:homotopy}, and get rigorous enclosures on finitely many eigenvalues of $\tD$. Then, coming back to~\eqref{eq:S0}, we see that the eigenvalues of $\S^{(0)}$ are exactly given by
\begin{equation*}
\{ \s_2 \tl^2 - \s_1 \tl + \s_0 \ | \ \tl \text{ eigenvalue of }\tD \} \cup \{ \s_\lambda \}.
\end{equation*}
Remember that, in order to start the main homotopy --- the one from $\S^{(0)}$ to $\S$, that aims at obtaining a rigorous lower bound on the smallest eigenvalue $\l^{(0)}(\S)$ of $\S$ --- we need a rigorous lower bound on all the eigenvalues of $\S^{(0)}$ that are smaller than some threshold $\Gamma$, which in practice is typically taken slightly larger than a numerically obtained value of $\l^{(0)}(\S)$. Since $\s_2$ is positive, the set of $\tl$ such that $\s_2 \tl^2 - \s_1 \tl + \s_0 \leq \Gamma$ is bounded, and can be computed explicitely in practice once $\s_2$, $\s_1$, $\s_0$ are known. Therefore, a first homotopy from $-\tD^{(0)}$ to $-\tD$ can be used, as explained just above, to get rigorous enclosures of the finitely many eigenvalues $\tl$ of $-\tD$ such that $\s_2 \tl^2 - \s_1 \tl + \s_0 \leq \Gamma$. 

This then yields the finitely many eigenvalues of $\S^{(0)}$ that are below $\Gamma$, after having added $\s_{\lambda}$ if it was also below $\Gamma$, and we can finally use a second homotopy, from $\S^{(0)}$ to $\S$, to get a rigorous lower bound on $\l^{(0)}(\S)$.

\begin{remark}
In order to reduce the computational cost associated to rigorously enclosing the eigenvalues of $-\tD$, one could make use of the fact that the eigenfunctions of $-\tD$ are of the form
\begin{equation*}
v_n(r)e^{in\psi},
\end{equation*}
where $v_n$ solves a one-dimensional eigenvalue problem:
\begin{equation*}
v_n''(r) -\frac{4}{r^2}v_n(r) = \lambda v_n(r).
\end{equation*}
However, in order to simply the presentation, we did not take advantage of this reduction and choose to directly apply the homotopy method to $-\tD$.
\end{remark}

\subsubsection{Lifting the operator norm estimate from $\Y$ to $\X$}
\label{sec:kappa0tokappa}

In Sections~\ref{sec:homotopy} to~\ref{sec:eig_tD}, we have seen how to obtain a rigorous lower bound on the smallest eigenvalue $\l^{(0)}(\S)$ of $\S$, or equivalently how to get an explicit constant $\kappa_0$ such that~\eqref{eq:kappa0} holds. $\kappa_0$ controls the operator norm of $S^{-1}$, seen as an operator from $\Y$ to itself. However, in order to apply Theorem~\ref{thm:zero_of_F} we need a constant $\kappa$ satisfying~\eqref{eq:kappa}, i.e. we need to control the operator norm of $S^{-1}$ from $\Y$ to $\X$. This is done in the next lemma, which is inspired from~\cite[Sections 6.2.3 and 9.4.1.1]{NakPluWat19}.

\begin{lemma} \label{lem:kappa}
Recalling the weight $\xi_2$ in the definition of $\left\Vert \cdot\right\Vert_\X$ in Section~\ref{sec:notations}, let now $\theta,\xi_2>0$ be such that
\begin{equation}
\label{eq:cond_theta}
\xi_2\left(1+\frac{1}{\theta}\right)\left(\frac{2\left\Vert \bu \right\Vert_{L^2}}{\sigma^2}\right)^2 <1,
\end{equation}
and consider $\kappa_1, \kappa_2$ given by
\begin{equation} \label{eq:valueforkappa1}
\kappa_1 := \frac{1}{\sigma^2}\left(C_V\kappa_0 + \sqrt{C_V^2\kappa_0^2 + 2\sigma^2\kappa_0\left(1+\kappa_0(-\bl)_+\right)}\right),
\end{equation}
where $(-\bl)_+=\max(-\bl,0)$, and 
\begin{equation} \label{eq:valueforkappa2}
\kappa_2 := \frac{2}{\sigma^2} \left(1 + C_V\kappa_1 + \vert\bl\vert\kappa_0 \right).
\end{equation}
If $\kappa_0$ satisfies~\eqref{eq:kappa0}, then~\eqref{eq:kappa} holds with
\begin{equation} \label{eq:valueforkappa}
\kappa := \sqrt{\frac{\kappa_0^2+\kappa_1^2+(1+\theta)\xi_2\kappa_2^2}{1-\xi_2\left(1+\frac{1}{\theta}\right)\left(\frac{2\left\Vert \bu \right\Vert_{L^2}}{\sigma^2}\right)^2}},
\end{equation}
\end{lemma}
\begin{proof} 
We start by establishing the following two estimates:
\begin{equation}
\label{eq:kappa1}
\left\Vert \tN u \right\Vert_{L^2} \leq \kappa_1 \left\Vert S[( u, \lambda)] \right\Vert_\Y \qquad \forall~(u,\l)\in\X,
\end{equation}
and
\begin{equation}
\label{eq:kappa2}
\left\Vert \tD u \right\Vert_{L^2} \leq \kappa_2\left\Vert S[(u,\lambda)] \right\Vert_\Y + \frac{2\left\Vert \bu \right\Vert_{L^2}}{\sigma^2}\vert\lambda\vert \qquad \forall~(u,\l)\in\X.
\end{equation}

In order to get~\eqref{eq:kappa1}, we use the Cauchy-Schwarz inequality
\begin{equation*}
\left\Vert S[( u, \lambda)] \right\Vert_\Y \left\Vert ( u, \lambda) \right\Vert_\Y \geq \Re \left(\langle -S[( u, \lambda)] , ( u, \lambda) \rangle_\Y\right),
\end{equation*}
and observe that
\begin{align*}
\langle -S[( u, \lambda)] , ( u, \lambda) \rangle_\Y  &= \langle -Lu +\bl u + \lambda \bu, u \rangle_{L^2} -\langle u,\bu\rangle_{L^2}\lambda^* \\
&= \langle -\left(\frac{\sigma^2}{2}\tD u + Vu\right) +\bl  u + \lambda \bu, u \rangle_{L^2} - \langle u,\bu\rangle_{L^2}\lambda^* \\
&= \left(\frac{\sigma^2}{2}\left\Vert \tN u \right\Vert_{L^2}^2 -\langle Vu, u \rangle_{L^2} + \bl \left\Vert u \right\Vert_{L^2}^2\right) -\left(\langle u,\bu\rangle_{L^2}\lambda^*-\lambda \langle \bu,u \rangle_{L^2}\right).
\end{align*}
We point out that the fact that the last two terms above are the complex conjugate of each other, and therefore vanish when we take the real part, is the main reason behind our choice of normalization condition for $\F$ (see Remark~\ref{rem:normalization}). Taking the real part, and using the estimates of Appendix~\ref{appx-elem_est}, we get
\begin{align*}
\Re \left(\langle -S[( u, \lambda)] , ( u, \lambda) \rangle_\Y\right) &\geq \left(\frac{\sigma^2}{2}\left\Vert \tN u \right\Vert_{L^2}^2 - \left\Vert Vu \right\Vert_{L^2} \left\Vert u \right\Vert_{L^2} + \bl\left\Vert u \right\Vert_{L^2}^2\right)  \\
&\geq \left(\frac{\sigma^2}{2}\left\Vert \tN u \right\Vert_{L^2}^2 - C_V \left\Vert \tN u \right\Vert_{L^2} \left\Vert u \right\Vert_{L^2} + \bl\left\Vert u \right\Vert_{L^2}^2\right),
\end{align*}
and therefore
\begin{align*}
\left(\frac{\sigma^2}{2}\left\Vert \tN u \right\Vert_{L^2}^2 - C_V \left\Vert \tN u \right\Vert_{L^2} \left\Vert u \right\Vert_{L^2} + \bl\left\Vert u \right\Vert_{L^2}^2\right)  
 \leq \left\Vert S[( u, \lambda)] \right\Vert_\Y \left\Vert ( u, \lambda) \right\Vert_\Y.
\end{align*}
Using~\eqref{eq:kappa0}, we end up with 
\begin{equation*}
\frac{\sigma^2}{2}\left\Vert \tN u \right\Vert_{L^2}^2 -  C_V\kappa_0 \left\Vert S[( u, \lambda)] \right\Vert_\Y \left\Vert \tN u \right\Vert_{L^2}  - \kappa_0\left\Vert S[( u, \lambda)] \right\Vert_\Y^2  \leq 0,
\end{equation*}
and, using that the above expression is a quadratic polynomial in $\left\Vert \tN u \right\Vert_{L^2}$, we obtain~\eqref{eq:kappa1}.

In order to get~\eqref{eq:kappa2}, we start by writing
\begin{equation*}
\tD u = \frac{2}{\sigma^2}\left((L-\bl I)u -\lambda\bu -(Vu-\bl u-\lambda\bu)\right),
\end{equation*}
and estimate
\begin{align*}
\left\Vert \tD u \right\Vert_{L^2} &\leq \frac{2}{\sigma^2} \left(\left\Vert S[(u,\lambda)] \right\Vert_\Y + \left\Vert Vu \right\Vert_{L^2} + \vert\bl\vert \left\Vert u \right\Vert_{L^2} + \vert\lambda\vert \left\Vert \bu \right\Vert_{L^2} \right) \\
&\leq \frac{2}{\sigma^2} \left(\left\Vert S[(u,\lambda)] \right\Vert_\Y + C_V\left\Vert \tN u \right\Vert_{L^2} + \vert\bl\vert \left\Vert u \right\Vert_{L^2} + \vert\lambda\vert \left\Vert \bu \right\Vert_{L^2} \right) \\
&\leq \frac{2}{\sigma^2} \left(  C_V\kappa_1 + \vert\bl\vert \kappa_0 \right)\left\Vert S[(u,\lambda)] \right\Vert_\Y + \frac{2\left\Vert \bu \right\Vert_{L^2}}{\sigma^2}\vert\lambda\vert \\
&= \kappa_2\left\Vert S[(u,\lambda)] \right\Vert_\Y + \frac{2\left\Vert \bu \right\Vert_{L^2}}{\sigma^2}\vert\lambda\vert,
\end{align*}
which yields~\eqref{eq:kappa2}.

Summing up the obtained estimates we can write, for any $(u,\l)\in\X$ and any $\theta>0$,
\begin{align*}
\left\Vert u \right\Vert_{L^2}^2 + \vert\lambda\vert^2 &\leq \kappa_0^2 \left\Vert S[(u,\lambda)] \right\Vert_\Y^2, \\
\left\Vert \tN u \right\Vert_{L^2}^2  &\leq \kappa_1^2 \left\Vert S[(u,\lambda)] \right\Vert_\Y^2, \\
\left\Vert \tD u \right\Vert_{L^2}^2  &\leq (1+\theta)\kappa_2^2 \left\Vert S[(u,\lambda)] \right\Vert_\Y^2 +\left(1+\frac{1}{\theta}\right)\left(\frac{2\left\Vert \bu \right\Vert_{L^2}}{\sigma^2}\right)^2 \vert\lambda\vert^2.
\end{align*}
Therefore, we obtain
\begin{align*}
&\left\Vert u \right\Vert_{L^2}^2 +  \left\Vert \tN u \right\Vert_{L^2}^2 +\xi_2 \left\Vert \tD u \right\Vert_{L^2}^2 + \vert\lambda\vert^2 \leq \\
&\qquad\qquad \left(\kappa_0^2+\frac{}{}\kappa_1^2+(1+\theta)\xi_2\kappa_2^2\right)\left\Vert S[(u,\lambda)] \right\Vert_\Y^2 + \xi_2\left(1+\frac{1}{\theta}\right)\left(\frac{2\left\Vert \bu \right\Vert_{L^2}}{\sigma^2}\right)^2 \vert\lambda\vert^2,
\end{align*}
and~\eqref{eq:cond_theta} allows us to conclude the proof.
\end{proof}

\begin{remark}
In practice, we take $\theta=1$ and $\xi_2 = \frac{\sigma^4}{16\left\Vert \bu \right\Vert_{L^2}^2}$.
\end{remark}

\subsection{Modifications for $L^*$}
\label{sec:Lstar}

The procedure that we use to rigorously compute an eigenpair of $L^*$ is the same as the one we just presented for $L$. In this section we just outline the small changes that have to be made to the estimates, without repeating all the computations.

Let $(\bu,\bl)$ be a numerically computed approximate eigenpair of $L^*$. As in Section~\ref{sec:fixed-point}, our starting point is a map whose zeros correspond to eigenpairs of $L^*$, and we thus replace~\eqref{OperatorcalF} by $\mathcal F: \X \to \Y$ defined as
\begin{equation} 
\mathcal{F} [(u, \lambda)] := 
\begin{pmatrix}[1.5]
L^* u - \lambda u \\ 
\langle u,\bu\rangle_{L^2} -1
\end{pmatrix}.
\end{equation}
Theorem~\ref{thm:zero_of_F} still holds for this new $\F$, and we just have to slightly adapt the way we compute $\kappa$, by replacing every instance of $V$ by $V^*=-V-h$ in the computations of Sections~\ref{sec:base_problem} and~\ref{sec:kappa0tokappa}.

Firstly, the self-adjoint operator, whose smallest eigenvalue we have to rigorously enclose using the homotopy method in order to get $\kappa_0$, is now
\begin{equation*}
\S = \begin{pmatrix}
AA^* +  \bu\bu^* & -A\bu \\
- (A\bu)^* & \bu^*\bu
\end{pmatrix},
\end{equation*}
still with $A=L  - \bl \Id$.
Repeating the computations done in Lemma~\ref{lem:BandB0} for this slightly different $\S$, we still obtain a base problem $\S^{(0)}$ of the form~\eqref{eq:S0}, with
\begin{equation*}
\s_2:=(1-\eta^{(1)}_{\L}-\eta^{(0)}_{\L})\frac{\sigma^4}{4}, \quad 
\s_1:=\frac{C_V^2}{\eta_{\L}^{(1)}}-\bl\sigma^2,\quad 
\s_0:=\bl^2+\bl h_0-\frac{\left\Vert A\bu \right\Vert_{L^2}^2}{\eta_{\S}} - \frac{\left\Vert h \right\Vert_\infty^2}{\eta_{\L}^{(0)}},\quad 
\s_{\lambda}:=\left(\left\Vert \bu \right\Vert_{L^2}^2 -\eta_{\S}\right).
\end{equation*}

Secondly, once $\kappa_0$ has been obtained, we get $\kappa$ in the same fashion as in Lemma~\ref{lem:kappa}, the constant $\kappa_1$ and $\kappa_2$ now having to be defined as follows
\begin{equation*}
\kappa_1 := \frac{1}{\sigma^2}\left(C_V\kappa_0 + \sqrt{C_V^2\kappa_0^2 + 2\sigma^2\kappa_0\left(1+\kappa_0(h_0-\bl)_+\right)}\right),
\end{equation*}and 
\begin{equation*}
\kappa_2 := \frac{2}{\sigma^2} \left(1 + C_V\kappa_1 + \left\Vert h+\bl \right\Vert_\infty\kappa_0 \right).
\end{equation*}

\subsection{Validation of the correct eigenpairs} \label{sec:smallesteigenvalue}

We have just finished describing a procedure that allows us to rigorously compute an eigenpair $(\eta,\l_\eta)$ of $L$ and an eigenpair $(\phi,\l_\phi)$ of $L^*$. In order to use these eigenpairs to rigorously compute the conditioned Lyapunov exponent $\Lambda_c$ via formula~\eqref{Conditioned Lyapunov exponent} (or more precisely via formula~\eqref{eq:FKformula_cond}), we will prove a posteriori that we indeed have validated the correct eigenpairs. For that prupose, we will use the following insights.

As indicated in Section~\ref{sec:mainresult},  we know from \cite[Proposition 4]{MeleardVillemonais} that
$$\d \nu(r, \psi) = \phi(r,\psi) \d r \d \psi,$$
where
$$L^* \phi = \lambda_0 \phi, \ \phi = 0 \ \text{on } \partial E,$$
is the limiting quasi-stationary distribution with escape rate $\lambda_0 < 0$ such that
$$ \mathbb{P}_{\nu} (T > t) = e^{\lambda_0 t}. $$
Let further be $E:= \Omega=(\rmin,\rmax)\times(0,2\pi] \subset \mathbb{R}^3$ our domain of interest, seen as a cylinder in $\mathbb{R}^3$ with absorbing boundary $\{r=\rmin\} \cup \{r=\rmax\}$.
We obtain the following result directly from the literature:
\begin{proposition} \label{prop:eta}
\begin{enumerate}[(a)]
\item There exists a non-negative function $\eta$ on $E \cup \partial E$, positive on $E$ and vanishing on $\partial E$, defined by
\begin{equation} \label{eta}
 \eta (x) = \lim_{t \to \infty} \frac{\mathbb{P}_x(T > t)}{\mathbb{P}_{\nu}(T > t)} = \lim_{t \to \infty} e^{-\lambda_0 t} \mathbb{P}_x (T > t)\,,
\end{equation}
where $ \int \eta \, \rmd \nu =1$ and the convergence holds uniformly in $E \cup \partial E$. 

Furthermore, $\eta$ is a bounded eigenfunction the backward Kolmogorov operator~\eqref{eq:backwardKO}  with eigenvalue $\lambda_0$, i.e. 
$$ L \eta = \lambda_0 \eta \,.$$
\item Let $f \in \tilde{\mathcal E}$ be an eigenfunction of $ L$ for an eigenvalue $\lambda$, being constant on $\partial E$. Then either
\begin{enumerate} [(i)]
\item $\lambda = 0$ and $f$ is constant,
\item or $\lambda = \lambda_0$, $f = \left( \int f \, \rmd \nu \right) \eta$ and $f|_{\partial E} \equiv 0$,
\item or $\Re(\lambda) \leq \lambda_0 - \gamma$, $\int f \, \rmd \nu = 0$ and $f|_{\partial E} \equiv 0$, where $\gamma >0$ is the rate of convergence to the quasi-stationary distribution $\nu$, and, in particular, $f$ is not non-negative.
\end{enumerate}
\end{enumerate}
\end{proposition}
\begin{proof}
This follows from \cite[Proposition 2.3 and Corollary 2.4]{CV16} which are applicable in our situation according to \cite[Section 5.3.2]{ccv16}. For statement (iii), we simply carry the proof over to the case with potentially complex eigenvalues and eigenfunctions, adopting the estimate for the real part.
\end{proof}
Hence, we obtain that a non-negative eigenfunction of $L$  vanishing at the boundary is necessarily associated to the largest non-zero eigenvalue $\lambda_0 < 0$.
Therefore, we can formulate the following statement which gives sufficient conditions, checkable in practice (see Section~\ref{sec:correct_eig_num} for the details), for having obtained the correct eigenpairs.

\begin{proposition}
\label{prop:correct_eig}
Let $(\eta,\lambda_\eta)\in\X$ be an eigenpair of $L$, and $(\phi,\lambda_\phi)\in\X$ be an eigenpair of $L^*$. If
\begin{equation}
\label{eq:cond_eta}
\eta(r,\psi) \geq 0 \quad \forall (r,\psi)\in\Omega,
\end{equation}
then $\lambda_{\eta}$ is the eigenvalue with largest real part $\lambda_0$ of $L$ and $\eta$ is exactly the  eigenfunction occurring in formula~\eqref{eq:FKformula_cond}. 
Assume further that 
\begin{equation}
\label{eq:cond_phi}
\langle \eta,\phi\rangle_{L^2} \neq 0,
\end{equation}
then also $\lambda_{\phi} = \lambda_0$ and $\phi$ is the eigenfunction of $L^*$ occurring in formula~\eqref{eq:FKformula_cond}.
\end{proposition}
\begin{proof}
The first part follows directly from Proposition~\ref{prop:eta}. To obtain the second part, simply notice that
\begin{align*}
\lambda_0 \langle \eta,\phi\rangle_{L^2} = \langle L\eta,\phi\rangle_{L^2} = \langle \eta,L^*\phi\rangle_{L^2} = \lambda_\phi^*\langle \eta,\phi\rangle_{L^2},
\end{align*}
thus~\eqref{eq:cond_phi} implies $\lambda_\phi = \lambda_0$.
\end{proof}


\section{Numerics and implementation}
\label{sec:numerics}

In this section we discuss the implementation of our strategy outlined in Section~\ref{sec:validation}, and present the obtained results. While we feel that the choices we have made here, in particular in terms of how to discretize the solution, are all well adapted to our precise problem, we emphasize that there is some freedom at the level of the implementation, and that different procedures could definitely be used.

\subsection{Discretization}

We represent elements of $X$ and $Z$ via Fourier-Chebyshev series
\begin{align}
\label{eq:FourCheb}
u(r,\psi) &= \sum_{n\in\ZZ} \left(u_{n,0} + 2\sum_{k=1}^{\infty} u_{k,n}(r) T_{\vert k\vert}(r)\right) e^{in\psi} \nonumber\\
&= \sum_{k,n\in\ZZ} u_{\vert k\vert,n} T_{\vert k\vert}(r) e^{in\psi},
\end{align}
where $T_k$ denotes the $k$-th Chebyshev polynomial of the first kind, rescaled from $[-1,1]$ to $[\rmin,\rmax]$. In practice, we naturally truncate the expansions. For instance, the approximate eigenvector $\bu$ of $L$ that we want to validate is of the form
\begin{align}
\label{eq:truncated}
\bu(r,\psi) &= \sum_{\vert k\vert <K}\sum_{\vert n\vert <N} u_{\vert k\vert,n} T_{\vert k\vert}(r) e^{in\psi},
\end{align}
for some given $N$ and $K$.

The choice of using a spectral method is motivated by the fact that the elements of $X$ and $Z$ that we need to approximate, namely eigenvectors of $L$ and $\S$, are much smoother than a typical element of $X$ or $Z$, and therefore they admit a representation of the form~\eqref{eq:FourCheb} with fast decaying coefficients $u_{k,n}$. For more background on Chebyshev series, see e.g.~\cite{Tre13}.

Notice also that, for a function $\bu$ of the form~\eqref{eq:truncated}, most of the operations involved in computing $L\bu$ or $\S\bu$, i.e.~taking derivatives in $r$ and $\psi$, being multiplied by functions like $g(r,\psi)$ or computing inner products, can be done easily and exactly in practice --- up to rounding errors, which we control using the Intlab package for interval arithmetic~\cite{Rum99}. The only exception concerns the terms of the form $\frac{1}{r}$ or $\frac{1}{r^2}$, which of course cannot be represented exactly using truncated Chebsyhev series. However, it is straightforward to approximate these terms on $[\rmin,\rmax]$ with high accuracy by truncated Chebyshev series, and to get tight and rigorous error bounds, for instance using a Newton-Kantorovich argument similar to Theorem~\ref{thm:zero_of_F} (see Appendix~\ref{appx-validation_inv_r} for more details).

We emphasize that most of the computations required in Section~\ref{sec:validation} can (and should) be done with usual floating point arithmetic, without worrying about rigorous error bounds. This is for instance the case when we find an approximate eigenpair for $L$ and $L^*$, or when we look for the approximate eigenvectors needed during the homotopy method (see Proposition~\ref{prop:RR} and Remark~\ref{rem:homotopy}). The only computations that have to be made rigorous, i.e.~where truncation error and rounding errors have to be explicitly controlled, are the ones that we use to verify the assumptions in Theorem~\ref{thm:zero_of_F} and Proposition~\ref{prop:correct_eig}. We give more details concerning theses rigorous computations in the following two subsections.

\subsection{Rigorous computation needed for Theorem~\ref{thm:zero_of_F}}
\label{sec:rig_kappa}

Most of the quantities that we need in Theorem~\ref{thm:zero_of_F}, such as $\delta$, are straightforward to compute rigorously using interval arithmetic. The single quantity whose rigorous computation is more involved is $\kappa$, and more specifically $\kappa_0$, which we will focus on in this subsection. In particular, in the process of obtaining rigorously a constant $\kappa_0$ satisfying~\eqref{eq:kappa0}, we use the homotopy method twice and we therefore have to:
\begin{enumerate}
\item\label{domain} Make sure that the approximate eigenvectors used in Propositions~\ref{prop:RR} and~\ref{prop:LM} exactly belong to the domain of the self-adjoint operator under consideration,
\item\label{A012} Make sure that the approximate eigenvectors used rigorously compute the entries of the matrices $A_0$, $A_1$, $A_2$, $B_1$ and $B_2$ in Propositions~\ref{prop:RR} and~\ref{prop:LM},
\item\label{verifyeigs} Rigorously solve the generalized eigenvalue problems~\eqref{eq:eig_RR} and~\eqref{eq:eig_LM} in order to get rigorous eigenvalue bounds.
\end{enumerate}

Let us focus on the case of the homotopy from $\tD^{(0)}$ to $\tD$, discussed in Section~\ref{sec:eig_tD}. The first point that has to be addressed, is that, according to Proposition~\ref{prop:RR}, we need the approximate eigenvectors $\bu^{(i)}$ to belong to $X$. The regularity requirements as well as the boundary conditions in $\psi$ are trivially satisfied by truncated Fourier-Chebyshev series of the form~\eqref{eq:truncated}. Therefore, having $\bu^{(i)}$ belonging to $X$ is equivalent to having $\bu^{(i)}$ satisfying the homogeneous Dirichlet boundary conditions in $r$. Even though we consider non-local representations of the solutions by using Chebyshev series, these conditions at the boundary are easy to enforce, because we can efficiently parametrize the subspace of elements of the form~\eqref{eq:truncated} which satisfy those boundary conditions. Indeed, just by making use of the fact that the (unrescaled) Chebyshev polynomials satisfy
\begin{equation*}
T_k(-1) = (-1)^k \qquad\text{and}\qquad T_k(1)=1 \qquad \forall~k\in\N,
\end{equation*}
the conditions $\bu^{(i)}(r_{min},\psi)=0=\bu^{(i)}(r_{max},\psi)$ rewrites
\begin{equation}
\label{eq:BC_impl}
\left\{
	\begin{aligned}
		&\sum_{\vert k\vert <K} \bu^{(i)}{_{\vert k\vert,n}} = 0 \\
		&\sum_{\vert k\vert <K} (-1)^k \bu^{(i)}{_{\vert k\vert,n}} = 0
	\end{aligned}
\right.
\qquad\qquad \forall~\vert n\vert <N.
\end{equation}
Therefore, for each Fourier mode $n$, we can parametrize the first two Chebyshev modes in terms of the other Chebyshev modes. In other words, the condition~\eqref{eq:BC_impl} is equivalent to having
\begin{equation}
\label{eq:BC_expl}
\left\{
	\begin{aligned}
		&\bu^{(i)}_{0,n} = -2\sum_{l=1}^{\lfloor \frac{K}{2}\rfloor} \bu^{(i)}_{2l,n} \\
		&\bu^{(i)}_{1,n} = -\sum_{l=1}^{\lfloor \frac{K-1}{2}\rfloor} \bu^{(i)}_{2l+1,n}
	\end{aligned}
\right.
\qquad\qquad \forall~\vert n\vert <N.
\end{equation}
In practice, when using the discretization~\eqref{eq:truncated}, we thus only consider the coefficients $\bu_{k,n}$ for $1\leq k <K$ and $\vert n\vert <N$ as unknowns, and define the remaining coefficients $\bu_{k,n}$ for $k\in\{0,1\}$ and $\vert n\vert <N$ via~\eqref{eq:BC_expl}. We then automatically get $\bu\in X$. This strategy can be easily generalized to tackle more boundary conditions, as in~\eqref{eq:def_Z}, which is required for the second homotopy. The only slight difference is that in the above case we could make the conversion between the implicit definition of some coefficients~\eqref{eq:BC_impl}, and the associated explicit definition~\eqref{eq:BC_expl} by hand (it basically amounts to inverting a $2\times 2$ system), whereas when we have more equations it becomes convenient do to the conversion using rigorous numerics instead.

Secondly, in order to rigorously compute the matrices $A_0$, $A_1$ and $A_2$, we need to rigorously evaluate quantities like
\begin{equation}
\label{eq:tD0_scal}
\langle \bu^{(i)},\bu^{(j)} \rangle_{L^2},\quad \langle \tD^{(0)}\bu^{(i)},\bu^{(j)} \rangle_{L^2},\quad \text{and}\quad \langle \tD^{(0)}\bu^{(i)},\tD^{(0)}\bu^{(j)} \rangle_{L^2},
\end{equation}
as well as
\begin{equation}
\label{eq:tD_scal}
\langle \bu^{(i)},\bu^{(j)} \rangle_{L^2},\quad \langle \tD\bu^{(i)},\bu^{(j)} \rangle_{L^2},\quad \text{and}\quad \langle \tD\bu^{(i)},\tD\bu^{(j)} \rangle_{L^2},
\end{equation}
where $\bu^{(i)}$ and $\bu^{(j)}$ are numerically computed approximate eigenvectors, represented by truncated Fourier-Chebyshev series of the form~\eqref{eq:truncated}, and belong to $X$ as explained above. Indeed, all the computations required for~\eqref{eq:tD0_scal} can be made exactly: since $\bu^{(i)}$ and $\bu^{(j)}$ are truncated Fourier-Chebyshev series, so are $\tD^{(0)}\bu^{(i)}$ and $\tD^{(0)}\bu^{(j)}$, as well as their products, and the integrals involved in the inner products can be computed. The output is then exact, up to potential rounding errors, which are explicitly controlled using interval arithmetic. For~\eqref{eq:tD_scal} we have to be slightly more careful, since the factor $\frac{1}{r^2}$ in $\tD$ cannot be represented exactly as a truncated Chebyshev series. However, as mentioned previously (see also Appendix~\ref{appx-validation_inv_r}), we can write 
\begin{equation*}
\frac{1}{r^2} = \varphi^{\inv}(r) + \epsilon^{\inv}(r) \qquad\forall~r\in[\rmin,\rmax],
\end{equation*}
where 
\begin{equation*}
\varphi^{\inv}(r) = \sum_{\vert k\vert <K} \varphi^{\inv}_{\vert k\vert} T_{\vert k\vert}(r),
\end{equation*}
is a truncated Chebyshev series whose coefficients we have computed explicitly, and $\epsilon^{\inv}$ is such that
\begin{equation*}
\sup_{r\in[\rmin,\rmax]} \left\vert \epsilon^{\inv}(r) \right\vert \leq \rho^{\inv},
\end{equation*}
where the error bound $\rho^{\inv}$ is also known explicitly (and very small, see Section~\ref{sec:results} for explicit numbers). We then introduce
\begin{equation*}
\bD = \frac{\partial^2}{\partial r^2} + \varphi^\inv(r) \frac{\partial^2}{\partial \psi^2} \qquad\text{and}\qquad \eD =  \epsilon^{\inv}(r) \frac{\partial^2}{\partial \psi^2},
\end{equation*}
so that
\begin{equation*}
\tD = \bD + \eD.
\end{equation*}
When having to compute $\langle \tD\bu^{(i)},\bu^{(j)} \rangle_{L^2}$ in~\eqref{eq:tD_scal}, we can thus use the splitting
\begin{equation*}
\langle \tD\bu^{(i)},\bu^{(j)} \rangle_{L^2} = \langle \bD\bu^{(i)},\bu^{(j)} \rangle_{L^2} + \langle \eD\bu^{(i)},\bu^{(j)} \rangle_{L^2},
\end{equation*}
where the first term can be compute exactly, and the second one can be estimated explicitly by
\begin{equation*}
\left\vert \langle \eD\bu^{(i)},\bu^{(j)} \rangle_{L^2} \right\vert \leq \rho^{\inv} \left\Vert \frac{\partial^2\bu^{(i)}}{\partial\psi^2} \right\Vert_{L^2} \left\Vert \bu^{(j)} \right\Vert_{L^2}.
\end{equation*}
Similarly, we write
\begin{equation*}
\langle \tD\bu^{(i)},\tD\bu^{(j)} \rangle_{L^2} = \langle \bD\bu^{(i)},\bD\bu^{(j)} \rangle_{L^2} + \langle \eD\bu^{(i)},\bD\bu^{(j)} \rangle_{L^2} + \langle \bD\bu^{(i)},\eD\bu^{(j)} \rangle_{L^2} + \langle \eD\bu^{(i)},\eD\bu^{(j)} \rangle_{L^2},
\end{equation*}
where the first term is computed exactly, and the rest is explicitly estimated as follows
\begin{align*}
&\left\vert \langle \eD\bu^{(i)},\bD\bu^{(j)} \rangle_{L^2} + \langle \bD\bu^{(i)},\eD\bu^{(j)} \rangle_{L^2} + \langle \eD\bu^{(i)},\eD\bu^{(j)} \rangle_{L^2} \right\vert \\
&\qquad\qquad \leq \rho^{\inv} \left(\left\Vert \frac{\partial^2\bu^{(i)}}{\partial\psi^2} \right\Vert_{L^2} \left\Vert \bD\bu^{(j)} \right\Vert_{L^2} + \left\Vert \frac{\partial^2\bu^{(j)}}{\partial\psi^2} \right\Vert_{L^2} \left\Vert \bD\bu^{(i)} \right\Vert_{L^2}\right) + \left(\rho^{\inv}\right)^2 \left\Vert \frac{\partial^2\bu^{(i)}}{\partial\psi^2} \right\Vert_{L^2} \left\Vert  \frac{\partial^2\bu^{(j)}}{\partial\psi^2} \right\Vert_{L^2}.
\end{align*}
We can therefore get rigorous enclosures of every coefficient of the matrices $A_0$, $A_1$ and $A_2$ (and thus also of $B_1$ and $B_2$). 

\begin{remark}
Once a rigorous enclosure of $A_0$, $A_1$ and $A_2$ has been computed, it is straightforward to obtain a rigorous enclosure of $B_2$, using interval arithmetic and the formula $B_2 = A_2 - 2\nu A_1 + \nu^2 A_0$ (see Proposition~\ref{prop:LM}). However, this formula is prone to cancellation errors, which can lead to rather large enclosures. Therefore, in practice we instead compute $B_2$ using the following formula
\begin{equation*}
B_2 = \left(\langle (\SS^{(s)}-\nu) x_i,(\SS^{(s)}-\nu) x_j \rangle\right)_{1\leq i,j\leq M},
\end{equation*}
for which we observed tighter enclosures.
\end{remark}

Finally, for the third point, in order to then rigorously solve the eigenproblems~\eqref{eq:eig_RR} and~\eqref{eq:eig_LM}, we use the built-in Intlab routine \texttt{verifyeig}.

\begin{remark}
From a theoretical point of view, the eigenproblem~\eqref{eq:eig_LM} is obviously equivalent (as soon as there is no zero eigenvalue) to the eigenproblem $B_2 v = \mu^{-1} B_1 v$, but in practice this last formulation seems better suited to rigorous validation via \texttt{verifyeig}, and so we used it instead of~\eqref{eq:eig_LM}. 
\end{remark}

The rigorous computations needed for the second homotopy --- the one from $\S^{(0)}$ to $\S$ --- are similar, and we omit the details.

\subsection{Rigorous computation needed for Proposition~\ref{prop:correct_eig}}
\label{sec:correct_eig_num}

We have seen in Section~\ref{sec:validation} how we could \emph{validate} an eigenpair of $L$ or $L^*$, and we just discussed the related implementation issues in the previous subsection. That is, given a numerically computed eigenpair $(\bar{\eta},\bl_\eta)$ of $L$, we can now prove the existence of an eigenpair $(\eta,\l_\eta)$ of $L$ such that 
\begin{equation*}
(\eta,\l_\eta) = (\bar\eta,\bl_\eta) + (\epsilon_\eta,\epsilon_{\lambda_\eta}),\quad \left\Vert (\epsilon_\eta,\epsilon_{\lambda_\eta}) \right\Vert_\X \leq \rho_\eta,
\end{equation*}
where $\rho_\eta$ is small and explicitly known. Similarly, we can validate a numerically computed eigenpair $(\bar{\phi},\bl_\phi)$ of $L^*$, with an error estimate of the form 
\begin{equation*}
(\phi,\l_\phi) = (\bar\phi,\bl_\phi) + (\epsilon_\phi,\epsilon_{\lambda_\phi}),\quad \left\Vert (\epsilon_\phi,\epsilon_{\lambda_\phi}) \right\Vert_\X \leq \rho_\phi.
\end{equation*}
In the following, we explain how Proposition~\ref{prop:correct_eig} can be applied in practice, ensuring that we have the correct eigenpairs needed for the rigorous computation of the Lyapunov exponent.

Condition~\eqref{eq:cond_phi} is straightforward to check in practice, using interval arithmetic. Indeed, we have
\begin{align*}
\langle \eta,\phi\rangle_{L^2} &= \langle \bar\eta,\bar\phi\rangle_{L^2} + \langle \bar\eta,\epsilon_\phi\rangle_{L^2} + \langle \epsilon_\eta,\bar\phi\rangle_{L^2} + \langle \epsilon_\eta,\epsilon_\phi\rangle_{L^2},
\end{align*}
and
\begin{align*}
\left\vert \langle \bar\eta,\epsilon_\phi\rangle_{L^2} + \langle \epsilon_\eta,\bar\phi\rangle_{L^2} + \langle \epsilon_\eta,\epsilon_\phi\rangle_{L^2} \right\vert &\leq \Vert \bar\eta\Vert_{L^2}\Vert \epsilon_\phi\Vert_{L^2} + \Vert \bar\phi\Vert_{L^2}\Vert \epsilon_\eta\Vert_{L^2} + \Vert \epsilon_\eta\Vert_{L^2}\Vert \epsilon_\phi\Vert_{L^2} \\
&\leq \Vert \bar\eta\Vert_{L^2}\rho_\phi + \Vert \bar\phi\Vert_{L^2}\rho_\eta + \rho_\eta\rho_\phi.
\end{align*}
Therefore, in order to prove that~\eqref{eq:cond_phi} holds, we only have to check that
\begin{equation*}
\vert \langle \bar\eta,\bar\phi\rangle_{L^2} \vert > \Vert \bar\eta\Vert_{L^2}\rho_\phi + \Vert \bar\phi\Vert_{L^2}\rho_\eta + \rho_\eta\rho_\phi,
\end{equation*}
which will be the case in practice as soon as the error bounds $\rho_\eta$ and $\rho_\phi$ are small enough.

Condition~\eqref{eq:cond_eta} is less straightforward to verify. The method we propose here is well adapted to the specific solutions we obtain, but we mention that the question of rigorously computing nonnegative solutions of elliptic PDEs has been investigated more generally in~\cite{Tan20}. We first consider a subdomain
\begin{equation}
\label{eq:Omega_eps}
\Omega_\varepsilon := (\rmin+\varepsilon,\rmax-\varepsilon)\times(0,2\pi)
\end{equation}
of $\Omega$ which stays safely away from the absorbing boundary. For some small but positive $\varepsilon$, we estimate
\begin{equation*}
\inf_{\Omega_\varepsilon} \eta \geq \inf_{\Omega_\varepsilon} \bar{\eta} - \left\Vert \epsilon_\eta \right\Vert_{C^0} 
\end{equation*} 
and check rigorously using interval arithmetic that
\begin{equation}
\label{eq:positivity_subdomain}
\inf_{\Omega_\varepsilon} \bar{\eta} - \left\Vert \epsilon_\eta \right\Vert_{C^0} > 0.
\end{equation}
In order to bound $\left\Vert \epsilon_\eta \right\Vert_{C^0}$ from above, we use the Sobolev embedding $H^2(\Omega)\hookrightarrow C^0(\bar\Omega)$
\begin{equation*}
\Vert u\Vert_{C^0} \leq \Upsilon_{X,C^0} \Vert u\Vert_X \qquad \forall~u\in X,
\end{equation*}
where the constant $\Upsilon_{X,C^0}$ is given explicitly in Appendix~\ref{appx-embedding}.

Note that~\eqref{eq:positivity_subdomain} cannot be true for $\varepsilon=0$, because the eigenfunction vanishes at the boundary. However, as illustrated in Figures~\ref{fig:eigenfunctions_rlarge} and~\ref{fig:eigenfunctions}, we observe that
\begin{equation*}
\frac{\partial \bar\eta}{\partial r}(r,\psi)>0
\end{equation*}
for $r$ close to $\rmin$, and that 
\begin{equation*}
\frac{\partial \bar\eta}{\partial r}(r,\psi)<0
\end{equation*}
for $r$ close to $\rmax$. Therefore, we would like to prove that 
\begin{equation*}
\frac{\partial \eta}{\partial r}(r,\psi) = \frac{\partial \bar\eta}{\partial r}(r,\psi) + \frac{\partial \epsilon_\eta}{\partial r}(r,\psi)>0
\end{equation*}
for all $(r,\psi)\in(\rmin,\rmin+\varepsilon)\times(0,2\pi)$ and get a similar estimate close to $\rmax$. Here is where we make use of Corollary~\ref{cor:1D}, and of the fact that $\eta$ is essentially a one-dimensional function, as it happens to be independent of the angle variable $\psi$. Indeed, $H^2$ is not embedded in $C^1$ in dimension 2, but it is in dimension 1, and we have
\begin{equation*}
\left\Vert \frac{\partial u}{\partial r}\right\Vert_{C^0} \leq \Upsilon_{X_\radial,C^1} \Vert u\Vert_X \qquad \forall~u\in X_\radial,
\end{equation*}
where the constant $\Upsilon_{X_\radial,C^1}$ is given explicitly in Appendix~\ref{appx-embedding}. Therefore we can compute an explicit lower bound for
\begin{equation*}
\inf_{r\in(\rmin,\rmin+\varepsilon]} \frac{\partial \bar\eta}{\partial r}(r,\psi) - \left\Vert \frac{\partial \epsilon_\eta}{\partial r}\right\Vert_{C^0},
\end{equation*}
and check that it is indeed non negative. Similarly, we check that
\begin{equation*}
\inf_{r\in[\rmax-\varepsilon,\rmax)} \frac{\partial \bar\eta}{\partial r}(r,\psi) - \left\Vert \frac{\partial \epsilon_\eta}{\partial r}\right\Vert_{C^0} \geq 0,
\end{equation*}
which allows us to conclude that~\eqref{eq:cond_eta} holds.

\subsection{Examples and validated results}
\label{sec:results}

In the following, we will fix $a=1$, $\alpha=1$ and $\beta=1$. We are going to vary the shear parameter $b$, the noise level $\sigma$, and the interval $[r_{min},r_{max}]$ in numerical simulations to demonstrate the parameter-dependent behavior in a broad range. Most of the calculations are not rigorous, i.e.~do not use the full homotopy method, due to reasons of running time. However, for particular parameter combinations, we run the complete algorithm as described in the previous sections in order to obtain the sign of the conditioned Lyapunov exponent $\Lambda_c$ rigorously.

\subsubsection{$[r_{min},r_{max}]=[0.75, 1.25]$}

When fixing the shear $b$, and numerically computing $\Lambda_c$ with respect to $\sigma$, we observe different behavior depending on the fixed value of $b$. 
In particular, Figure~\ref{fig:smalldomain} suggests that for large enough shear, there is a transition from negative conditioned Lyapunov exponent to positive  conditioned Lyapunov exponent. For large enough noise, there is also a second transition from positive to negative values of $\Lambda_c$. These findings are in accordance with the results and numerics in~\cite{DoanEngelLambRasmussen} and similar to the behavior described in \cite{engellambrasmussen19_1}, where a simplified model of a stochastically driven limit cycle is considered. The main difference of similar calculations for the latter model (see e.g.~\cite[Figure 1 (a)]{engellambrasmussen19_1} or the respective figures in~\cite{ly08}) is that the graph of the largest Lyapunov exponent $(\sigma, \Lambda_1(\sigma))$ shows no second extremum but increases monotonously after the minimum for small $\sigma$ has been passed, due to the far simpler structure of the model. To sum up, the behavior of $\Lambda_c$ is analogous to what we can expect numerically from the first Lyapunov exponent $\Lambda_1$ for the global model: for large enough fixed $b$, the Lyapunov exponent depends smoothly on $\sigma$, firstly decreasing from $0$ to a minimum and then increasing up to a positive number such that a change of sign occurs for some critical $\sigma^*(b)$, indicating a two-parameter bifurcation.
\begin{figure}[h!]
\begin{center}
\subfloat[\label{fig:smalldomain_b25}$b=2.5$]{
\begin{overpic}[width=0.4\linewidth]{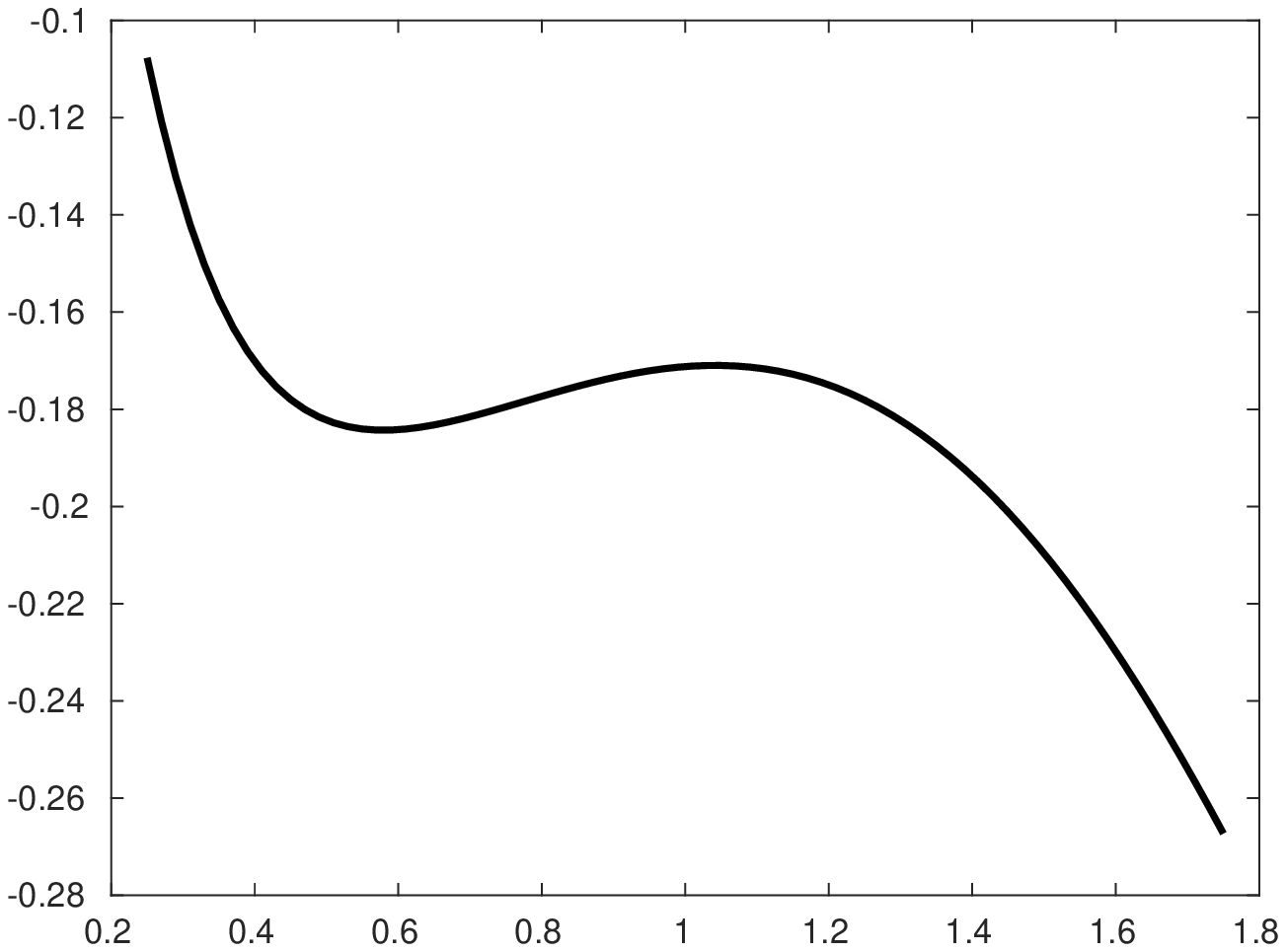}
\put(50,0){\scriptsize $\sigma$}
\put(-5,35){\scriptsize $\Lambda_c$}
\end{overpic} 
}
\subfloat[\label{fig:smalldomain_b3}$b=3$]{
\begin{overpic}[width=0.4\linewidth]{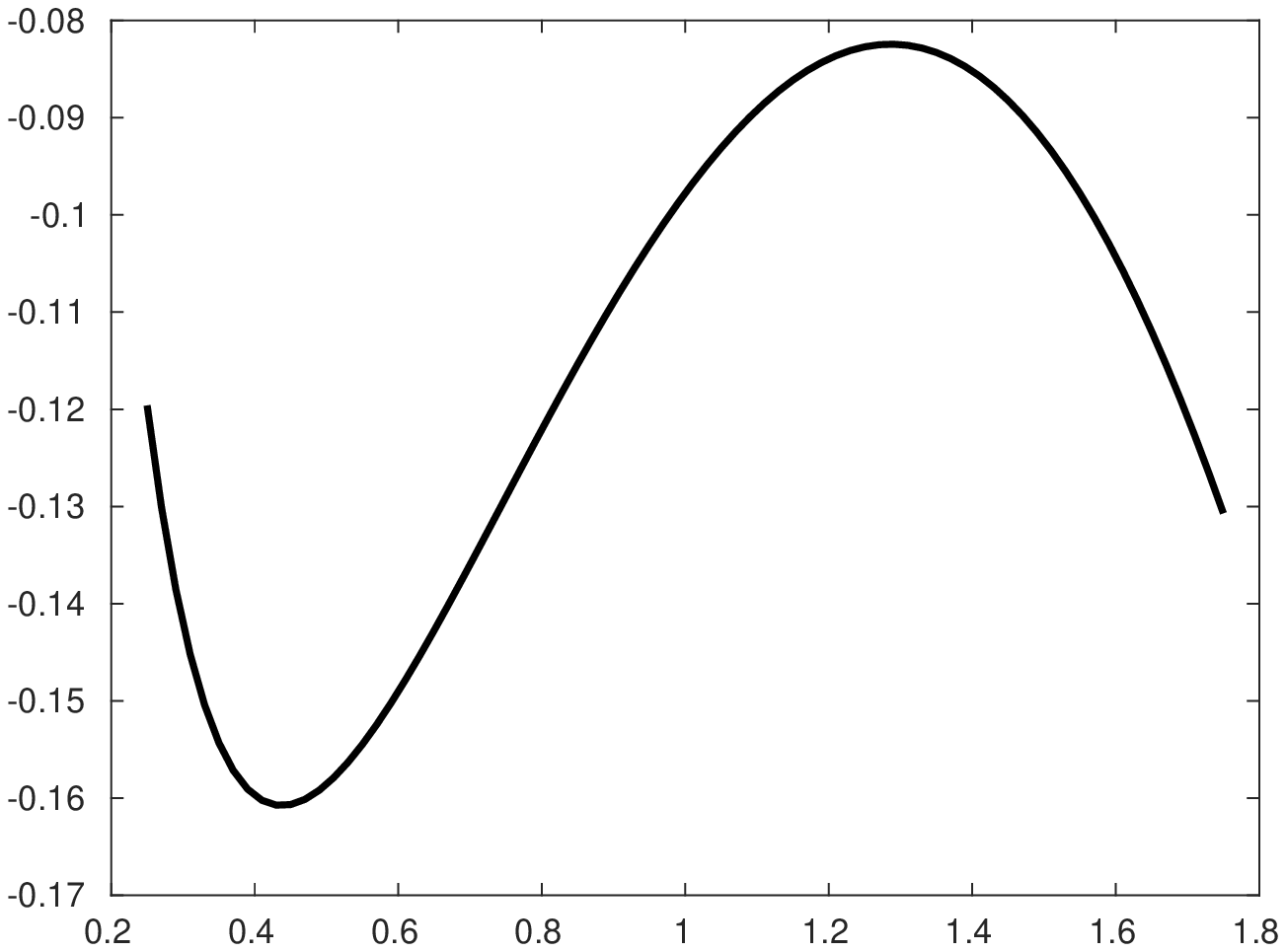}
\put(50,0){\scriptsize $\sigma$}
\put(-5,35){\scriptsize $\Lambda_c$}
\end{overpic} 
}

\subfloat[\label{fig:smalldomain_b35}$b=3.5$]{
\begin{overpic}[width=0.4\linewidth]{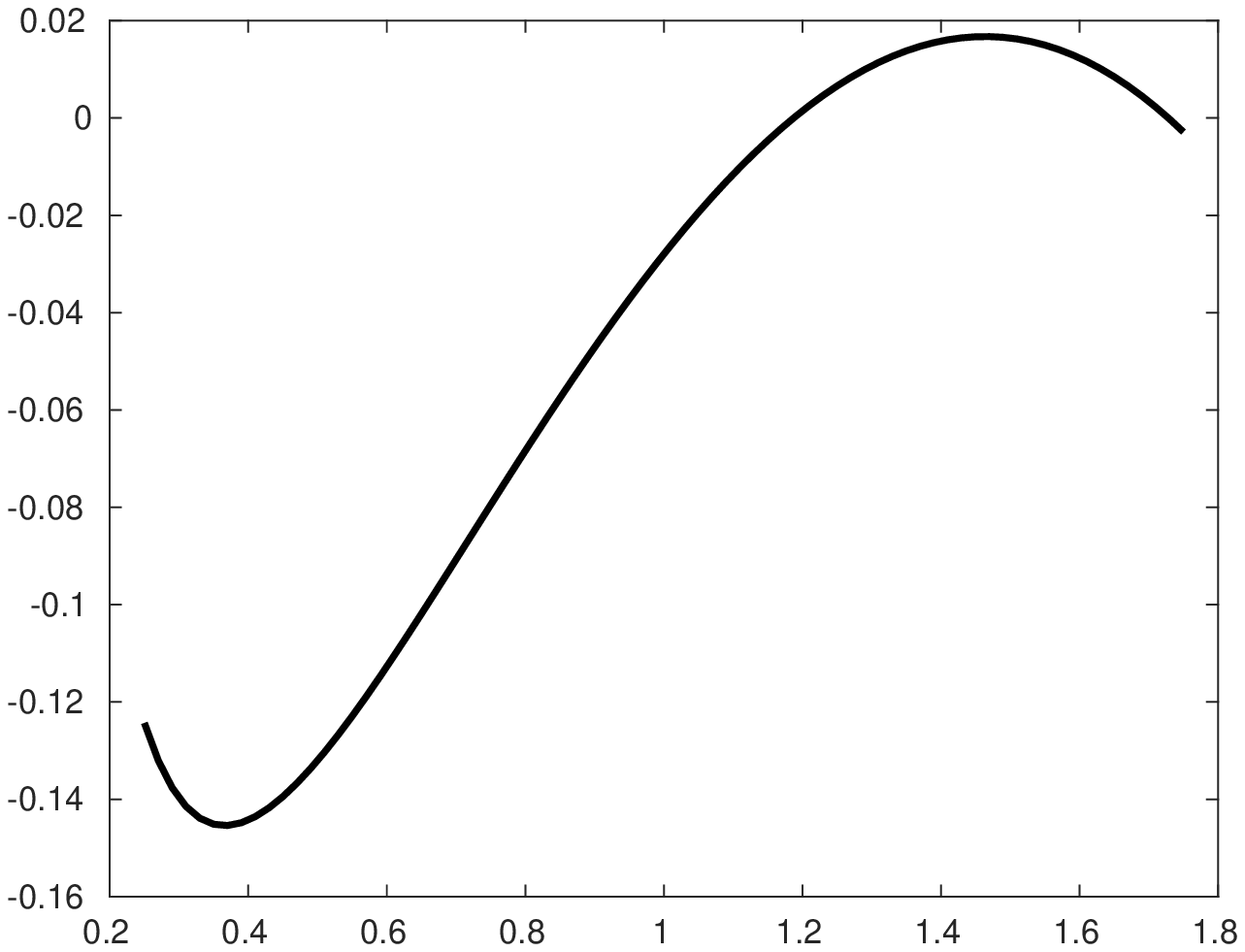}
\put(60,62.5){\tiny $\mathbf \oplus$}
\put(57,60){\tiny $\mathbf \ominus$}
\put(50,0){\scriptsize $\sigma$}
\put(-5,35){\scriptsize $\Lambda_c$}
\end{overpic} 
}
\subfloat[\label{fig:smalldomain_b4}$b=4$]{
\begin{overpic}[width=0.4\linewidth]{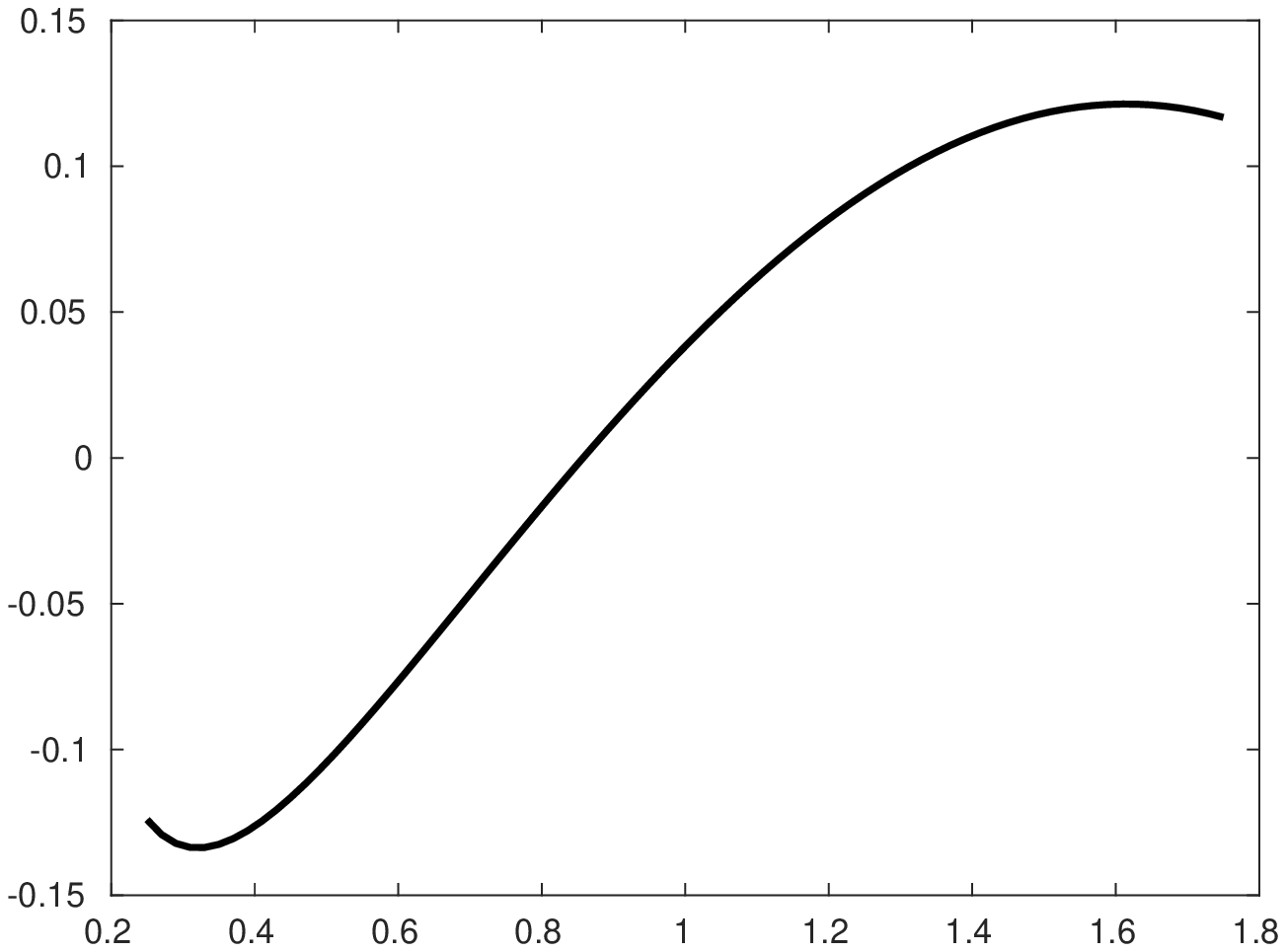}
\put(50,0){\scriptsize $\sigma$}
\put(-5,35){\scriptsize $\Lambda_c$}
\end{overpic} 
}
\end{center}
\caption{Numerical computation of the conditioned Lyapunov exponent $\Lambda_c$ as a function of $\sigma$ for $[\rmin, \rmax] = [0.75, 1.25]$, and $b=2.5$, $3$, $3.5$ and $4$. For $b=3.5$, we proved that a transition occurs by rigorously computing two $\Lambda_c$ close to the crossing through $0$, indicated by $\oplus$ and $\ominus$ respectively in (c), for which we can prove that $\Lambda_c>0$ (resp. $\Lambda_c<0$).}
\label{fig:smalldomain}
\end{figure}

We prove that such a transition occurs, by rigorously computing $\Lambda_c$ for $b=3.5$ and two noise level $\sigma=1.20$ and $\sigma=1.15$ for which we get different signs.
These proofs are done with truncation levels $K=30$ and $N=30$ of the Chebyshev-Fourier modes. Notice that the number of modes is not dictated by the eigenfunctions $\eta$ and $\phi$ themselves, which could be very accurately approximated with fewer modes, but by the rigorous validation process, and in particular by the fact that we need to compute many more eigenfunctions, of $\tD$ and $\S$, with reasonable accuracy during the homotpy method.

\begin{figure}[h!]
\begin{center}
\subfloat[\label{fig:eta_smalldomain}$\eta(r, \psi)$]{
\begin{overpic}[width=0.45\linewidth]{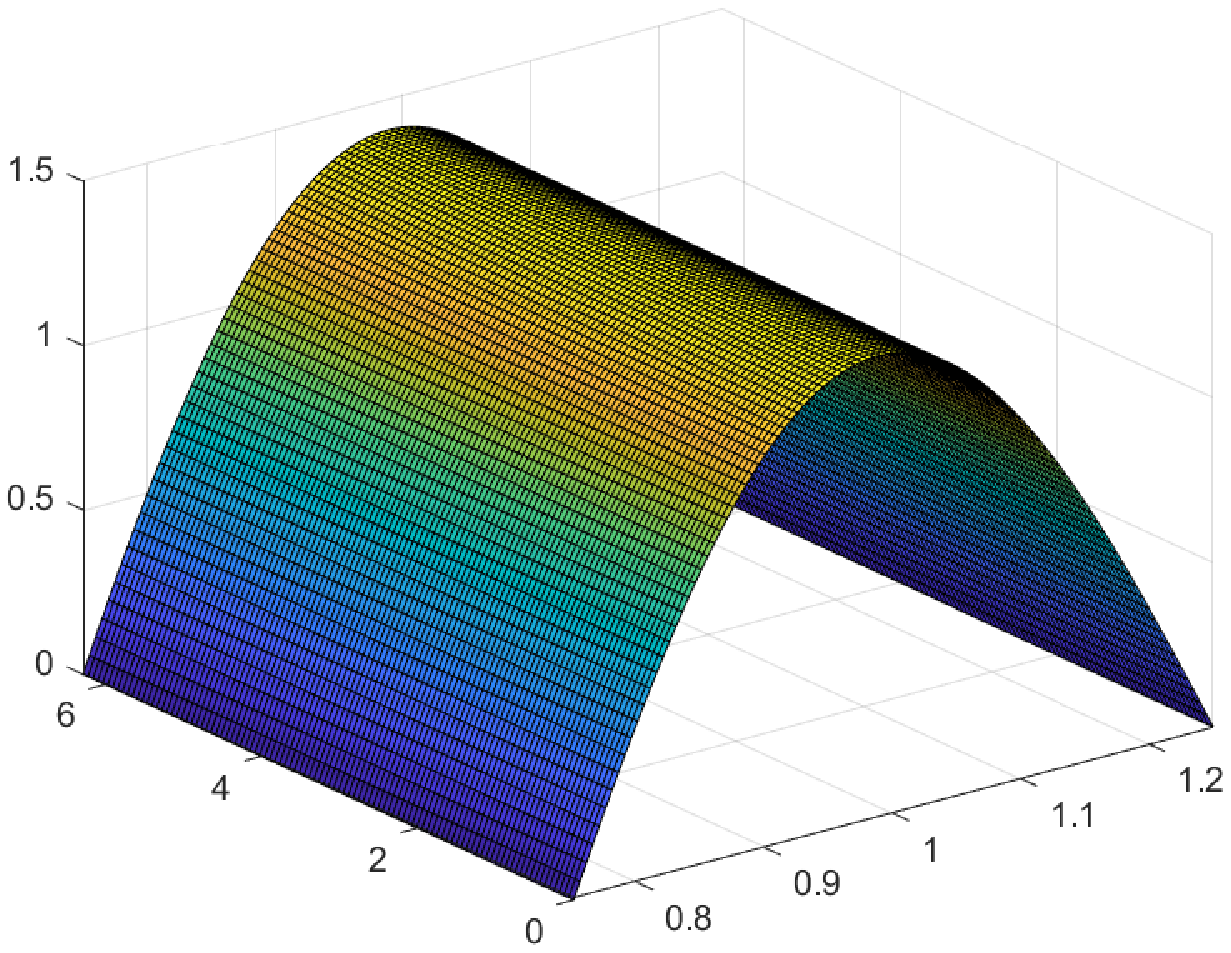}
\put(75,8){\scriptsize $r$}
\put(20,10){\scriptsize $\psi$}
\end{overpic} 
}
\subfloat[\label{fig:phi_smalldomain}$\phi(r, \psi)$]{
\begin{overpic}[width=0.45\linewidth]{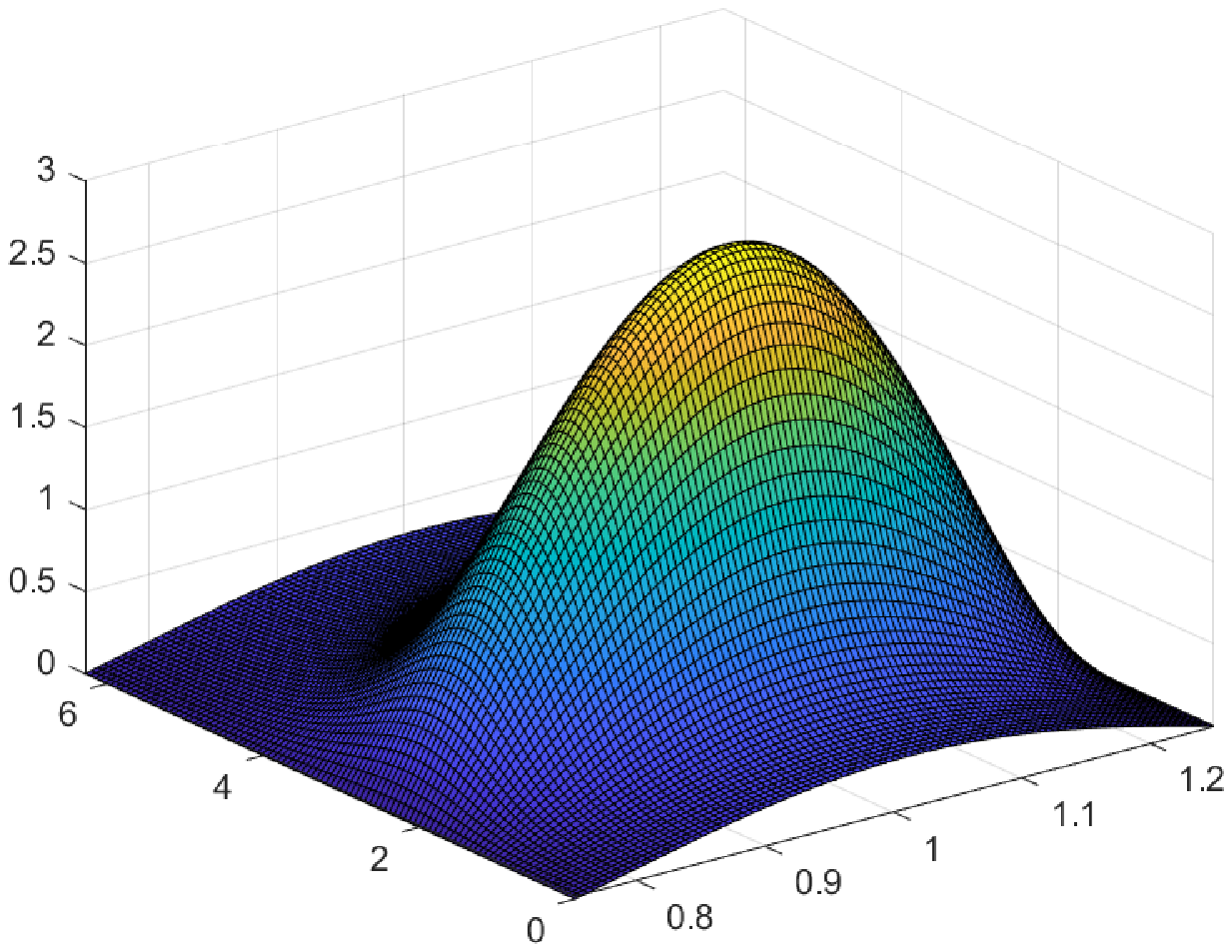}
\put(75,8){\scriptsize $r$}
\put(20,10){\scriptsize $\psi$}
\end{overpic} 
}
\end{center}
\caption{The validated eigenfunctions $\eta$ of $L$ (a) and $\phi$ of $L^*$ (b), for $[\rmin, \rmax] = [0.75, 1.25]$, $b=3.5$ and $\sigma=1.2$.}
\label{fig:eigenfunctions}
\end{figure}

For $\sigma = 1.20$, the eigenfunctions $\eta$ of $L$ and $\phi$ of $L^*$ represented in Figure~\ref{fig:eigenfunctions} are validated using Theorem~\ref{thm:zero_of_F}. We obtain error bounds of $\rho_{\eta} = 1.15\times 10^{-10}$ and $\rho_{\phi} = 2.70\times 10^{-10}$. We then rigorously check assumption~\eqref{eq:cond_eta}, as explained in Section~\ref{sec:correct_eig_num}, and assumption~\eqref{eq:cond_phi}, which ensures we have validated the correct eigenfunctions. Finally, a rigorous evaluation of the modified Furstenberg-Khasminskii formula~\eqref{eq:FKformula_cond} yields
\begin{equation*}
\Lambda_c  \in [0.001453,0.001456],
\end{equation*}
and so we have proven that $\Lambda_c$ is positive in this case. (The escape rate, i.e. the eigenvalue associated to $\eta$ and $\phi$, is approximately equal to $-27.2$).

A similar computer-assisted argument yields that, for $\sigma = 1.15$, 
\begin{equation*}
\Lambda_c  \in [-0.004618,-0.004615],
\end{equation*}
and so we have proven that $\Lambda_c$ is negative in that case. 

All the computer-assisted parts of the proofs can be reproduced using the Matlab code available at~\cite{BreEng21}.

\subsubsection{$[r_{min},r_{max}]=[0.5, 1.5]$}
On the domain $[r_{min},r_{max}]=[0.5, 1.5]$, we observe a similar behavior as in the previous case, i.e.~on the domain given by $[r_{min},r_{max}]=[0.75, 1.25]$ (see also the close similarity of the eigenfunctions in Figure~\ref{fig:eigenfunctions_rlarge} and Figure~\ref{fig:eigenfunctions}).

One difference is that we need slightly more shear to obtain a positive Lyapunov exponent. The other difference is that, since the domain is now larger, the obtained escape rates are lower. An illustration of the numerically obtained behavior of $\Lambda_c$ as a function of $\sigma$ for $b=3.6$ is given in Figure~\ref{fig:numerics_largerdom}.

\begin{figure}[!ht]
\centering
\begin{overpic}[width=0.45\linewidth]{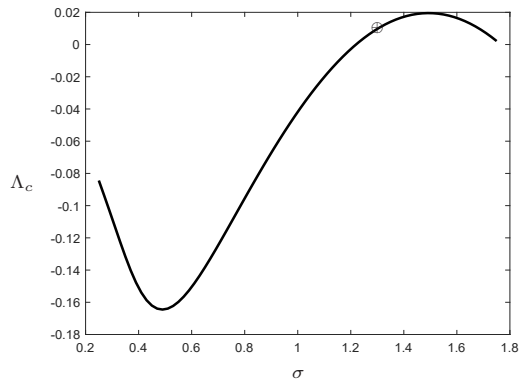}
\put(64.5,64){\tiny $\mathbf \oplus$}
\put(50,0){\scriptsize $\sigma$}
\put(-2,35){\scriptsize $\Lambda_c$}
\end{overpic} 
\caption{Numerical computation of the conditioned Lyapunov exponent $\Lambda_c$ as a function of $\sigma$ for $[\rmin, \rmax] = [0.5, 1.5]$, and $b=3.6$. For $\sigma=1.3$, indicated by $\oplus$, we rigorously computed $\Lambda_c$ and proved that it was positive.}\label{fig:numerics_largerdom}
\end{figure}

With this larger domain, the proof is computationally more demanding (in each homotopy the base problem is in some sense further away from the end problem), and this challenge generally increases with larger domains (and smaller $\sigma$). Nonetheless, we managed to apply the whole procedure described in this paper for $b=3.6$ and $\sigma=1.3$, this time with $K=70$ and $N=50$, and obtain that
\begin{equation*}
\Lambda_c  \in [0.00970, 0.00972], \ 
\end{equation*}
with an escape rate of roughly $-7.1$. The fact that $[0.00970, 0.00972]\subset (0,\infty)$ proves Theorem~\ref{thm:main}. 


\section*{Acknowledgment}

We heartily thank M. Plum for several helpful discussions and references about rigorous eigenvalue bounds. M. Engel has been supported by Germany's Excellence Strategy -- The Berlin Mathematics Research Center MATH+ (EXC-2046/1, project ID: 390685689).


\bibliographystyle{abbrv}
\bibliography{bib_maxis}

\begin{thebibliography}{10}

\bibitem{AriKoc19}
G.~Arioli and H.~Koch.
\newblock Non-radial solutions for some semilinear elliptic equations on the
  disk.
\newblock {\em Nonlinear Analysis}, 179:294--308, 2019.

\bibitem{AriKocTer05}
G.~Arioli, H.~Koch, and S.~Terracini.
\newblock Two novel methods and multi-mode periodic solutions for the
  fermi-pasta-ulam model.
\newblock {\em Communications in mathematical physics}, 255(1):1--19, 2005.

\bibitem{a98}
L.~Arnold.
\newblock {\em Random dynamical systems}.
\newblock Springer, Berlin, 1998.

\bibitem{b91}
P.~Baxendale.
\newblock Statistical equilibrium and two-point motion for a stochastic flow of
  diffeomorphisms.
\newblock In {\em Spatial stochastic processes}, volume~19 of {\em Progress in
  Probability}, pages 189--218. Birkh\"auser, Boston, 1991.

\bibitem{BedrossBlumPanshon2020}
J.~Bedrossian, A.~Blumenthal, and S.~Punshon-Smith.
\newblock A regularity method for lower bounds on the {L}yapunov exponent for
  stochastic differential equations.
\newblock {\em arXiv:2007.15827[math.DS]}, 2020.

\bibitem{BehGoe94}
H.~Behnke and G.~F.
\newblock Inclusions for eigenvalues of selfadjoint problems.
\newblock In {\em Topics in Validated Computations}, volume~5 of {\em Studies
  in Computational Mathematics}, pages 277--322. North-Holland, Amsterdam,
  1994.

\bibitem{bdew14}
N.~Blackbeard, P.~Dutta, H.~Erzgraber, and S.~Wieczorek.
\newblock From synchronisation to optical turbulence in laser arrays.
\newblock {\em Physica D}, 286--287:43--58, 2014.

\bibitem{bew11}
N.~Blackbeard, H.~Erzgraber, and S.~Wieczorek.
\newblock Shear-induced bifurcations and chaos in couple-laser models.
\newblock {\em SIAM Journal on Applied Dynamical Systems}, 10(2):469--509,
  2011.

\bibitem{bxy17}
A.~Blumenthal, Xue.J., and L.~Young.
\newblock Lyapunov exponents for random perturbations of some area-preserving
  maps including the standard map.
\newblock {\em Annals of Mathematics}, 185:1--26, 2017.

\bibitem{BreEng21}
M.~Breden and M.~Engel.
\newblock Matlab code for "{Computer-assisted proof of shear-induced chaos in
  stochastically perturbed Hopf systems}".
\newblock \url{https://sites.google.com/site/maximebreden/research}, 2021.

\bibitem{BreKue19}
M.~Breden and C.~Kuehn.
\newblock Rigorous validation of stochastic transition paths.
\newblock {\em Journal de Math{\'e}matiques Pures et Appliqu{\'e}es},
  131:88--129, 2019.

\bibitem{bro99}
L.~Breyer and G.~Roberts.
\newblock A quasi-ergodic theorem for evanescent processes.
\newblock {\em Stochastic Processes and their Applications}, 84:177--186, 1999.

\bibitem{CanDusMadStaVoh20}
E.~Canc{\`e}s, G.~Dusson, Y.~Maday, B.~Stamm, and M.~Vohral{\'\i}k.
\newblock Guaranteed a posteriori bounds for eigenvalues and eigenvectors:
  multiplicities and clusters.
\newblock {\em Mathematics of Computation}, 89(326):2563--2611, 2020.

\bibitem{ccv16}
N.~Champagnat, K.~A. Coulibaly-Pasquier, and D.~Villemonais.
\newblock Criteria for exponential convergence to quasi-stationary
  distributions and applications to multi-dimensional diffusions.
\newblock In {\em S\'{e}minaire de {P}robabilit\'{e}s {XLIX}}, volume 2215 of
  {\em Lecture Notes in Math.}, pages 165--182. Springer, Cham, 2018.

\bibitem{CV16}
N.~Champagnat and D.~Villemonais.
\newblock Exponential convergence to quasi-stationary distribution and
  {$Q$}-process.
\newblock {\em Probability Theory and Related Fields}, 164(1-2):243--283, 2016.

\bibitem{cmm13}
P.~Collet, S.~Martinez, and J.~Martin.
\newblock {\em Quasi-stationary distributions}.
\newblock Probability and its Applications. Springer, Berlin, 2013.

\bibitem{c91}
H.~Crauel.
\newblock Markov measures for random dynamical systems.
\newblock {\em Stochastics and Stochastics Reports}, 37(3):153--173, 1991.

\bibitem{DayLesMir07}
S.~Day, J.-P. Lessard, and K.~Mischaikow.
\newblock Validated continuation for equilibria of pdes.
\newblock {\em SIAM Journal on Numerical Analysis}, 45(4):1398--1424, 2007.

\bibitem{dsr11}
R.~E.~L. DeVille, N.~Sri~Namachchivaya, and Z.~Rapti.
\newblock Stability of a stochastic two-dimensional non-{H}amiltonian system.
\newblock {\em SIAM J. Appl. Math.}, 71(4):1458--1475, 2011.

\bibitem{dfh08}
H.~Dijkstra, L.~Frankcombe, and H.~Von~der Heydt.
\newblock A stochastic dynamical systems view of the atlantic multidecadal
  oscillation.
\newblock {\em Philos Trans A Math Phys Eng Sci.}, 366:2545--2560, 2008.

\bibitem{ds11}
G.~Dimitroff and M.~Scheutzow.
\newblock Attractors and expansion for{ B}rownian flows.
\newblock {\em Electonical Journal of Probability}, 16(42):1193--1213, 2011.

\bibitem{DoanEngelLambRasmussen}
T.~S. Doan, M.~Engel, J.~S.~W. Lamb, and M.~Rasmussen.
\newblock Hopf bifurcation with additive noise.
\newblock {\em Nonlinearity}, 31(10):4567--4601, 2018.

\bibitem{d94}
P.~Duarte.
\newblock Plenty of elliptic islands for the standard family of area preserving
  maps.
\newblock {\em Annales de l'Institut Henri Poincar\'e. Analyse Non Lin\'eaire},
  11(4):359--409, 1994.

\bibitem{engellambrasmussen19_1}
M.~Engel, J.~S.~W. Lamb, and M.~Rasmussen.
\newblock Bifurcation analysis of a stochastically driven limit cycle.
\newblock {\em Comm. Math. Phys.}, 365(3):935--942, 2019.

\bibitem{engellambrasmussen19_2}
M.~Engel, J.~S.~W. Lamb, and M.~Rasmussen.
\newblock Conditioned {L}yapunov exponents for random dynamical systems.
\newblock {\em Trans. Amer. Math. Soc.}, 372(9):6343--6370, 2019.

\bibitem{fgs16}
F.~Flandoli, B.~Gess, and M.~Scheutzow.
\newblock Synchronization by noise.
\newblock {\em Probability Theory and Related Fields}, 168(3--4):511--556,
  2017.

\bibitem{GalMonNis20}
S.~Galatolo, M.~Monge, and I.~Nisoli.
\newblock Existence of noise induced order, a computer aided proof.
\newblock {\em Nonlinearity}, 33(9):4237, 2020.

\bibitem{GalZgl98}
Z.~Galias and P.~Zgliczy{\'n}ski.
\newblock Computer assisted proof of chaos in the lorenz equations.
\newblock {\em Physica D: Nonlinear Phenomena}, 115(3-4):165--188, 1998.

\bibitem{Goe87}
F.~Goerisch.
\newblock Ein stufenverfahren zur berechnung von eigenwertschranken.
\newblock In {\em Numerical Treatment of Eigenvalue Problems Vol. 4/Numerische
  Behandlung von Eigenwertaufgaben Band 4}, pages 104--114. Springer, 1987.

\bibitem{Gom19}
J.~G{\'o}mez-Serrano.
\newblock Computer-assisted proofs in pde: a survey.
\newblock {\em SeMA Journal}, 76(3):459--484, 2019.

\bibitem{KapMroWilZgl20}
T.~Kapela, M.~Mrozek, D.~Wilczak, and P.~Zgliczynski.
\newblock Capd:: Dynsys: a flexible c++ toolbox for rigorous numerical analysis
  of dynamical systems.
\newblock 2020.

\bibitem{KocSchWit96}
H.~Koch, A.~Schenkel, and P.~Wittwer.
\newblock Computer-assisted proofs in analysis and programming in logic: a case
  study.
\newblock {\em SIAM review}, 38(4):565--604, 1996.

\bibitem{Lan82}
O.~E. Lanford~III.
\newblock A computer-assisted proof of the feigenbaum conjectures.
\newblock {\em Bulletin of the American Mathematical Society}, 6(3):427--434,
  1982.

\bibitem{ly08}
K.~K. Lin and L.-S. Young.
\newblock Shear-induced chaos.
\newblock {\em Nonlinearity}, 21(5):899--922, 2008.

\bibitem{Liu15}
X.~Liu.
\newblock A framework of verified eigenvalue bounds for self-adjoint
  differential operators.
\newblock {\em Applied Mathematics and Computation}, 267:341--355, 2015.

\bibitem{Marti}
J.~T. Marti.
\newblock Evaluation of the least constant in {S}obolev's inequality for
  {$H^{1}(0,\,s)$}.
\newblock {\em SIAM J. Numer. Anal.}, 20(6):1239--1242, 1983.

\bibitem{mv12}
S.~M\'el\'eard and D.~Villemonais.
\newblock Quasi-stationary distributions and population processes.
\newblock {\em Probability Surveys}, 9:340--410, 2012.

\bibitem{MeleardVillemonais}
S.~M\'{e}l\'{e}ard and D.~Villemonais.
\newblock Quasi-stationary distributions and population processes.
\newblock {\em Probab. Surv.}, 9:340--410, 2012.

\bibitem{MisMro95}
K.~Mischaikow and M.~Mrozek.
\newblock Chaos in the lorenz equations: a computer-assisted proof.
\newblock {\em Bulletin of the American Mathematical Society}, 32(1):66--72,
  1995.

\bibitem{NakPluWat19}
M.~T. Nakao, M.~Plum, and Y.~Watanabe.
\newblock {\em Numerical Verification Methods and Computer-Assisted Proofs for
  Partial Differential Equations}, volume~53 of {\em Springer Series in
  Computational Mathematics}.
\newblock Springer Singapore, 2019.

\bibitem{Ois95}
S.~Oishi.
\newblock Numerical verification of existence and inclusion of solutions for
  nonlinear operator equations.
\newblock {\em Journal of Computational and Applied Mathematics},
  60(1-2):171--185, 1995.

\bibitem{Plu90}
M.~Plum.
\newblock Eigenvalue inclusions for second-order ordinary differential
  operators by a numerical homotopy method.
\newblock {\em Zeitschrift f{\"u}r angewandte Mathematik und Physik ZAMP},
  41(2):205--226, 1990.

\bibitem{Plu91}
M.~Plum.
\newblock Computer-assisted existence proofs for two-point boundary value
  problems.
\newblock {\em Computing}, 46(1):19--34, 1991.

\bibitem{Rum99}
S.~M. Rump.
\newblock Intlab—interval laboratory.
\newblock In {\em Developments in reliable computing}, pages 77--104. Springer,
  1999.

\bibitem{Rum10}
S.~M. Rump.
\newblock Verification methods: Rigorous results using floating-point
  arithmetic.
\newblock In {\em Proceedings of the 2010 International Symposium on Symbolic
  and Algebraic Computation}, pages 3--4, 2010.

\bibitem{Schuss}
Z.~Schuss.
\newblock {\em Theory and applications of stochastic processes}, volume 170 of
  {\em Applied Mathematical Sciences}.
\newblock Springer, New York, 2010.
\newblock An analytical approach.

\bibitem{She18}
R.~Sheombarsing.
\newblock {\em Validated Chebyshev-based computations for ordinary and partial
  differential equations}.
\newblock PhD thesis, VU Amsterdam, 2018.

\bibitem{TakLiuOis13}
A.~Takayasu, X.~Liu, and S.~Oishi.
\newblock Verified computations to semilinear elliptic boundary value problems
  on arbitrary polygonal domains.
\newblock {\em Nonlinear Theory and Its Applications, IEICE}, 4(1):34--61,
  2013.

\bibitem{Tan20}
K.~Tanaka.
\newblock Numerical verification method for positive solutions of elliptic
  problems.
\newblock {\em Journal of Computational and Applied Mathematics}, 370:112647,
  2020.

\bibitem{Tre13}
L.~N. Trefethen.
\newblock {\em Approximation Theory and Approximation Practice}, volume 128.
\newblock SIAM, 2013.

\bibitem{Tuc02}
W.~Tucker.
\newblock A rigorous ode solver and smale’s 14th problem.
\newblock {\em Foundations of Computational Mathematics}, 2(1):53--117, 2002.

\bibitem{Tuc11}
W.~Tucker.
\newblock {\em Validated numerics: a short introduction to rigorous
  computations}.
\newblock Princeton University Press, 2011.

\bibitem{Ura65}
M.~Urabe.
\newblock Galerkin's procedure for nonlinear periodic systems.
\newblock {\em Archive for Rational Mechanics and Analysis}, 20:120--152, 1965.

\bibitem{BerLes15}
J.~B. van~den Berg and J.-P. Lessard.
\newblock Rigorous numerics in dynamics.
\newblock {\em Notices Amer. Math. Soc}, 62(9):1057--1061, 2015.

\bibitem{WatKinNak20}
Y.~Watanabe, T.~Kinoshita, and M.~T. Nakao.
\newblock Some improvements of invertibility verifications for second-order
  linear elliptic operators.
\newblock {\em Applied Numerical Mathematics}, 2020.

\bibitem{wieczorek2009stochastic}
S.~Wieczorek.
\newblock Stochastic bifurcation in noise-driven lasers and hopf oscillators.
\newblock {\em Physical Review E}, 79(3):036209, 2009.

\bibitem{yag47}
A.~Yaglom.
\newblock Certain limit theorems of the theory of branching processes.
\newblock {\em Doklady Akademii Nauk SSSR}, 56:795--798, 1947.

\bibitem{Yam98}
N.~Yamamoto.
\newblock A numerical verification method for solutions of boundary value
  problems with local uniqueness by banach's fixed-point theorem.
\newblock {\em SIAM Journal on Numerical Analysis}, 35(5):2004--2013, 1998.

\bibitem{y08}
L.-S. Young.
\newblock Chaotic phenomena in three settings: large, noisy and out of
  equilibrium.
\newblock {\em Nonlinearity}, 21(11):245--252, 2008.

\end{thebibliography}


\appendix

\section*{Appendix}

\section{Definition of a random dynamical system}
\label{appx-rds}

A random dynamical system consists of two ingredients: a ergodic dynamical system that models the noise, and a cocycle that models the dynamics of the system. The definition of a random dynamical system is given as follows \cite[Definition 1.1.2]{a98}.
\begin{definition}[Random dynamical system]
  Let $(\Omega,\mathcal F,\mathbb P)$ be a probability space. A random dynamical system is a pair of mappings $(\theta, \varphi)$.
  \begin{itemize}
    \item[$\bullet$] The ($\mathcal B(\R)\otimes \mathcal F$, $\mathcal F$)-measurable mapping $\theta: \R\times \Omega\to \Omega$, $(t,\omega)\mapsto \theta_t\omega$, is an ergodic dynamical system, i.e.
    \begin{itemize}
		\item [(i)]	$\theta_0=\id$ and $\theta_{t+s}=\theta_t\circ \theta_s$ for $t,s\in\mathbb R$,
		\item [(ii)] $\mathbb P (A) = \mathbb P(\theta_t A)$ for all $A\in \mathcal F$ and $t\in\mathbb R$,
		\item [(iii)] any $A\in\mathcal F$ with $\theta_tA =A$ for all $t\in\mathbb R$ satisfies $\mathbb P(A)\in \{0,1\}$.
	\end{itemize}	
    \item[$\bullet$]
	The ($\mathcal{B}(\mathbb{R})\otimes \mathcal{F} \otimes \mathcal{B}(\mathbb R^d)$, $\mathcal{B}(\mathbb R^d)$)-measurable mapping $\varphi: \mathbb{R} \times \Omega \times \mathbb R^d \to \mathbb R^d, (t, \omega, x) \mapsto \varphi(t, \omega,x)$, is a cocycle over $\theta$, i.e.
	\[
	\varphi(0, \omega, \cdot) \equiv \Id \quad \text{and} \quad \varphi(t+s, \omega,x) = \varphi(t, \theta_s \omega,\varphi(s, \omega,x)) \fa \omega \in \Omega, x \in \R^d \text{ and } t, s \in\R\,.
	\]
  \end{itemize}
  The random dynamical system $(\theta,\varphi)$ is called continuous if $(t, x) \mapsto \varphi(t, \omega)x$ is continuous for every $\omega \in \Omega$. If the mapping $\varphi$ is only defined on $\mathbb{R}_0^+ \times \Omega \times \mathbb R^d$, we speak of a \emph{one-sided} random dynamical system.
\end{definition}
\subsection{RDS induced by an SDE} \label{app:RDSbySDE}
In this paper, we investigate random dynamical systems induced by stochastic differential equations. Hence, we are interested in random dynamical systems adapted to a suitable filtration and of white noise type. Following \cite{fgs16}, we make the following definition:
\begin{definition} \label{whitenoiseRDS}
Let $(\theta, \varphi)$ be a random dynamical system over a probability space $(\Omega,\mathcal F,\mathbb P)$ on a topological space $X$ where $\varphi$ is defined in forward time. 
Let $(\mathcal{F}_s^t)_{-\infty \leq s \leq t \leq \infty} $ be a family of sub-$\sigma$-algebras of $\mathcal F$ such that 
\begin{enumerate}[(i)]
\item $F_t^u \subset F_s^v$ for all $ s \leq t \leq u \leq v$,
\item $F_s^t$ is independent from $F_u^v$ for all $ s \leq t \leq u \leq v$,
\item $ \theta_r^{-1}(\mathcal F_s^t) = \mathcal{F}_{s+r}^{t+r}$ for all $ s \leq t$, $r \in \mathbb{R}$,
\item $\varphi(t, \cdot, x)$ is $\mathcal F_0^t$-measurable for all $t \geq 0$ and $x \in X$.
\end{enumerate}  
Furthermore we denote by $\mathcal{F}_{-\infty}^t$ the smallest $sigma$-algebra containing all $\mathcal F_s^t$, $s \leq t$, and by $\mathcal{F}_{t}^{\infty}$ the smallest $sigma$-algebra containing all $\mathcal F_t^u$, $t \leq u$. Then $(\theta, \varphi)$ is called a \textit{white noise (filtered) random dynamical system}.
\end{definition}
Consider a stochastic differential equation (SDE)
\begin{equation} \label{SDE_RDS}
\rmd X_t = f(X_t) \rmd t + g(X_t) d W_t, \ X_0 \in \mathbb{R}^d\,,
\end{equation} 
where $(W_t)$ denotes some r-dimensional standard Brownian motion, the drift $f: \mathbb{R}^d \to \mathbb{R}^d$ is a locally Lipschitz continuous vector field and the diffusion coefficient $g: \mathbb{R}^d \to \mathbb{R}^{d \times r}$ a Lipschitz continuous matrix-valued map.  If in addition $f$ satisfies a bounded growth condition, as for example a one-sided Lipschitz condition, then by \cite{ds11} there is a white noise random dynamical system $(\theta, \varphi)$ associated to the diffusion process solving~\eqref{SDE_RDS}. The probabilistic setting is as follows: We set $\Omega=C_0(\mathbb R,\mathbb R^r)$, i.e. the space of all continuous functions $\omega:\mathbb R\rightarrow \mathbb R^r$ satisfying that $\omega(0)=0\in \mathbb R^r$. If we endow $\Omega$ with the compact open topology given by the complete metric
\[
\kappa(\omega,\widehat{\omega})
:=
\sum_{n=1}^{\infty}
\frac{1}{2^n}
\frac{\|\omega-\widehat{\omega}\|_n}{1+\|\omega-\widehat{\omega}\|_n},\quad
\|\omega-\widehat \omega\|_n:=\sup_{|t|\leq n } \|\omega(t)-\widehat{\omega}(t)\|\,,
\]
we can set $\mathcal F = \mathcal{B} (\Omega)$, the Borel-sigma algebra on $(\Omega,\kappa)$. There exists a probability measure $\mathbb P$ on $(\Omega,\mathcal F)$ called \emph{Wiener measure} such that the $r$ processes $(W_t^1), \dots, (W_t^r)$ defined by $(W_t^1(\omega), \dots, W_t^r(\omega))^{\mathrm T}:=\omega(t)$ for $\omega\in\Omega$ are independent one-dimensional Brownian motions. Furthermore, we define the sub-$\sigma$-algebra $\mathcal{F}_s^t$ as the $\sigma$-algebra generated by $\omega(u)- \omega(v)$ for $s \leq v \leq u \leq t$. The ergodic metric dynamical system $(\theta_t)_{t\in\mathbb R}$  on $(\Omega,\mathcal F,\mathbb P)$  is given by the shift maps
\begin{equation*}
\theta_t:\Omega\rightarrow \Omega, \quad (\theta_t \omega)(s) = \omega(s+t) - \omega(t)\,.
\end{equation*}
Indeed, these maps form an ergodic flow preserving the probability $\mathbb P$, see e.g. \cite{a98}.

\subsection{Invariant measures} \label{app:invmeas}
Let $(\theta, \varphi)$ be a random dynamical system with the cocycle $\varphi$ being defined on one-or two-sided time $\mathbb T \in \{ \mathbb{R}_0^+, \mathbb{R} \}$. Then the system generates a skew product flow, i.e. a family of maps $(\Theta_t)_{t \in \mathbb{T}}$ from $\Omega \times \R^d $ to itself such that for all $t \in \mathbb T$ and $\omega \in \Omega, x \in \R^d$
\begin{equation*}
\Theta_t(\omega, x) = (\theta_t \omega, \varphi(t, \omega,x))\,.
\end{equation*}
The notion of an invariant measure for the random dynamical system is given via the invariance with respect to the skew product flow, see e.g. \cite[Definition 1.4.1]{a98}. We denote by $ T^*\mu$ the push forward of a measure $\mu$ by a map $T$, i.e. $T^* \mu(\cdot) = \mu (T^{-1}(\cdot))$.
\begin{definition} A probability measure $\mu$ on $\Omega \times \R^d$ is invariant for the random dynamical system $(\theta, \varphi)$ if
\begin{enumerate}
\item[(i)] $\Theta_t^* \mu = \mu$ for all $ t \in \mathbb T$\,,
\item[(ii)] the marginal of $\mu$ on $\Omega$ is $\mathbb{P}$, i.e.~$\mu$ can be factorised uniquely into $\mu(\rmd\omega, \rmd x) = \mu_{\omega}(\rmd x) \mathbb{P}(\rmd \omega)$ where $\omega \mapsto \mu_{\omega}$ is a random measure on $\R^d$.
\end{enumerate}
\end{definition}
The marginal of $\mu$ on the probability space is demanded to be $\mathbb{P}$ as we assume the model of the noise to be fixed.
Note that the invariance of $\mu$ is equivalent to the invariance of the random measure $\omega \mapsto \mu_{\omega}$ on the state space in the sense that
\begin{equation} \label{Markovmeasure}
\varphi(t,\omega, \cdot)^* \mu_{\omega} = \mu_{\theta_t \omega} \quad \mathbb{P}\text{-a.s. for all} \ t \in \mathbb T\,.
\end{equation}
For white noise random dynamical systems $(\theta, \varphi)$, in particular random dynamical systems induced by a stochastic differential equation, there is a one-to-one correspondence between certain invariant random measures and stationary measures of the associated stochastic process, first observed in \cite{c91}. In more detail, we can define a Markov semigroup $(P_t)_{t \geq 0}$ by setting
$$ P_t f(x) = \mathbb{E}(f(\varphi(t, \cdot,x))$$
for all measurable and bounded functions $f: X \to \mathbb{R}$. If $\omega \mapsto \mu_{\omega}$ is a $\mathcal{F}_{-\infty}^{0}$-measurable invariant random measure in the sense of \eqref{Markovmeasure}, also called \textit{Markov measure}, then
\begin{equation*}
\rho (\cdot) = \mathbb{E} [\mu_{\omega}(\cdot)] = \int_{\Omega} \mu_{\omega} (\cdot) \mathbb{P}(d \omega)
\end{equation*}
turns out to be an invariant measure for the Markov semigroup $(P_t)_{t \geq 0}$, often also called stationary measure for the associated process. If $\rho$ is an invariant measure for the Markov semigroup, then
\begin{equation*}
\mu_{\omega} = \lim_{t \to \infty} \varphi(t, \theta_{-t}\omega, \cdot)\rho
\end{equation*}
exists $\mathbb{P}$-a.s.~and is an $\mathcal{F}_{-\infty}^{0}$-measurable invariant random measure.

We observe similarly to \cite{b91} that in the situation of $\mu$ and $\rho$ corresponding in the way described above
$$ \mathbb E [ \mu_{\omega}(\cdot) | \mathcal{F}_{0}^{\infty}] = \mathbb{E} [ \mu_{\omega} (\cdot) ] = \rho(\cdot)\,,$$
and, hence,
$$ \mathbb E [ \mu(\cdot) | \mathcal{F}_{0}^{\infty}] = (\mathbb P \times \rho)(\cdot)\,.$$
Therefore the probability measure $\mathbb{P} \times \rho$ is invariant for $(\Theta_t)_{t \geq 0}$ on $(\Omega \times \R^d, \mathcal{F}_0^\infty \times \mathcal{B}(\R^d))$. In words, the product measure with marginals $\mathbb{P}$ and $\rho$ is invariant for the random dynamical system restricted to one-sided path space. 

\section{Lyapunov spectrum}
\label{appx-LyapSpec}

The random dynamical system $(\theta, \varphi)$ is called $C^k$ if $\varphi(t, \omega, \cdot) \in C^k$ for all $t\in\mathbb T$ and $\omega\in\Omega$, where again $ \mathbb T \in \{ \R, \R_0^+\}$. Let's assume that $(\theta, \varphi)$ is  $C^1$. The \textit{linearisation} or \textit{derivative} $\rmD \varphi(t,\omega,x)$ of $\varphi(t,\omega,\cdot)$ at $x \in \R^d$ is the Jacobian $d\times d$ matrix
$$ \rmD_x \varphi(t,\omega,x) = \frac{\partial \varphi(t, \omega,x)}{\partial x}\,.$$
Differentiating the equation
$$ \varphi(t+s,\omega,x) = \varphi(t, \theta_s \omega, \varphi(s,\omega,x))$$
on both sides and applying the chain rule to the right hand side yields
$$ \rmD_x \varphi(t+s,\omega,x) = \rmD_x \varphi(t, \theta_s \omega, \varphi(s,\omega,x))\rmD_x \varphi(s,\omega,x) =  \rmD_x \varphi(t, \Theta_s(\omega,x))\rmD_x \varphi(s,\omega,x)\,,$$
i.e. the cocycle property of the fibrewise mappings with respect to the skew product maps $(\Theta_t)_{t \in \mathbb T}$. Let us now assume that the random dynamical system possesses an ergodic invariant measure $\mu$. This implies that $(\Theta,\rmD_x \varphi)$ is a random dynamical system with linear cocycle $\rmD_x \varphi$ over the metric dynamical system $(\Omega \times \R^d, \mathcal F \times \mathcal{B}(\R^d), (\Theta_t)_{t \in \mathbb T})$, see e.g. \cite[Proposition 4.2.1]{a98}.
Suppose that $\Phi:=\rmD_x \varphi$ satisfies the integrability condition
\[
  \sup_{0 \leq t \leq 1} \ln^+ \| \Phi(t, \omega,x) \| \in L^1(\mu)\,,
\]
where $\ln^+(x):= \max\{\ln(x),0\}$. Then the Multiplicative Ergodic Theorem \cite[Theorem 3.4.1]{a98} guarantees the existence of a $\Theta$-forward invariant set $\widehat \Omega \subset \Omega \times \R^d$ with $\mu (\widehat \Omega) = 1$, the Lyapunov exponents $\Lambda_1 > \dots > \Lambda_p$, and an invariant measurable filtration
\begin{displaymath}
\mathbb R^d = V_1(\omega,x) \supsetneq V_2(\omega,x)\supsetneq \dots \supsetneq V_p(\omega,x) \supsetneq V_{p+1}(\omega)= \{0\}\,,
\end{displaymath}
such that  for all $0 \neq v \in \mathbb{R}^d$, the \emph{Lyapunov exponent} $\Lambda(v, \omega, x)$, defined by
\begin{equation*}
  \Lambda(v, \omega, x) = \lim_{t \to \infty} \frac{1}{t} \ln \| \Phi(t, \omega, x) v \|
\end{equation*}
exists, and we have
\begin{equation*}
  \Lambda (v, \omega, x) = \Lambda_{i}\quad \Longleftrightarrow \quad v \in V_i(\omega,x) \setminus V_{i+1}(\omega,x) \fa i\in\{1, \dots, p\}\,.
\end{equation*}

\section{Conditioned Lyapunov exponents for RDS} \label{appx-cle}
The following is a short summary of the results from Engel et al.~\cite{engellambrasmussen19_2} which are relevant for this paper. 

We consider the stochastic differential equation \eqref{MainEq_killed} and the time-homogeneous Markov process $(X_t)_{t\geq 0}$ on the Wiener space 
$\Omega=C_0(\mathbb R_0^+,\mathbb R^d)$, i.e.~the space of all continuous functions $\omega:\mathbb R_0^+\rightarrow \mathbb R^d$ satisfying that $\omega(0)=0\in \mathbb R^d$, for an initial condition $X_0\in E$.  Let $(\theta:\mathbb R_0^+\times \Omega\to\Omega, \, \varphi:\mathbb R_0^+\times \Omega\times \bar E \to \bar E )$ be the continuous random dynamical system generated by (see Appendix~\ref{appx-rds}). Similarly to Appendix~\ref{app:RDSbySDE}, the family $(\theta_t)_{t\in\mathbb R_0^+}$ is the $\mathcal B(\mathbb R_0^+)\otimes \mathcal F$-measurable collection of shift maps 
\begin{equation*}
\theta_t:\Omega\rightarrow \Omega, \quad (\theta_t \omega)(s) = \omega(s+t) - \omega(t)\,,
\end{equation*}
preserving the ergodic probability measure $\mathbb P:\mathcal F\to[0,1]$. 
We assume that the cocycle $\varphi: \mathbb{R}_{0}^+ \times \Omega \times \bar E  \to \bar E$ is globally defined in time, in the sense that it takes a constant value in $\partial E$ if the system is killed at the boundary $\partial E$.
We have
$$
  \mathbb P_x( X_t \in B) = \mathbb{P}(\varphi(t, \cdot, x) \in B) \fa  t \geq 0\,,\, x \in E \mbox{ and } B \in \mathcal B (\bar E)\,.
$$
Define the stopping time $\tilde T: \Omega \times E \to \mathbb{R}_0^+$ as
$$ \tilde T(\omega,x) = \inf \big\{ t > 0\,:\, \hat{\varphi}(t, \omega,x) \in \partial E \big\}\,$$
such that for all $x \in E$ and $t\geq 0$
$$ \mathbb P_x( T > t) = \mathbb{P}( \tilde T(\cdot,x) > t)\,$$
where $T$ is the absorption time of the Markov process.

Note that $\varphi(t, \omega,\cdot)$ is differentiable for all $\omega \in\Omega$, $x\in E$ and $t < \tilde T(\omega,x)$. We consider the finite-time Lyapunov exponents
\begin{displaymath}
  \Lambda_{v}(t,\omega,x) = \frac{1}{t}  \ln \frac{\| \rmD_x \varphi(t,\omega,x) v \|}{\|v\|} \fa t \in \big(0, \tilde T(\omega,x)\big)\,,
\end{displaymath}
where $\rmD_x \varphi$ solves the variational equation~\eqref{Linearization}.

Consider the extended process $(X_t, s_t)_{t \geq 0}$, where
$$s_t(\omega, x, v) = \frac{\rmD_x \varphi(t,\omega,x)v}{\| \rmD_x \varphi(t,\omega,X_0) v\|}$$
denotes the induced process on the unit sphere.
Then we can apply Furstenberg--Khasminskii averaging to show that
\begin{equation} \label{lambda_v}
\Lambda_c := \lim_{t \to \infty} \mathbb{E} \big[ \Lambda_v (t, \cdot, x)  \big| \tilde T(\cdot,x) > t \big] \fa x\in E \mbox{ and } v\in \R^d\setminus \{0\}
\end{equation}
exists and is independent of $x$ and $v$.
\begin{theorem}[Conditioned Lyapunov exponent] \label{fk_killed}
  Consider the process $(X_t, s_t)_{t \geq 0}$ under the assumption that it possesses a joint quasi-ergodic distribution $\tilde{m}$ on $E\times \mathbb S^{d-1}$. Then the \emph{conditioned Lyapunov exponent} $\Lambda_c$ as defined in~\eqref{lambda_v} exists and is given by
  \begin{equation}\label{fkformula_quasi}
  \Lambda_c = \lim_{t \to \infty} \mathbb{E} \big[ \Lambda_v (t, \cdot, x)  \big| \tilde T(\cdot,x) > t \big] = \int_{\mathbb S^{d-1} \times E}  \langle s, \rmD f(y) s \rangle \ \tilde{m}(\rmd s, \rmd y)\,,
  \end{equation}
  where the convergence is uniform over all $x \in E$ and $v \in \mathbb{R}^d \setminus \{0\}$.
\end{theorem}
In this paper, we find the quasi-ergodic distribution $m$ as
\begin{equation*}
m(\rmd x) = \eta(x) \phi(x) \rmd x\,,
\end{equation*}
where $\eta$ is an eigenfunction of the backward Kolmogorov operator $L$ and $\phi$ is an eigenfunction of the forward Kolmogorov operator $L^*$.

Additionally, we mention the following theorem which equips the limit of expected values $\Lambda_c$ with the strongest possible dynamical meaning in the setting of killed processes.

\begin{theorem}[Convergence in $L^p$ and conditional probability]  \label{coninprob_md}
Consider a stochastic differential equation of the form~\eqref{MainEq_killed} corresponding to the Markov process $(X_t)_{t \geq 0}$ that is killed at $\partial E$ such that the conditioned Lyapunov exponent $\Lambda_c$ exists.
Then we have for all $1 \leq p \leq 2$ that 
\begin{equation} \label{Lp law_md}
\lim_{t \to \infty} \mathbb{E} \big[ \left|\Lambda_v(t, \cdot, x) - \Lambda_c \right|^p \big| \tilde T(\cdot,x) > t \big] = 0\,,
\end{equation}
and for all $\epsilon > 0$
\begin{equation} \label{weak law_md}
\lim_{t  \to \infty} \mathbb{P} \Big( \big| \Lambda_v(t,\cdot,x) - \Lambda_c \big| \geq \epsilon \Big| \tilde T(\cdot,x) > t \Big) = 0\,,
\end{equation}
in each case uniformly for all $x \in E$ and $v \in \mathbb S^{d-1}$.
This means that the finite-time Lyapunov exponents of the surviving trajectories converge to the Lyapunov exponent $\Lambda_c$ in $L^p$, for $1\leq p \leq 2$, and in probability.
\end{theorem}

\section{Some elementary estimates}
\label{appx-elem_est}

We present here some elementary estimate enabling to compare the operator $\nabla$ and $\Delta$ with $\tN$ and $\tD$ respectively, with explicit constants. The constant $C_V$ plays a role in Section~\ref{sec:kappa0tokappa}, and the other estimates obtained here are used to get explicit embedding constants in Appendix~\ref{appx-embedding}.

First, we define $\rmean$ by
\begin{equation*}
\frac{1}{\rmean^2} = \frac{1}{2}\left(\frac{1}{\rmin^2}+\frac{1}{\rmax^2}\right).
\end{equation*}
We can then show the following estimates for these differential operators, where the norms are in $L^2(\Omega)$ for $\Omega$ as in~\eqref{eq:domain}:
\begin{lemma} \label{lem:basic_est}
The operators $\tN$ \eqref{nablatilde} and $\tD$ \eqref{deltatilde} satisfy
 \begin{equation} \label{eq:nablau_basic}
\min\left(\frac{2}{\rmax}, 1 \right) \Vert \nabla u \Vert_{L^2} \leq \Vert \tN u \Vert_{L^2} \leq \max\left(\frac{2}{\rmin}, 1\right) \Vert \nabla u \Vert_{L^2},
\end{equation}
\begin{equation} \label{eq:Vandnablau}
\Vert Vu \Vert_{L^2} \leq C_V \Vert \tN u \Vert_{L^2},
\end{equation}
where 
\begin{equation} \label{eq:CV}
C_V := \sqrt{\left\Vert f^2 \right\Vert_\infty  +  \left\Vert \frac{r^2}{4} g^2 \right\Vert_\infty},
\end{equation}
and
\begin{equation} \label{eq:Deltau_estimate}
\min\left(\frac{4}{\rmax^2},\frac{2\rmin^2}{\rmin^2+\rmax^2}\right) \left\Vert \Delta u \right\Vert_{L^2}  \leq \left\Vert \tD u \right\Vert_{L^2} \leq \max\left(\frac{4}{\rmin^2},\frac{2\rmax^2}{\rmin^2+\rmax^2}\right) \left\Vert \Delta u \right\Vert_{L^2},
\end{equation}
or equivalently
\begin{equation} \label{eq:Deltau_estimate_equiv}
\begin{cases}
\frac{4}{\rmax^2} \left\Vert \Delta u \right\Vert_{L^2}  &\leq \left\Vert \tD u \right\Vert_{L^2} \leq \frac{2\rmax^2}{\rmin^2+\rmax^2} \left\Vert \Delta u \right\Vert_{L^2}, \qquad \text{if }\ \frac{4}{\rmean^2}\leq 1,\\
\frac{2\rmin^2}{\rmin^2+\rmax^2} \left\Vert \Delta u \right\Vert_{L^2}  &\leq \left\Vert \tD u \right\Vert_{L^2} \leq \frac{4}{\rmin^2} \left\Vert \Delta u \right\Vert_{L^2}, \qquad \qquad \text{if }\ \frac{4}{\rmean^2}\geq 1.
\end{cases}
\end{equation}
\end{lemma}
\begin{proof}
 The inequality~\eqref{eq:nablau_basic} follows immediately from observing that
\begin{equation*}
\min \left(\frac{4}{\rmax^2},1\right) \left\langle -\Delta u,u \right\rangle_{L^2} \leq \langle -\tD u,u \rangle_{L^2} \leq \max\left(\frac{4}{\rmin^2}, 1\right) \left\langle -\Delta u,u \right\rangle_{L^2}.
\end{equation*}

Furthermore, for any $\theta>0$,
\begin{align*}
\langle Vu,Vu \rangle_{L^2} &\leq (1+\theta)\langle f\partial_r u,f\partial_r u \rangle_{L^2} + \left(1+\frac{1}{\theta}\right)\langle g\partial_\psi u,g\partial_\psi u \rangle_{L^2} \\
&\leq (1+\theta)\left\Vert f^2 \right\Vert_\infty \langle \partial_r u,\partial_r u \rangle_{L^2} +  \left(1+\frac{1}{\theta}\right)\left\Vert \frac{r^2}{4} g^2 \right\Vert_\infty \langle \frac{2}{r} \partial_\psi u,\frac{2}{r} \partial_\psi u \rangle_{L^2} \\
&\leq \max\left((1+\theta)\left\Vert f^2 \right\Vert_\infty, \left(1+\frac{1}{\theta}\right)\left\Vert \frac{r^2}{4} g^2 \right\Vert_\infty\right) \langle \tN u, \tN u \rangle_{L^2}.
\end{align*}
Optimizing by taking 
$$\theta = \frac{\left\Vert \frac{r^2}{4} g^2 \right\Vert_\infty}{\left\Vert f^2 \right\Vert_\infty}$$  yields equation~\eqref{eq:Vandnablau} with 
\begin{equation*}
C_V := \sqrt{\left\Vert f^2 \right\Vert_\infty  +  \left\Vert \frac{r^2}{4} g^2 \right\Vert_\infty}.
\end{equation*}

Finally, recall that $\rmean$ satisfies
\begin{equation*}
\frac{1}{2}\left(\frac{4}{\rmin^2}+\frac{4}{\rmax^2}\right)=\frac{4}{\rmean^2}
\end{equation*}
and introduce
\begin{equation*}
\tD^{(0)} u := \frac{\partial^2 u}{\partial r^2} + \frac{4}{\rmean^2}\frac{\partial^2 u}{\partial \psi^2}.
\end{equation*}
Then, we write
\begin{equation*}
\tD u = \tD^{(0)} u + \left(\frac{4}{r^2}-\frac{4}{\rmean^2}\right)\frac{\partial^2 u}{\partial \psi^2},
\end{equation*}
and estimate
\begin{align*}
\left\Vert \left(\frac{4}{r^2}-\frac{4}{\rmean^2}\right)\frac{\partial^2 u}{\partial \psi^2} \right\Vert_{L^2} &\leq \left\Vert \frac{4}{r^2}-\frac{4}{\rmean^2} \right\Vert_{L^\infty} \left\Vert \frac{\partial^2 u}{\partial \psi^2} \right\Vert_{L^2} \\
&= \frac{1}{2}\left(\frac{4}{\rmin^2}-\frac{4}{\rmax^2}\right) \left\Vert \frac{\partial^2 u}{\partial \psi^2} \right\Vert_{L^2} \\
&= 2\frac{\rmax^2-\rmin^2}{\rmin^2\rmax^2} \left\Vert \frac{\partial^2 u}{\partial \psi^2} \right\Vert_{L^2} \\
&\leq 2\frac{\rmax^2-\rmin^2}{\rmin^2\rmax^2}\frac{\rmean^2}{4} \left\Vert \tD^{(0)} u \right\Vert_{L^2} \\
&= \frac{\rmax^2-\rmin^2}{\rmax^2+\rmin^2} \left\Vert \tD^{(0)} u \right\Vert_{L^2}.
\end{align*}
Hence, we obtain
\begin{equation*}
\left(1-\frac{\rmax^2-\rmin^2}{\rmax^2+\rmin^2}\right) \left\Vert \tD^{(0)} u \right\Vert_{L^2}  \leq \left\Vert \tD u \right\Vert_{L^2} \leq \left(1+\frac{\rmax^2-\rmin^2}{\rmax^2+\rmin^2}\right) \left\Vert \tD^{(0)} u \right\Vert_{L^2}.
\end{equation*}
Then, using
\begin{equation*}
 \min\left(\frac{4}{\rmean^2},1\right) \left\Vert \Delta u \right\Vert_{L^2}  \leq \left\Vert \tD^{(0)} u \right\Vert_{L^2} \leq \max\left(\frac{4}{\rmean^2},1\right) \left\Vert \Delta u \right\Vert_{L^2},
\end{equation*}
we end up with the estimates~\eqref{eq:Deltau_estimate} and~\eqref{eq:Deltau_estimate_equiv}.
\end{proof}

\section{Embedding constants}
\label{appx-embedding}

We obtain here the explicit embedding constants needed in Section~\ref{sec:correct_eig_num}. Firstly, we derive an explicit constant $\Upsilon_{X,C^0}$ for the Sobolev embedding $H^2(\Omega)\hookrightarrow C^0(\bar\Omega)$
\begin{equation*}
\left\Vert u\right\Vert_{L^\infty(\Omega)} \leq \Upsilon_{X,C^0} \left\Vert u\right\Vert_X \qquad \forall~u\in X.
\end{equation*}
For the following, we define
\begin{align}
\gamma_1 &= 1.1548, \quad \gamma_2 = 0.22361\,,\\
l_1 &= 2 \pi, \quad l_2 = \rmax - \rmin,\\
C_0 &= (l_1 l_2)^{1/2}, \quad C_1 = \frac{\gamma_1}{\sqrt{3}} \sqrt{\frac{l_1^2 + l_2^2}{l_1 l_2}}, \quad C_2 = \frac{\gamma_2}{3} \sqrt{\frac{\left(l_1^2 + l_2^2 \right)^2 + \frac{4}{3}\left(l_1^4 + l_2^4\right)}{l_1 l_2}},\\
m_1 &= \max \left(\frac{\rmax}{2},1 \right), \quad m_2 = \max \left(\frac{\rmax^2}{4}, \frac{\rmin^2 + \rmax^2}{2 \rmin^2} \right).
\end{align}
Then \cite[Example 6.12 b)]{NakPluWat19}\footnote{Note that in \cite[Section 6.2.6]{NakPluWat19}, the estimate involves the $L^2$-norm of the Hessian $u_{xx}$ which due to the absorbing and periodic boundary conditions on $\Omega$ coincides with the $L^2$-norm of the Laplacian, as can be seen from integration by parts.}, in combination with Lemma~\ref{lem:basic_est}, gives 
\begin{align*}
\Vert u\Vert_{L^\infty(\Omega)} &\leq \sqrt{2\pi(\rmax-\rmin)}\left(C_0 \Vert u\Vert_{L^2(\Omega)} + C_1 \Vert \nabla u\Vert_{L^2(\Omega)}  + C_2 \Vert \Delta u\Vert_{L^2(\Omega)}\right) \\
&\leq \sqrt{2\pi(\rmax-\rmin)}\left(C_0 \Vert u\Vert_{L^2(\Omega)} + m_1 C_1 \Vert \tN u\Vert_{L^2(\Omega)} + m_2 C_2 \Vert \tD u\Vert_{L^2(\Omega)}\right)\\
&\leq \sqrt{2\pi(\rmax-\rmin)}\sqrt{3}\sqrt{  (C_0 \Vert u\Vert_{L^2(\Omega)})^2 + (m_1 C_1 \Vert \tN u\Vert_{L^2(\Omega)})^2 + (m_2 C_2 \Vert \tD u\Vert_{L^2(\Omega)})^2}.
\end{align*}
The factor $\sqrt{2\pi(\rmax-\rmin)}$ comes from the normalization we choose for the $L^2$ norm, see~\eqref{eq:scal_L2}. Hence, recalling the weigh $\xi_2$ in the norm on $X$, we obtain
\begin{equation} \label{eq:upsilon_C0}
\Upsilon_{X,C^0} = \sqrt{6\pi(\rmax-\rmin)} \max \left( C_0, C_1 m_1, \frac{C_2 m_2}{\sqrt{\xi_2}} \right).
\end{equation}
Additionally, we determine a constant $\Upsilon_{X_\radial,C^1}$ such that
\begin{equation*}
\left\Vert \frac{\partial u}{\partial r}\right\Vert_{L^\infty(\Omega)} \leq \Upsilon_{X_\radial,C^1} \left\Vert u\right\Vert_X \qquad \forall~u\in X_\radial,
\end{equation*}
related to the Sobolev embedding $H^2((\rmin, \rmax))\hookrightarrow C^1((\rmin, \rmax))$.
We obtain directly from \cite{Marti} that, for any $u\in H^2((\rmin, \rmax))$,
\begin{equation*}
\left\Vert \frac{\partial u}{\partial r}\right\Vert_{L^\infty((\rmin,\rmax))} \leq  \sqrt{\frac{\rmax-\rmin}{\tanh(\rmax-\rmin)}} \sqrt{\left\Vert \frac{\partial u}{\partial r} \right\Vert_{L^2((\rmin,\rmax))}^2 +  \left\Vert \frac{\partial^2 u}{\partial^2 r} \right\Vert_{L^2((\rmin,\rmax))}^2},
\end{equation*}
and, hence, for all $u\in X_\radial$,
\begin{equation*}
\left\Vert \frac{\partial u}{\partial r}\right\Vert_{L^\infty(\Omega)} \leq  \sqrt{\frac{\rmax-\rmin}{\tanh(\rmax-\rmin)}} \sqrt{ \left\Vert \tN u\right\Vert_{L^2(\Omega)}^2 + \left\Vert \tD u\right\Vert_{L^2(\Omega)}^2},
\end{equation*}
which yields
\begin{equation} \label{eq:upsilon_C1}
\Upsilon_{X_\radial,C^1} = \sqrt{\frac{\rmax-\rmin}{\tanh(\rmax-\rmin)}} \max \left( 1, \frac{1}{\sqrt{\xi_2}} \right).
\end{equation}

\section{Approximation of $\frac{1}{r^2}$ using Chebyshev series, and rigorous error bounds}
\label{appx-validation_inv_r}

We present here some basic estimates needed to get a good approximation --- together with tight error bounds --- of the function $r\mapsto \frac{1}{r^2}$ on $[\rmin,\rmax]$ using Chebyshev series. For more background on the usage of Chebyshev series for computer-assisted proofs, we refer to~\cite{She18} and the references therein.

Let $\nu>1$ and consider the following weighted $\ell^1$ space
\begin{equation*}
\ell^1_\nu := \left\{ \Bphi\in \R^{\N},\ \left\Vert \Bphi \right\Vert_{\ell^1_\nu} :=  \vert \Bphi_0\vert +2\sum_{k=1}^{\infty} \vert \Bphi_{k} \vert \nu^{k} <\infty \right\}.
\end{equation*}
For any $\Bphi\in\ell^1_\nu$, the corresponding Chebyshev series
\begin{equation*}
\varphi(t) = \Bphi_0 + 2 \sum_{k=1}^{\infty} \Bphi_k T_k(t),
\end{equation*}
is well defined and smooth on $[-1,1]$. It is in fact analytic on the Bernstein ellipse of size $\nu$, and reciprocally, any function which is analytic on a Bernstein ellipse of size $\nu'$ for some $\nu'>\nu$ has its Chebyshev coefficients in $\ell^1_\nu$, see~\cite{Tre13}. 

For any $\Bphi,\Bpsi\in\ell^1_\nu$, we define their convolution product $\Bphi\ast\Bpsi$ by
\begin{equation*}
\left(\Bphi\ast\Bpsi\right)_k = \sum_{l\in\ZZ} \Bphi_{\vert l\vert} \Bpsi_{\vert k-l\vert} \qquad \forall~k\in\N.
\end{equation*}
We point out that $\Bphi\ast\Bpsi$ is nothing but the sequence of Chebyshev coefficients of the product $\varphi\psi$, and that $\left(\ell^1_\nu,\ast\right)$ is a Banach algebra:
\begin{equation*}
\left\Vert \Bphi\ast\Bpsi \right\Vert_{\ell^1_\nu} \leq \left\Vert \Bphi \right\Vert_{\ell^1_\nu}  \left\Vert \Bpsi \right\Vert_{\ell^1_\nu} \qquad \forall~\Bphi,\Bpsi\in\ell^1_\nu.
\end{equation*}

Define $\BR\in\ell^1_\nu$ by
\begin{equation*}
\BR_k := 
\left\{
\begin{aligned}
&\left(\frac{\rmin+\rmax}{2}\right)^2+\frac{(\rmax-\rmin)^2}{8}  \qquad & k=0, \\
&\frac{\rmax^2-\rmin^2}{4} \qquad & k=1,\\
&  \frac{(\rmax-\rmin)^2}{16}  \qquad & k=2,\\
&0  \qquad & k\geq 3.
\end{aligned}
\right.
\end{equation*}
The associated Chebyshev series $R$ is nothing but $r\mapsto r^2$, rescaled from $[\rmin,\rmax]$ to $[-1,1]$: it is straightforward to check that
\begin{equation*}
R(t) = \left(\frac{\rmax-\rmin}{2}t+\frac{\rmax+\rmin}{2}\right)^2.
\end{equation*}

Our goal is to find a Chebyshev series $\bBphi$, such that 
\begin{equation*}
R(t)\bphi(t) \approx 1 \qquad \forall~t\in[-1,1],
\end{equation*}
with an explicit error bound. To do so, we introduce
\begin{equation*}
\fF :
\left\{
\begin{aligned}
\ell^1_\nu &\to \ell^1_\nu \\
\Bphi &\mapsto \BR\ast\Bphi - \Bone
\end{aligned}
\right.
\end{equation*}
and use the following lemma (which is similar to Theorem~\ref{thm:zero_of_F}, but simpler, because the problem is linear).
\begin{lemma}
Let $\nu \in (1,\nu_{\text{max}})$, where
\begin{equation}
\label{eq:numax}
\nu_{\text{max}} := \frac{\rmax+\rmin}{\rmax-\rmin} \left(1+\sqrt{1-\left(\frac{\rmax-\rmin}{\rmax+\rmin}\right)^2}\right).
\end{equation}
Let $\bBphi\in\ell^1_\nu$ such that there exists $\delta<1$ satisfying
\begin{align*}
\Vert \fF(\bBphi) \Vert_{\ell^1_\nu} &\leq \delta.
\end{align*}
Then, $\fF$ has a unique zero $\Bphi \in \ell^1_\nu$, and it satisfies
\begin{equation*}
\left\Vert \Bphi -\bBphi  \right\Vert_{\ell^1_\nu} \leq \frac{\delta}{1-\delta} \left\Vert \bBphi \right\Vert_{\ell^1_\nu}.
\end{equation*}
\end{lemma}
\begin{proof}
The map $\fF$ is affine, therefore to get the existence of a unique zero we only have to prove that $\Bphi \mapsto \BR\ast\Bphi$ is invertible on $\ell^1_\nu$. For any given $\Bpsi\in\ell^1_\nu$, the equation
\begin{equation}
\label{eq:Bphi}
\BR\ast\Bphi = \Bpsi,
\end{equation}
is equivalent to having
\begin{equation*}
R\varphi = \psi
\end{equation*}
at the level of functions. $\nu_{\text{max}}$ is defined so that the function $t\mapsto 1/R(t)$ is analytic on a Berstein ellipse of size $\nu'>\nu$, therefore the function $\psi/R$ is analytic on the same ellipse, and~\eqref{eq:Bphi} does have a unique solution in $\ell^1_\nu$. 

Hence we have the existence of a unique zero $\Bphi$ of $\fF$, and it remains to get the a priori error estimate between $\Bphi$ and $\bBphi$. We have
\begin{align*}
\BR\ast\bBphi - \Bone = \BR\ast\left(\bBphi-\Bphi\right),
\end{align*}
and therefore
\begin{align*}
\left\Vert \Bphi -\bBphi  \right\Vert_{\ell^1_\nu} &\leq \left\Vert \BR^{-1}  \right\Vert_{B(\ell^1_\nu,\ell^1_\nu)}  \left\Vert \BR\ast\bBphi - \Bone \right\Vert_{\ell^1_\nu} \\
&\leq \left\Vert \BR^{-1}  \right\Vert_{B(\ell^1_\nu,\ell^1_\nu)}  \delta,
\end{align*}
where $\left\Vert \BR^{-1}  \right\Vert_{B(\ell^1_\nu,\ell^1_\nu)}$ must be understood as the operator norm of the inverse of $\Bphi\mapsto \BR\ast\Bphi$. Besides, 
\begin{equation*}
\left\Vert \BR\ast\bBphi - \Bone \right\Vert_{\ell^1_\nu} \leq \delta <1,
\end{equation*}
yields that the operator $\Bphi\mapsto \BR\ast\bBphi\ast\Bphi$ is invertible, and that
\begin{equation*}
\left\Vert \left(\BR\ast\bBphi\right)^{-1}  \right\Vert_{B(\ell^1_\nu,\ell^1_\nu)} \leq \frac{1}{1-\delta}.
\end{equation*}
Therefore
\begin{align*}
\left\Vert \BR^{-1}  \right\Vert_{B(\ell^1_\nu,\ell^1_\nu)} &= \left\Vert \bBphi \left(\BR\ast\bBphi\right)^{-1}  \right\Vert_{B(\ell^1_\nu,\ell^1_\nu)} \\
&\leq  \left\Vert \bBphi\right\Vert_{B(\ell^1_\nu,\ell^1_\nu)} \left\Vert \left(\BR\ast\bBphi\right)^{-1}  \right\Vert_{B(\ell^1_\nu,\ell^1_\nu)} \\
&\leq  \left\Vert \bBphi\right\Vert_{B(\ell^1_\nu,\ell^1_\nu)} \frac{1}{1-\delta},
\end{align*}
and because the operator norm $\left\Vert \bBphi\right\Vert_{B(\ell^1_\nu,\ell^1_\nu)}$ is equal to the norm $\left\Vert \bBphi\right\Vert_{\ell^1_\nu}$ of the element (this holds for any $\Bphi\in\ell^1_\nu$), we have
\begin{align*}
\left\Vert \BR^{-1}  \right\Vert_{B(\ell^1_\nu,\ell^1_\nu)} \leq  \left\Vert \bBphi\right\Vert_{\ell^1_\nu} \frac{1}{1-\delta},
\end{align*}
and
\begin{align*}
\left\Vert \Bphi -\bBphi  \right\Vert_{\ell^1_\nu} \leq \left\Vert \bBphi\right\Vert_{\ell^1_\nu} \frac{1}{1-\delta} \delta.
\end{align*}
This finishes the proof.
\end{proof}

\end{document}